\documentclass[10pt,reqno]{amsart} 
\RequirePackage[OT1]{fontenc}
\RequirePackage{amsthm,amsmath}
\RequirePackage[colorlinks,linkcolor=blue,citecolor=blue,urlcolor=blue]{hyperref}
\usepackage{amssymb}
\usepackage{anysize}
\usepackage{hyperref}
\usepackage{extarrows}
\usepackage{multirow}
\usepackage{stmaryrd}
\usepackage{color}
\usepackage{amsmath}
\usepackage{enumitem} 
\usepackage{mathtools} 
\usepackage{dsfont} 
\usepackage{comment}
\usepackage{soul}
\usepackage{calrsfs}

\numberwithin{equation}{section}
\theoremstyle{plain} 
\newtheorem{theorem}{Theorem}[section]
\newtheorem{lemma}[theorem]{Lemma}
\newtheorem{corollary}[theorem]{Corollary}
\newtheorem{proposition}[theorem]{Proposition}
\newtheorem{remark}[theorem]{Remark}
\newtheorem{definition}[theorem]{Definition}
\newtheorem{assumption}[theorem]{Assumption}

\theoremstyle{definition}
\newtheorem{notation}[theorem]{Notation}

\renewcommand{\Re}{\mathrm{Re}\,}
\renewcommand{\Im}{\mathrm{Im}\,}

\newcommand{\E}{{\mathbf E }}
\newcommand{\Cov}{{\mathbf{Cov} }}
\newcommand{\V}{{\mathbf{Var} }}
\newcommand{\R}{{\mathbb R }}
\newcommand{\N}{{\mathbb N}}
\newcommand{\Z}{{\mathbb Z}}
\renewcommand{\P}{{\mathbf P}}
\newcommand{\C}{{\mathbb C}}
\newcommand{\OO}{{ O}}
\newcommand{\X}{{\mathcal X}}
\newcommand{\m}{{\mathfrak m}}
\newcommand{\A}{{\mathcal A}}
\newcommand{\Q}{{\mathcal Q}}
\newcommand{\B}{{\mathcal B}}
\newcommand{\s}{{\mathcal S}}

\newcommand{\ii}{\mathrm{i}}

\newcommand{\dd}{\mathrm{d}}
\newcommand{\ie}{\emph{i.e., }}
\newcommand{\eg}{\emph{e.g., }}
\newcommand{\cf}{\emph{c.f., }}
\newcommand{\wt}{\widetilde}
\newcommand{\ud}{\underline}
\newcommand{\wh}{\widehat}
\newcommand{\gz}{G^{z}}

\newcommand{\bs}{\boldsymbol}

\def\ga{G^{z_1}}
\def\gb{G^{z_2}}
\def\ma{\m^{z_1}}
\def\mb{\m^{z_2}}

\def\Tr{\mathrm{Tr}}

\def\F{\mathcal{F}}
\def\K{\mathcal{K}}
\def\one{\mathds{1}}
\def\Dim{\Delta \widetilde{\mathrm{Im}}\,}

\def\<{\langle}
\def\>{\rangle}

\renewcommand{\mathbf}[1]{\bs{#1}}
 
\marginsize{25mm}{25mm}{25mm}{26mm}

\allowdisplaybreaks

\begin{document}
	
	\begin{minipage}{0.85\textwidth}
		\vspace{2.5cm}
	\end{minipage}
	\begin{center}
		\large\bf  
		Precise asymptotics for the spectral radius of a large random matrix	
	
	\end{center}

	\renewcommand*{\thefootnote}{\fnsymbol{footnote}}	
	\vspace{0.5cm}
	
	\begin{center}
		\begin{minipage}{1.0\textwidth}
			\begin{minipage}{0.33\textwidth}
				\begin{center}
					Giorgio Cipolloni\footnotemark[1]\\
					\footnotesize 
					{Princeton University}\\
					{\it gc4233@princeton.edu}
				\end{center}
			\end{minipage}
			\begin{minipage}{0.33\textwidth}
				\begin{center}
					L\'aszl\'o Erd\H{o}s\footnotemark[2]\\
					\footnotesize 
					{IST Austria}\\
					{\it lerdos@ist.ac.at}
				\end{center}
			\end{minipage}
			\begin{minipage}{0.33\textwidth}
				\begin{center}
					Yuanyuan Xu\footnotemark[3]\\
					\footnotesize 
					{AMSS,CAS}\\
					{\it yyxu2023@amss.ac.cn}
				\end{center}
			\end{minipage}
		\end{minipage}
	\end{center}
	
	\bigskip

	\footnotetext[2]{\footnotesize{Partially supported by ERC Advanced Grant ``RMTBeyond" No.~101020331.}}
	\footnotetext[3]{Supported by ERC Advanced Grant ``RMTBeyond" No.~101020331.}

\renewcommand*{\thefootnote}{\arabic{footnote}}
	
	\vspace{5mm}

	\begin{center}
		\begin{minipage}{0.91\textwidth}\small{
				{\bf Abstract.}}
We consider  the spectral radius of 
a large  random matrix $X$ with independent, identically distributed 
entries. 
We show that 
its typical size is given by a precise three-term asymptotics  with an optimal error term
beyond the radius of the celebrated circular law. The coefficients in this asymptotics
are universal but they differ from a similar asymptotics recently proved
for the rightmost eigenvalue of $X$ in~\cite{maxRe}.  To access the more complicated
spectral radius, we need to establish a new decorrelation mechanism
for the low-lying singular values  of $X-z$ for 
different complex shift parameters $z$ using the Dyson Brownian Motion.

		\end{minipage}
	\end{center}

	\vspace{5mm}
	
	{\footnotesize
		{\noindent\textit{Date}:  \today \\
	}
	
	\vspace{2mm}

	\thispagestyle{headings}

\bigskip

\normalsize

\section{Introduction}

Large $n\times n$ random matrices $X=(x_{ij})_{i,j=1}^n$ 
with independent, identically distributed (i.i.d.) entries are the natural non-Hermitian
counterparts of the celebrated Wigner matrices in the Hermitian world. They have been used to study
large systems of damped linear ordinary differential equations, 
$$
\frac{\dd}{\dd t}{\mathbf u}(t) = X{\mathbf u}(t) -{\mathbf u}(t), \qquad {\mathbf u}(t) \in\C^n,
$$
where no specific information is available for the coefficients $x_{ij}$ apart from a general statistical ansatz
that they are i.i.d.. Such situation arises in mathematical biology~\cite{Hastings82} to model
the temporal growth of an ecological system with $n$ interacting species as well as
in theoretical neuroscience~\cite{SCS88} to describe the evolution of $n$ fully connected neurons.
 In his pioneering paper~\cite{May72} in 1972 R. May  raised the question of the long time stability of this ODE system
 which amounts to understand the typical behavior of the rightmost eigenvalue of $X$, which we denote by $\max\Re \mbox{Spec}(X)$.
 He correctly identified a sharp transition in stability, namely that the system is stable if 
 the variance $\E |x_{ij}|^2$ is strictly smaller than $1/n$ and unstable if it is strictly bigger\footnote{
 Interestingly, May correctly {\it located} the transition but his very intuitive
 prediction on its {\it width}  was erroneous since he  connected it with the similar transition for the 
 singular values of $X$ that, as a Hermitian eigenvalue problem,
  behave quite differently from the genuine non-Hermitian eigenvalues.}.
  
  In the early literature on the subject there has   been some ambiguity whether
  stability of the system should be  defined via the spectral radius $\rho(X)$ of $X$ or via the rightmost  eigenvalue 
   $\max\Re \mbox{Spec}(X)$.
  While the growth of ${\mathbf u}(t) $ is determined by the latter, it is bounded by the former
  and  the exact behavior of the
  spectral radius became  the main 
  focus of interest since  the first mathematically rigorous work on the subject  in 1984 by
   Cohen and Newman~\cite[p.285]{CohNew1984}. They also 
  raise the basic question about the precise relation between $\rho(X)$ and $n\E |x_{ij}|^2$
  in~\cite[Section 4]{CohNew1984}, which has initiated many subsequent works.
  
  From now on, without loss of generality, we assume that
  $\E |x_{ij}|^2 =\frac{1}{n}$ to fix the scale. The question of Cohen and Newman from one side
   was answered shortly afterwards
   by Geman~\cite{Geman1986} (see also~\cite{Hwang86}) by showing that $\limsup_n \rho(X) \le 1$ almost surely
   under a high moment condition  on $\sqrt{n}x_{ij}$, which was relaxed by Bai and Yin~\cite{BaiYin} to a finite four moment condition.
   This was further relaxed to $2+\epsilon$ moment in~\cite{BCCT18}  more than thirty years later
   and finally the result under the optimal
   second moment condition was achieved by Bordenave, Chafa\"i and Garcia-Zelada very recently in~\cite{BCG22}.
  The matching lower bound follows from Girko's circular law~\cite{Girko1984, Bai1997}, asserting 
  that the density of eigenvalues converges to the uniform measure on the unit disk, 
  that was proven in 2010 by Tao and Vu~\cite{TV10b} under the weakest  condition that the second moment of $\sqrt{n} x_{ij}$ is finite. 
  This yields the asymptotics $\lim_{n\to\infty} \rho(X)=1$.  
   An almost optimal bound on the
  speed of convergence, in the form
  \begin{equation}
  |\rho(X)-1|\le n^{-1/2+\epsilon}
  \label{rhospeed}
  \end{equation}
    with  high probability, was established  a few years ago
  in~\cite{AEK19b} under high moment conditions\footnote{The paper~\cite{AEK19b} also  proved the 
  analogue of~\eqref{rhospeed} for much  more general non-Hermitian random matrices 
  with independent centred entries that may have different distribution; here even the
  identification of the leading term of $\rho(X)$ was a challenge.  Later this leading term  was also identified
  even if the entries $x_{ij}$ have certain nontrivial correlation in~\cite{AK21}.}.  
  Note that all these results would hold verbatim for $\max\Re \mbox{Spec}(X)$ as well, since the difference between
   $\rho(X)$ and $\max\Re \mbox{Spec}(X)$ is not visible at scales that the above mentioned results can identify. 
   In contrast, we now compute the size $\rho(X)$ and $\max\Re \mbox{Spec}(X)$ 
   with a very high precision so that their difference indeed emerges. 
  
  Much more precise estimates on $\rho(X)$ are known for the special  Ginibre case, i.e.  where $x_{ij}$ are 
  standard i.i.d. Gaussians; in this case
    exact formulas are available. The complex Ginibre case
  is especially simple since $X$ has an additional 
   rotational symmetry (in distribution). Using Kostlan's observation~\cite{Kostlan92} the precise result, 
   stated in  this form by Rider in~\cite{R03}, asserts that
    \begin{equation}\label{Crho}
    \rho(X) \stackrel{\text{d}}{=} 1 + \sqrt{\frac{\gamma_n}{4n}} + 
    \frac{1}{\sqrt{4n\gamma_n}} G_n, \qquad \gamma_n:=\log n -2\log\log n-
    \log (2\pi),
\end{equation}
where $G_n$ converges in distribution to a standard \emph{Gumbel random variable}, i.e.
\[
\lim_{n\to \infty} \mathbf{P} (G_n\le t) = \exp{(-e^{-t})}
\]
for any fixed $t\in \R$. The analogous result for the real Ginibre  
ensemble, obtained by Rider and Sinclair in~\cite{RS14},   shows
that~\eqref{Crho} also holds for the real Ginibre case with the same
scaling factor $\gamma_n$, but the limit of $G_n$ 
is a  rescaled Gumbel  with distribution function $\exp{(-\frac{1}{2}e^{-t})}$.
The emergence of the Gumbel law is not surprising since the few eigenvalues that compete
for the maximal modulus are typically far away from each other, hence are
almost independent (in fact asymptotically they form a Poisson process~\cite{Bender, Ake, maxRe_Gin}).
 This scenario is very different from the strong correlation among
the top eigenvalues of a Hermitian random matrix with their Tracy-Widom fluctuations.

In the spirit of {\it universality} of spectral statistics for random matrices, it is very natural to 
ask whether~\eqref{Crho} holds beyond Gaussians, i.e. for matrices $X$ with more general i.i.d. entry distribution. 
This obvious question was raised several times, e.g. in~\cite[Section 7]{BC12}, in the first online
version of~\cite{BCG22} as well as in D. Chafa\"i's  excellent blog~\cite{Chblog}.
Our main result in this paper is the optimal asymptotics for $\rho(X)$ up to the precision of the conjectured
Gumbel scale, i.e. that for any $\epsilon>0$ there is a $C_\epsilon$ such that
\begin{equation}
\label{maxR}
 \limsup_n \mathbf{P}\Bigg( \Big|\rho(X)-1- \sqrt{\frac{\gamma_n}{4n}}\Big|  \ge \frac{C_\epsilon}{\sqrt{n\log n}} \Bigg) 
 \le \epsilon.
\end{equation}
Note that this result substantially  improves on~\eqref{rhospeed} by replacing the $n^\epsilon$ factor
by a precise   three-term  asymptotics encoded in $\gamma_n$.

An analogous result for $\max\Re \mbox{Spec}(X)$ instead of $\rho(X)$
was recently obtained in~\cite{maxRe}, with the only difference that the coefficients of
the three terms in $\gamma_n$ in~\eqref{Crho} are modified to 
$\gamma_n' = \frac{1}{2} \log n -\frac{5}{2}\log\log n -\frac{1}{2} \log(2\pi^4)$.
We remark that three-term asymptotics for extremal statistics is quite common;
in the random matrix context this has been extensively investigated 
in connection with the Fyodorov-Hiary-Keating conjecture~\cite{FHK}
for extreme values of the characteristic polynomial of various ensembles~\cite{FS,ABB, PZ, CMN, Lambert, PZ1}.

 We remark that more than a year after the completion of the current paper in October 2022, we succeeded in 
proving the original conjecture on the universality of the Gumbel distribution both for 
$\max\Re \mbox{Spec}(X)$ and $\rho(X)$, the result was posted on ArXiv in December 2023
 \cite{Gumbel}. We will comment on the relation between 
these two works at the end of Section~\ref{sec:summary}. 

\bigskip

Now we explain the key novelties of our current work compared with the proof of
the similar result for $\max\Re \mbox{Spec}(X)$ in~\cite{maxRe}.  Both proofs start with
Girko's formula~\cite{Girko1984} for the   linear eigenvalue statistics that translates the original non-Hermitian
spectral problem to a Hermitian one:
\begin{equation}\label{girko}
\sum_{\sigma\in \mbox{Spec}(X)} f(\sigma)=-\frac{1}{4 \pi} \int_{\C} \Delta f(z) A(z) \dd^2 z, \qquad
A(z):= \int_0^\infty \Im \Tr G^{z}(\ii \eta) \dd \eta ,
\end{equation}
where $G^z$  is family of Hermitized resolvents, 
parametrized by a new spectral parameter $z\in \C$:
\begin{align}\label{initial1}
	 G^{z}(w):=(H^{z}-w)^{-1}, \qquad H^{z}:=\begin{pmatrix}
		0  &  X-z  \\
		X^*-\overline{z}   & 0
	\end{pmatrix},\qquad w \in \C\setminus \R.
\end{align}
The test function $f$ is chosen carefully  as a (smoothed) characteristic function
of a domain where the few largest (in modulus) eigenvalues of $X$ are expected. In the current analysis of $\rho(X)$
the support of $f$ will be a narrow  annulus with radius $1+\sqrt{\gamma_n/4n}$ and width $C/\sqrt{n\log n}$;
in the study of $\max\Re \mbox{Spec}(X)$ the corresponding domain was a narrow vertical  rectangle of width
$n^{-1/2}$ and height $n^{-1/4}$ (modulo logarithmic factors) at a distance $1+\sqrt{\gamma_n'/4n}$ from the
origin.  Both domains are meticulously scaled to contain (essentially)
 finitely many eigenvalues and to
 contain the one that realizes the extremal statistics. The precise dimensions are
determined  from 
the explicit Ginibre calculation.
 
 The key point is that we work in the {\it atypical} or {\it large deviation} regime for the eigenvalues of $X$. Locally 
 the eigenvalues have density $n$ on the unit disk that decays as $n e^{-n(|z|^2-1)/2}$ 
 outside of the disk, $|z|>1$ as seen from the explicit formulas.
  The eigenvalues fluctuate near the unit circle\footnote{This can be explicitly computed in the Ginibre case while for  the general i.i.d. case one can infer from  the universality of local correlation functions near the 
 edge~\cite{CES21}.} only on scale $n^{-1/2}$.
 Hence in a small neighborhood 
 of any fixed point $z_0$ on the unit circle there will be typically no  eigenvalue that is
  $\sqrt{\gamma_n/4n}\sim n^{-1/2}\sqrt{\log n} $ away from the unit disk. The  relatively large area  of our domains
  make up for the small  probability of finding an eigenvalue locally. 
  
  Technically we work
  on the right hand (Hermitian) side of Girko's formula 
  and we essentially need two types of information on $H^z$. To explain them
  we may write
  \begin{equation}\label{Az}
     A(z)=  \int_0^\infty \Im \Tr G^{z}(\ii \eta) \dd \eta =  \int_0^{\eta_0} \sum_{i=-n}^n \frac{\eta}{ (\lambda_i^z)^2+\eta^2}
     \dd\eta + \int_{\eta_0}^\infty \Im \Tr G^{z}(\ii \eta) \dd \eta,
  \end{equation}
  where $\lambda_i^z$ are the eigenvalues of $H^z$  (equivalently, singular values of $X-z$)
   and $\eta_0:= n^{-1+\epsilon}$    is a cutoff parameter. 
As a consequence of the block structure of $H^z$, its $2n$ eigenvalues come in opposite pairs, 
  $\lambda_{-i}^z=-\lambda_i^z$.
  We intentionally wrote out the integrand in~\eqref{Az} for  small $\eta$ in terms of the Hermitized eigenvalues
    since  in this regime essentially only the  two smallest (in modulus) 
    eigenvalues $\lambda_1^z=-\lambda_{-1}^z$ will play a role, while in the large $\eta$
    regime the entire resolvent is relevant.
   Among the two terms in~\eqref{Az}, the first one is more critical since the small $\eta$ regime is 
  more sensitive to the behavior of the individual low lying eigenvalues $\lambda_i^z$.
  The second term can be dealt with more robust resolvent methods, as long as $\eta\gg 1/n$ 
   ensuring a small error term in the local law of $G^z$.
   The cutoff $\eta_0$ is therefore chosen to be the smallest possible so that we are still 
    able to use resolvent methods.

    More precisely, for the second term in~\eqref{Az} we use a sophisticated {\it iterative cumulant expansion}\footnote{The iterative cumulant expansion has been systematically developed in~\cite{SX22, SX22+} extending 
  the  iterative  gain from so-called un-matched indices~\cite{EYY12, EKY13} and 
  exploiting that the leading deterministic terms may cancel 
   in certain situations~\cite{LS18, HLY20, HK21}.}
     in the {\it Green function comparison (GFT)}
      argument to compare the Green function $G^z$ 
     for i.i.d. and  Ginibre ensembles. Note that $\eta_0$ is just a  bit above the scale $1/n$ which is the threshold
     for iterative GFT arguments, since every expansion step, roughly speaking, 
     has a potential to gain a factor $1/(n\eta_0)$.  We remark that, unlike in most applications of GFT, 
     in our case there are no explicit formulas for the distribution of second term  in~\eqref{Az} even for
     the Gaussian case since the joint distribution of the spectra of $H^z$ for {\it different} $z$'s is unknown
     for Ginibre. Thus after GFT, for the Ginibre case we need to use Girko's formula ``backward" 
     since the distribution of its left hand side is explicitly understood.

    As to the first term in~\eqref{Az}, when $z$ is very close to the unit circle,
      the density $\rho^z$ of eigenvalues of $H^z$ vanishes near zero as a
    cubic root cusp, hence the typical fluctuation scale of $\lambda_1^z$ is of order $n^{-3/4}$. 
    Note that the cutoff $\eta_0$ is much smaller than this scale, thus the first term  in~\eqref{Az}
    is of order $(\eta_0/\lambda_1^z)^2
    \sim n^{-1/2+2\epsilon}$ in the typical regime for $\lambda_1^z$.  For the event where $\lambda_1^z$ is
    atypically small, we will need a lower tail estimate
    in the outside regime, $|z|>1$ of the form 
\begin{equation}\label{tailp}
\P\big( |\lambda_1^z| \leq y n^{-3/4}\big)\lesssim y^2 e^{-n\delta^2/2}, \qquad \delta:=|z|^2-1,
\end{equation}
for any $y>0$ (see~\eqref{tail} below for the precise statement). The smallness in our relevant $y\ll 1$ regime 
comes from two unrelated effects: the $y^2$ factor represents the level repulsion (between $\lambda_1^z$ and
its mirror image $\lambda_{-1}^z=-\lambda_1^z$), while the exponential factor stems from the 
decay of the density to find an eigenvalue $\lambda_1^z$ well within the gap
in the support of the limiting density of states $\rho^z$. 
 The tail bound~\eqref{tailp} with the exponential factor was originally proved and used in~\cite{maxRe} but only for the Ginibre ensemble. 
For the current paper we have to extend it to the i.i.d. case by  another GFT argument
 because now we need to exploit both smallness
effects on the i.i.d. level.

    In the analysis of both terms in~\eqref{Az} the main complication is that Girko's formula~\eqref{girko} contains $\Delta f$
    and for our  $f$, a smoothed characteristic function of a very anisotropic annular regime,
     we have $\int |\Delta f| \sim n^{1/2}$. So in the error terms    we need to obtain
    a precision that is by a factor $n^{-1/2}$ better than any leading term we compute.  This is 
    an unnaturally stringent requirement, but it is unclear how to exploit the cancellation effect 
    present in the $\int \Delta f(z) [\ldots] \dd^2 z$ integrals for the error terms that require to insert
    an absolute value. Moreover, we need
    these estimates not only in expectation but also in variance sense since we need to 
    prove a concentration of $\sum f(\sigma)$ around its nontrivial mean  to show that 
    with some large  probability  there is an eigenvalue on the support of $f$.
    
    All these difficulties have already been present in~\cite{maxRe} but less dramatically.
    The relevant domain, the vertical rectangle, used in~\cite{maxRe} is less anisotropic
    and yields only an additional factor $\int |\Delta f|\sim n^{1/4}$ to fight against, 
     which we could handle  solely with GFT methods. 
    In particular
    we could choose the cutoff threshold at an intermediate value $\eta_0= n^{-7/8-\epsilon}$ to make
    the typical contribution of the first term in~\eqref{Az} of order $(\eta_0/\lambda_1^z)^2
    \sim n^{-1/4-2\epsilon}$, immediately compensating for the loss in $\int |\Delta f|\sim n^{1/4}$, hence negligible.
    Moreover, the larger $\eta_0$ threshold made the analysis of the second term in~\eqref{Az} easier;
    since GFT gains a factor $1/(n\eta_0)\ge n^{-1/8+\epsilon}$ in every step, hence after at most three iteration
    steps we could compensate for $\int |\Delta f|\sim n^{1/4}$. In the current work  $1/(n\eta_0)=n^{-\epsilon}$, so
    we practically need to perform  $1/\epsilon$ iterations and carefully track all the leading terms
    whose integral against $\Delta f$ is small by explicit calculation, similarly to a few explicit terms in~\cite{maxRe}. 
    
     More fundamentally, however, all the  methods and ideas used in~\cite{maxRe}  even in some improved 
     form, like a more refined iteration scheme,  
     would not be able to handle the $\int |\Delta f|\sim n^{1/2}$ loss present in the spectral radius problem.
  This loss  is prohibitive in both 
     regimes in~\eqref{Az}; but it is the simplest to see in the contribution of order
      $(\eta_0/\lambda_1^z)^2    \sim n^{-1/2+2\epsilon}$ to $A(z)$
       from the typical $\lambda_1^z$-regime as we discussed above.
      On the one hand, the threshold $\eta_0$ cannot be chosen smaller than $n^{-1+\epsilon}$
       otherwise the iterative cumulant
      expansion  does not work. On the other hand,  $\lambda_1^z$ is a genuinely fluctuating quantity, it cannot be
      approximated effectively by any leading deterministic term  with the hope to integrate it out explicitly against 
      $\Delta f$ without inserting absolute value. So with our best efforts we are still off by an $n^\epsilon$ factor.
     
     We thus need to  exploit 
      a new mechanism that we call the {\it $|z_1-z_2|$-decorrelation effect}, which constitutes
      the main methodological novelty of this paper. 
      Based upon the explicit Ginibre formulas 
      the local spectrum of $X$ around a point $z_1$ and around $z_2$ are asymptotically independent 
      if $|z_1-z_2|\gg n^{-1/2}$. This  intuitively indicates, but {\it does not prove}\footnote{Although there is no direct  relation between the eigenvalues   $\sigma\in \mbox{Spec}(X)$ and the singular values of $H^z$ 
  apart from the trivial fact that $z=\sigma$ is an eigenvalue if and only if
  $H^z$ has a zero singular value, we still expect their correlation decay to be similar.}, 
  that $\lambda_1^{z_1}$
      and $\lambda_2^{z_2}$ are also asymptotically independent if   $|z_1-z_2|\gg n^{-1/2}$. Similarly, one expects
      independence of    $G^{z_1}$ and $G^{z_2} $ in the same regime.  Both these independences
      are necessary to handle the variance of~\eqref{girko}. Clearly we have
      \begin{equation}\label{cov}
         \V \Big[ \sum_{\sigma\in \mbox{Spec}(X)} f(\sigma)\Big]=\frac{1}{(4 \pi)^2} \iint_{\C} \Delta f(z_1) \Delta f(z_2) 
     \Cov\big(    A(z_1),
         A(z_2) \big) \dd^2 z_1 \dd^2 z_2,
      \end{equation}
     and we will prove that in the regime where $|z_1-z_2|\ge n^{-\gamma}$, with some small $\gamma>0$,
     the covariance in~\eqref{cov} is much smaller than the geometric mean of the corresponding two variances.
     In the opposite regime, $|z_1-z_2|< n^{-\gamma}$, we gain an additional factor $n^{-\gamma}$
     from the integration volume. These steps provide an additional small factor to compensate for 
     the $n^\epsilon$ explained above.
   
   A certain  version of the $|z_1-z_2|$-decorrelation effect
   has already been  used in the proof of the CLT for linear statistics~\cite[Theorem 5.2]{CES19}
   and~\cite[Theorem 3.2]{Zigzag}
   in the form of a two-resolvent local law for $\Tr G^{z_1}  A G^{z_2} B$ where the error
   term improves if $z_1$ and $z_2$ are far away. However, these results were in the 
    bulk regime $|z_1|, |z_2|\le 1-\epsilon$, now we need this effect  also at the edge (and even a bit beyond)
    and  more importantly in the  atypical tail regimes.
    
   The most remarkable instance of the $|z_1-z_2|$-decorrelation effect is  used for the
   first term in~\eqref{Az} in the atypical  regime of $\lambda_1^z$.  We prove the 
   essential decorrelation of the  lower tails of $\lambda_1^{z_1}$ and $\lambda_1^{z_2}$
  in the form
     \begin{align}\label{lambdatail1}
		\P\Big( |\lambda_1^{z_1}| \leq n^{-3/4-\alpha}, ~ |\lambda_1^{z_2}| \leq n^{-3/4-\alpha}\Big) 
		\lesssim \Big[\P\big( |\lambda_1^z| \leq 10n^{-3/4-\alpha}\big)\Big]^2 + O(n^{-D})
	\end{align}
	if $|z_1-z_2|\ge n^{-\gamma}$
with some small $\alpha, \gamma$ and any large $D$.  To appreciate this estimate,
notice that in our regime the probability on the right hand side is unusually  small for two unrelated reasons as we explained after~\eqref{tailp} (with $y\sim n^{-\alpha}$) 
and our decorrelation estimate accurately catches both effects.

We remark that we need and prove  the bound~\eqref{lambdatail1} only for the Ginibre ensemble
 (see Proposition~\ref{lemma_new_1} below), but it is
easy to extend to arbitrary i.i.d. matrices with a weaker error term.
 Note also that the left hand side of~\eqref{lambdatail1} is not accessible with 
explicit calculations since no formula is available for the {\it joint} distribution of the spectra of $H^{z_1}$ and 
$H^{z_2}$. Our proof uses the decorrelation mechanism of
the  {\it Dyson Brownian Motion (DBM)} with weakly correlated driving processes. It is somewhat surprising that sophisticated DBM methods are apparently necessary for a purely equilibrium result.

 The second instance of the $|z_1-z_2|$-decorrelation effect is to estimate
the contribution of the  first term in~\eqref{Az} to the covariance in~\eqref{cov} 
on the typical event when $\lambda_1^z$ is not too small. On this event we approximate
this term by a resolvent  $G^z(\ii\wt\eta)$ with an  increased spectral parameter $\wt\eta= n^{-3/4-\alpha}$,
then we have 
\begin{equation}\label{covv}
	\Cov\big( \Tr G^{z_1} (\ii\wt\eta), \Tr G^{z_2} (\ii\wt\eta)\big) 
	\leq   n^{-\gamma'}   \Big[ \V \big(\Tr G^{z_1} (\ii\wt\eta)) \V \big(\Tr G^{z_2} (\ii\wt\eta))\Big]^{1/2}
\end{equation}
as long as $|z_1-z_2|\ge n^{-\gamma}$ with some small  $\gamma,\gamma'>0$. 
We prove~\eqref{covv} for the Ginibre case in Proposition~\ref{lemma_new_2} but it can be
directly extended to the general i.i.d. case using Proposition~\ref{prop2_old}. We point out that while this decorrelation 
inequality is natural and a similar bound was proven in the bulk regime $|z_i|\le 1-\epsilon$
 in~\cite[Proposition 3.4]{CES19},
now we need~\eqref{covv} in the atypical regime well outside of the unit disk, $|z_i|- 1\gtrsim \sqrt{\log n/n}$
 where $\Tr G^z$ is much smaller.
This smallness factor needs to be preserved along all the estimates. 
 Moreover a similar $|z_1-z_2|$-decorrelation effect will also be used to estimate products
  of multiple resolvents on the level $\eta_0=n^{-1+\epsilon}$ stemmed from the
   iterative GFT for the second term in~\eqref{Az}; see \eg Lemma \ref{lemma4}.

In summary, the proof of the precise three term asymptotics for the spectral radius is not only
technically more demanding than the similar analysis for the rightmost eigenvalue, 
but it requires to explore a new decorrelation mechanism in the atypical regime.  We did not
mention several other additional difficulties in this introduction, but they will be highlighted
in Section~\ref{sec:summary}, where we give a more extensive summary of the proof strategy.

\subsection*{Notations and conventions}

We introduce some notations we use throughout the paper. For integers \(k,l\in\N \) with $k\leq l$ we use the notation 
\(\llbracket k, l \rrbracket:= \{k,k+1,\dots, l\}\).  For positive quantities \(f,g\) we write \(f\lesssim g\) and \(f\sim g\) if \(f \le C g\) or \(c g\le f\le Cg\), 
respectively, for some constants \(c,C>0\).
For $n$-dependent positive
sequences $f=f_n, g=g_n$ we also introduce $f\ll g$ indicating that $f_n=o(g_n)$. Even if not stated explicitly, $n$ is always taken sufficiently large depending on all other parameters throughout the paper. We always use $c, C>0$ to denote some $n$-indepedent constants that might be different from line to line.

We denote vectors by bold-faced lower case Roman letters \(\mathbf{x}, \mathbf{y}\in\C ^k\), for some \(k\in\N\). 
Vector and matrix norms, \(\lVert{\mathbf x}\rVert\) and \(\lVert A\rVert\), indicate the usual Euclidean norm 
and the corresponding induced matrix norm. For any \(2n\times 2n\) matrix \(A\) we use 
the notation \(\langle A\rangle:= (2n)^{-1}\mathrm{Tr}  A\) to denote the normalized trace of \(A\). 
Moreover, for vectors \({\mathbf x}, {\mathbf y} \in\C ^n\) and matrices \(A,B\in \C ^{2n\times 2n}\) we define 
\[ 
\langle{\mathbf x},{\mathbf y}\rangle:= \sum \overline{x}_i y_i, \qquad \langle A,B\rangle:= \langle A^*B\rangle. 
\]
Moreover, we use $\Delta = 4\partial_z\partial_{\bar z}$ to denote the usual Laplacian
and \(\mathrm{d}^2 z \) denotes the Lebesgue measure on $\C$.

We use $\E^{\mathrm{Gin}}$ and $\E$ to denote the expectation with respect to the Ginibre ensemble and generic i.i.d. marices repsectively. We also use similar notations $\V^{\mathrm{Gin}}$ and $\V$ for the corresponding variances. We will use the concept of ``with very high probability'' meaning that for any fixed \(D>0\) the probability of the event is bigger
than \(1-n^{-D}\) if \(n\ge n_0(D)\). Moreover, we use the convention that \(\xi>0\) denotes an arbitrary small 
constant which is independent of \(n\). Finally, we introduce the notion of 
\emph{stochastic domination} (see e.g.~\cite{EKYY13}): given two families of non-negative random variables
\[
X=\left(X^{(n)}(u) : n\in\N, u\in U^{(n)} \right)\quad\text{and}\quad Y=\left( Y^{(n)}(u) : n\in\N, u\in U^{(n)} \right)
\] 
indexed by \(n\) (and possibly some parameter \(u\)  in some parameter space $U^{(n)}$), 
we say that \(X\) is stochastically dominated by \(Y\), if for all \(\xi, D>0\) we have \begin{equation}\label{stochdom}
	\sup_{u\in U^{(n)}} \mathbb{P}\left[X^{(n)}(u)>n^\xi  Y^{(n)}(u)\right]\leq n^{-D}
\end{equation}
for large enough \(n\geq n_0(\xi,D)\). In this case we use the notation \(X\prec Y\) or \(X= \OO_\prec(Y)\). We often use the notation 
$\prec$ also for deterministic quantities, then the probability in~\eqref{stochdom} is
zero for any $\xi>0$ and sufficiently large $n$.

\section{Statement of the main result} 

We consider $n\times n$  matrices $X$ with independent identically distributed (i.i.d.) entries $x_{ab}\stackrel{\mathrm{d}}{=}n^{-1/2}\chi$. 
On the $n$-independent random variable $\chi$ we make the following assumption:

\begin{assumption}
	\label{ass:mainass}
	We assume that $\E \chi=0$, $\E |\chi|^2=1$; additionally in the complex case we also assume that 
	$\E \chi^2=0$.
	 Furthermore, for any $p\in\N$ we assume that there exists constants $C_p>0$ such that
	\begin{equation}
		\label{eq:hmb}
		\E\big|\chi^p\big|\le C_p.
	\end{equation}
Moreover, we assume that there exists $\alpha,\beta>0$ such that the probability density of $\chi$, denoted by $g$, satisfies 
\begin{align}\label{assumption_b}
	g \in L^{1+\alpha}(\mathbb{F}), \quad \|g\|_{1+\alpha} \leq n^{\beta}, \qquad \mathbb{F}=\R~\mathrm{or}~\C.
\end{align}
\end{assumption}
Let $\{\sigma_i\}_{i\in \llbracket 1, n \rrbracket}$ be the eigenvalues of $X$ and define the spectral radius 
$$ \rho(X):=\max_{i \in \llbracket 1,n\rrbracket}|\sigma_i|.$$ 
The main result of this paper is the estimate of the precise size of the spectral radius $\rho(X)$ in the complex case:
\begin{theorem}\label{main}
	Let $X$ be an $n\times n$ matrix satisfying\footnote{ The matrix entries of $X$ do 
	not have to be identically distributed. Our proof still works with minor modifications
	if $\E x_{ij}=\E x_{ij}^2=0$, 
	 $\E |x_{ij}|^2=1/n$ and $\E |\sqrt{n}x_{ij}|^p \le C_p$, but for simplicity we consider the i.i.d. case only.  }
	 Assumption \ref{ass:mainass} in the complex case. Set
	$$\gamma_n:= \log n-2\log \log n-\log 2 \pi.$$
	Then we have
	\begin{align}\label{spectral_radius}
		\lim_{n\rightarrow \infty}\P \Big( \Big|\rho(X)-1-\sqrt{\frac{\gamma_n }{4n}} 
		\Big| \geq \frac{C_n}{\sqrt{n \log n}}\Big)=0,
	\end{align}	
	for any\footnote{Our proof also gives an  effective control on the probability  in~\eqref{spectral_radius}
	of order $O(C_n^{-\tau}+ n^{-\tau})$ for some small $\tau>0$.} sequence $C_n \rightarrow \infty$.
\end{theorem}

\begin{remark}\label{TV_tech} 
	The assumption~(\ref{assumption_b}) is used only 
	to control the  unlikely event that there is a tiny singular value of $X-z$ in a simple way (see~\eqref{tiny_int} below). 
	We make this assumption only  to simplify the presentation of the proof, but it can easily be removed with a separate argument as
	in~\cite[Section 6.1]{TV15} (see also a slightly streamlined version in \cite[Section 2.2]{Kopel15}) 
	as explained in \cite[Remark 2.2]{maxRe}.	
	We will not present the details here
	since they are fairly standard and they are independent of our main arguments. 	
\end{remark}

Similarly to \cite{maxRe}, we stated the main result only for the complex case. Even though it also holds for the real case,
we do not carry out the complete proof, the reason is explained in \cite[Section 2.4]{maxRe}. 
 In particular, the GFT argument would still work with minor modifications; see Remark~\ref{handwave_real}. 
  Our  proof for the spectral radius has one more new ingredient, the Dyson Brownian motion analysis in 
  Section~\ref{sec:DBM}, that in principle is sensitive
to the symmetry class. The necessary modifications to the DBM analysis from the complex to the real 
case have been handled in detail in \cite[Section 7]{real_CLT} assuming  that we have the overlap bound
(see~\eqref{ovass} later) for all eigenvectors not just the ones near the cusp regime that we now use in the complex case.
Since the cusp regime is the most complicated one, similar but easier arguments would also give the overlap bound
uniformly in the spectrum, but we did not work out  the detailed proof.

We also comment that our result in (\ref{spectral_radius}) directly implies that the sequence of normalized fluctuations
$$G_n:=\sqrt{4n \gamma_n}\Big( \rho(X)-1-\sqrt{\frac{\gamma_n}{4n}} \Big)$$
has subsequential limits by Prokhorov’s theorem. The limit is conjectured to be a unique Gumbel distribution as in (\ref{Crho}) for the special Ginibre case. Similar statements can also be found in \cite[Remark 2.4]{maxRe} for the rightmost eigenvalue. The uniqueness of the limit and the universality of the Gumbel distribution are left to future work.

\section{Summary of the  methods of the proof}\label{sec:summary}

In this section we collect necessary background information and we sketch the main
ideas of the proof, explaining the main novelties of our approach.

\subsection{Hermitization and local laws}

From the local circular law in \cite[Theorem 1.2]{BYY14} and \cite[Theorem 2.1]{AEK19b}, for any small $\tau>0$, the eigenvalue which determines the spectral radius  is located 
 in the following annulus
\begin{align}\label{omega_0}
	\Omega_0:=\Big\{ z=re^{\ii \theta} \in \C \; : \; r \in \Big[1-\frac{n^{\tau}}{\sqrt{n}},1+\frac{n^{\tau}}{\sqrt{n}}\Big],\quad \theta \in [0,2\pi)\Big\},
\end{align}
with very high probability.  To prove the precise location in~(\ref{spectral_radius}), we introduce a narrow annulus $\Omega_1\subset \Omega_0$ around the center circle with radius $L_n:=1+\sqrt{\frac{\gamma_n}{4n}}$ with width $l_n:=\frac{C_n}{\sqrt{n \log n}}$~~$(C_n \gg 1)$ beyond the conjectured Gumbel scale and its complement annulus $\Omega_2 \subset \Omega_0$. We may assume without loss of generality that $C_n\ll \sqrt{\log n}$. More precisely, we define
\begin{align}\label{omega_2}
		\Omega_1:=\Big\{ r \in [L_n-l_n,L_n+l_n],~\theta \in [0,2\pi)\Big\}, \qquad \Omega_2:=\Big\{ r \in \Big[L_n+l_n,1+\frac{n^{\tau}}{\sqrt{n}}\Big],~ \theta \in [0,2\pi)\Big\}. 
\end{align}
It then suffices to show the expectation estimates
 \begin{align}\label{upper}
 	\E [\#\{\sigma_i \in \Omega_2\}]=o(1), \qquad \E [\#\{\sigma_i \in \Omega_1\}] \geq C,
 \end{align}
 for some constant $C>0$, and the concentration estimate
 \begin{align}\label{concentrate}
 	\E \big|\#\{\sigma_i \in \Omega_1\}-\E[\#\{\sigma_i \in \Omega_1\}]\big|=o(1)\E[\#\{\sigma_i \in \Omega_1\}].
 \end{align}
Given the estimates in (\ref{upper}) and (\ref{concentrate}), we can prove Theorem~\ref{main} using the Markov inequality; see similar arguments in \cite[Section 3]{maxRe} locating the rightmost eigenvalue among $\{\sigma_i\}_{i\in \llbracket 1, n \rrbracket}$.

\bigskip

To study the eigenvalues of $X$, we use Girko's Hermitization formula \cite{Girko1984, TV15}. For any $T>0$ and for any compactly supported smooth test function $f \in C_c^{2}(\C)$, we have
\begin{align}\label{linear_stat}
	\sum_{i=1}^{n} f(\sigma_i)=&-\frac{1}{4 \pi} \int_{\C} \Delta f(z) \int_0^T \Im \Tr G^{z}(\ii \eta) \dd \eta \dd^2 z+\frac{1}{4 \pi} \int_{\C} \Delta f(z) \log |\det (H^{z}-\ii T)| \dd^2 z,
\end{align} 
where the  $2n \times 2n$ Hermitian matrix $H^{z}$ and its resolvent $G^{z}$ are defined by
\begin{align}\label{initial}
	H^{z}:=\begin{pmatrix}
		0  &  X-z  \\
		X^*-\overline{z}   & 0
	\end{pmatrix}, \qquad G^{z}(w):=(H^{z}-w)^{-1}, \quad w \in \C\setminus \R,~z\in \C.
\end{align}
The $2\times 2$ block structure of $H^z$ induces a symmetric spectrum around zero, i.e. the eigenvalues of $H^{z}$ are $\{\lambda^{z}_{\pm i}\}_{i\in \llbracket 1, n \rrbracket}$ (labelled in a non-decreasing order) with $\lambda_{-i}^z=-\lambda_i^z$ for $i\in \llbracket 1, n \rrbracket$. Note that $\{\lambda_i^{z}\}_{i\in \llbracket 1, n \rrbracket}$ exactly coincide with the singular values of $X-z$. Moreover, the corresponding normalized eigenvectors of $\lambda^z_{\pm i}$ are denoted by $\mathbf{w}^z_{\pm i}=(\mathbf{u}^z_{i}, \pm \mathbf{v}^z_{i})$. As a consequence of the spectral symmetry of $H^{z}$, we find that, on the imaginary axis,
\begin{align}\label{Gvv}
	G_{vv}^{z} (\ii \eta)=\ii \Im G_{vv}^{z}(\ii \eta), \quad \Im G^{z}_{vv}(\ii \eta) >0, \qquad   v \in \llbracket 1, 2n \rrbracket, \quad \eta>0.
\end{align}

Before we state the local law for $G^{z}$, we first define a deterministic $2n \times 2n$ block constant matrix by
\begin{align}\label{Mmatrix}
	M^{z}(w)=\begin{pmatrix}
		m^{z}(w)  &  -zu^z(w)  \\
		-\overline{z}u^z(w)   & m^{z}(w)
	\end{pmatrix},\qquad  u^z(w):=\frac{{m}^z(w)}{w+ {m}^z(w)},
\end{align}
where ${m}^z(w)$ is the unique solution of the scalar equation
\begin{align}
	\label{m_function}
	-\frac{1}{{m}^z(w)}=w+{m}^z(w)-\frac{|z|^2}{w+{m}^z(w)}, \quad \mbox{with}\quad \mathrm{Im}[m^z(w)]\mathrm{Im}w>0.
\end{align}
 In fact, the given form of $M^z$  comes from the corresponding \emph{matrix Dyson equation} 
  (MDE)~\cite{AEK19a} for the Hermitian
matrix $H^z$ which has the form
\begin{equation}\label{dyson}
 - \big[ M^z(w)\big]^{-1} = w + Z + \mathcal{S}\big[ M^z(w) \big], \qquad 
 Z:= \begin{pmatrix} 0 & z \cr \bar z & 0 \end{pmatrix}, \quad w\in \C\setminus \R,
\end{equation}
where $\mathcal{S}$ is the covariance (or self-energy) operator acting on $2n\times 2n$ matrices and given by
\begin{equation}\label{S}
   \mathcal{S}[T]:= \E \Big[\begin{pmatrix} 0 & X \cr X^* & 0 \end{pmatrix} T \begin{pmatrix} 0 & X \cr X^* & 0 \end{pmatrix}\Big],
\end{equation}
where the expectation is taken for the i.i.d. matrix $X$. Simple calculation shows that in our case
$$
  \mathcal{S}[T] =  \langle T \rangle - \langle TE_- \rangle E_-, \qquad E_-: = \begin{pmatrix} 1 & 0 \cr 0 & -1 \end{pmatrix}.
$$
It follows from the
general theory of MDE (see e.g.~\cite{HRS07})
that~\eqref{dyson} has a unique solution with the side condition that $(\Im w)\Im M^z(w)>0$.
Owing to the special block constant form of $Z$ and that the image of $\mathcal{S}$ is also a
block constant matrix, it is easy to see that $M^z$ is also block constant with identical 
upper and lower blocks in its diagonal. In particular $\mathcal{S}[M^z] =  \langle M^z \rangle$
and thus~\eqref{dyson} simplifies to
\begin{equation}\label{dyson1}
 - \big[ M^z(w)\big]^{-1} = w + Z +  \langle M^z \rangle.
\end{equation}
This observation allows us to use many results from~\cite{EKS20}, \cite{AEK18b} and \cite{AEK19a}.
Note that these papers  considered a matrix
Dyson equation with a covariance operator satisfying $\mathcal{S}[T]\ge c \langle T \rangle $
to hold for any matrix $T>0$ with some positive constant $c$ (the so-called 
 \emph{flatness condition}). 
The main observation is that even though our $\mathcal{S}$ does not satisfy the flatness condition,
for the purpose of analysing $M^z$, 
defined as the solution of~\eqref{dyson} with $(\Im M^z)(\Im w)>0$,
 we may replace $\mathcal{S}$ from~\eqref{S} with $\mathcal{S}'[T]:= \langle T \rangle$
thanks to~\eqref{dyson1}. In other words, our $M^z$ is the same as the solution to the MDE corresponding to
a deformed Wigner matrix with deformation given by the Hermitian matrix $Z$, hence the results
 from~\cite{EKS20, AEK18b, AEK19a} directly apply.

Note that the functions $m^{z}$ and $u^{z}$  depend only on the radial part  $r=|z|$
of $z$.  Both functions (hence $M^z(w)$)  have continuous extensions to the real axis
that we will denote by $M^z(E\pm \ii 0)$ for $E\in \R$.
The \emph{self-consistent density of states} is defined as
\begin{equation}
\label{eq:scdos}
 \rho^z(E):=\pi^{-1}\Im \<M^z(E+\ii 0)\>,
 \end{equation}
 where $\langle M \rangle = \frac{1}{2n} \Tr M$ denotes the normalized trace. 
 If $|z|\leq 1$, then  the density $\rho^z(E)$
  has a local minimum at $E=0$ of height $\rho^z(0) \sim (1-|z|^2)^{1/2}$ from \cite[Eq. (3.13)]{AEK19b} 
  and the general shape analysis \cite[Theorem 7.1(d) and Remark 7.3]{AEK18b}.
  For $|z|=1$ the density has a cubic cusp singularity at the origin.
  Moreover if $|z|> 1$, then there is a small gap $[-\frac{\Delta}{2},\frac{\Delta}{2}]$ 
   in the support of the symmetric density function $\rho^z(E)$. 
  The size of the gap, $\Delta \sim \big(|z|^2-1\big)^{3/2}$, 
  can be easily computed from taking the imaginary part of~\eqref{m_function} and using the behavior of $\rho^z$
  (see below)
  and hence its Stieltjes transform  $m^z$ near zero.
  
 We extend the density $\rho^z$ to the complex plane,  
 and for $|z|\le 1$ it has the following scaling behaviour
$$ \rho^z(w):=\frac{\Im \<M^z(w)\>}{\pi} \sim (1-|z|^2)^{1/2} +(E+\eta)^{1/3}, \qquad |E|\leq c,\quad 0\leq \eta\leq 1, \qquad w:= E+\ii \eta
 $$
while in the complementary regime
 $|z|\geq 1$, for any $w=\frac{\Delta}{2}+\kappa+\ii \eta$, and a small constant $c$, we have
\begin{align}\label{rho_E}
	\rho^z(w)=\frac{\Im \<M^z(w)\>}{\pi} \sim \begin{cases}
		(|\kappa|+\eta)^{1/2} (\Delta+|\kappa|+\eta)^{-1/6}, &\quad \kappa \in [0,c]\\ 
		\frac{\eta}{(\Delta+|\kappa|+\eta)^{1/6}(|\kappa|+\eta)^{1/2}} , &\quad  \kappa \in [-\Delta/2,0]
	\end{cases},
	\quad 0\leq \eta\leq 1.  
\end{align}
These asymptotics  can be  computed  from solving the cubic equation~\eqref{m_function} by Cardano's formula
and selecting the correct branch satisfying the side condition,  but they actually directly 
follow from~\cite[Proposition 3.2 (ii)]{EKS20} that itself relies on~\cite[Remark 7.3]{AEK18b}.

In particular for $w=\ii \eta$ on the imaginary axis, by taking the real part of \eqref{m_function}, it follows that 
$m^z(\ii \eta)$ is purely imaginary, hence ${m}^z(\ii \eta)=\ii \Im {m}^z(\ii \eta)$ (which also implies that $u^z(\ii\eta)$ is real) and
\begin{align} \label{rho}
\rho^z(\ii \eta)=\frac{\Im m^z(\ii \eta)}{\pi} \sim \begin{cases}
		\frac{\eta}{|1-|z|^2|+\eta^{2/3}}, &\qquad |z| > 1\\
		\eta^{1/3}+|1-|z|^2|^{1/2} , &\qquad |z| \leq 1
	\end{cases},
	\qquad  0\leq \eta\leq 1.
\end{align}

With these notations we have the  following local law for the resolvent $G^z$ 
for $z$ near the edge of the circular law, $||z|-1| \leq \tau$, not only on the imaginary axis, but also in its small neighborhood. 
\begin{theorem}
	\label{local_thmw}
	There are sufficiently small constants $\tau,\tau'>0$ such that
	for any deterministic vectors $\mathbf{x}, \mathbf{y}\in\C^{2n}$
	and matrix $A\in \C^{2n \times 2n}$, for any $z$ with $\big| |z|-1\big| \leq \tau$  and 
	for any $w\in \C_+$ with $|\Re w|\le \tau'$ and any $\eta:=\Im w>0$, we have
	\begin{align}\label{entrywisew}
		\big| \langle \mathbf{x},  (G^{z}(w)-M^{z}(w)) \mathbf{y} \rangle\big| \prec \|\mathbf{x}\| \|\mathbf{y}\| \left( \sqrt{\frac{\rho^z(w)}{n\eta}}+\frac{1}{n\eta}\right),
	\end{align}
	\begin{align}\label{averagew}
		\big|\big\<A\big( G^{z}(w)-M^{z}(w)\big)\big\>\big| \prec \frac{\|A\|}{n\eta}.
	\end{align}
\end{theorem}
In \cite[Theorem 5.2]{AEK19b} this local law was restricted to the imaginary axis~($\Re w=0$)
in the regime $||z|-1|\leq \tau$ and $n^{\xi}\eta_f \leq \eta\leq 1$, for any small $\xi>0$, with 
 $\eta_f>0$ being the local eigenvalue spacing at zero defined in \cite[Eq (5.2)]{AEK19b}. This result
 was further extended down to $\eta=n^{-1}$~(even to any $\eta>0$) in~\cite{CES21} for a smaller regime $||z|-1|\lesssim n^{-1/2}$.
 In Theorem~\ref{local_thmw}  we extend the local law to a small neighborhood of the 
 imaginary axis and for any $z$ with $\big| |z|-1 \big| \leq \tau$. Though such a broad regime 
 of $z$ will not be used in the paper, we still present the statement in this generality. 
The proof heavily relies on both \cite{AEK19b} and \cite{EKS20} and will be presented in Appendix~\ref{app:local_law}.

As a corollary of Theorem \ref{local_thmw}, we have the following rigidity estimates for the eigenvalues. The proof is standard and similar to \cite[Corollary 2.6-2.7]{EKS20} so we omit the details. 
\begin{corollary}
\label{cor:rigidity} 
	Fix any small $\tau>0$. Then for any $\big| |z|-1\big| \leq \tau$, there exists a small $c>0$ such that
\begin{equation}\label{rigidity}
	|\lambda_i^z-\gamma_i^z|\prec \max\left\{\frac{1}{n^{3/4}|i|^{1/4}}, \frac{\Delta^{1/9}}{n^{2/3}|i|^{1/3}}\right\}, \qquad \quad  \mathrm{for}\quad |i|\le c n,
\end{equation}
where $\gamma_i^z$ is the $i$-th quantile of the self-consistent (symmetric) density $\rho^z$, i.e.
$$\int_0^{\gamma_i^z}\rho^z = \frac{i}{2n}, \qquad \gamma_{-i}^z=-\gamma_i^z,\qquad i\in \llbracket 1, n \rrbracket,$$ 
and $\Delta\sim (|z|-1)_{+}^{3/2}$ denotes the size of the gap for $|z|>1$ around zero
in the support of $\rho^z$. 
In addition there exists a small $c'>0$ such that, for any $E_1<E_2$ with $\max\{|E_1|,|E_2|\} \leq c'$, we have
\begin{equation}\label{rigidity3}
	\Big| \#\{j: E_1 \leq \lambda^z_j \leq E_2\}-2n \int_{E_1}^{E_2} \rho^z(x) \dd x \Big| \prec 1.
\end{equation} 
\end{corollary}
Moreover we have the following delocalization estimates for eigenvectors as in \cite[Corollary 3.12]{EKS20}
\begin{corollary}\label{cor:eigenvector}
Fix any small $\tau>0$. Then, for any $\big| |z|-1\big| \leq \tau$, we have
	\begin{align}\label{eigenvector}
		|\<\mathbf{w}^z_{i}, \mathbf{x}\>| \prec n^{-1/2} \|\mathbf{x}\| ,
	\end{align}
for any deterministic unit vector $\mathbf{x}\in \C^{2n}$,   whenever $|i|\le c n$.
\end{corollary}

Next we sketch the proof strategy of Theorem \ref{main}.
\subsection{Proof strategy}
Having reduced the proof of Theorem \ref{main} to showing the upper and lower bound estimates in~(\ref{upper})-(\ref{concentrate}), we introduce the following
smooth cut-off functions $f^-_1, f^+_2 \in [0,1]$ for $\Omega_1$ and $\Omega_2$ in (\ref{omega_2}) respectively.  Fixing a small $\tau>0$,  define
\begin{align}
	&f^-_1(z)=f^-_1(r):=\begin{cases}
	1, &  |r-L_n| \leq 4l_n/5, \\
	0, & |r -L_n| \geq l_n, 
\end{cases}, \qquad z=re^{\ii \theta},\quad\theta\in [0,2\pi),\label{function1}\\
&f^+_2(z)=f^+_2(r):=\begin{cases}
	1, &  r \in  [L_n+l_n, 1+n^{-\frac{1}{2}+\tau}], \\
	0, & r \in [0, L_n+4l_n/5) \cup (1+n^{-\frac{1}{2}+\tau}+l_n/5,\infty),
\end{cases}\label{function2}
\end{align} 
with $L_n$ and $l_n$ given as in (\ref{omega_2}),
such that
\begin{align}\label{lower_upper}
	\#\{\sigma_i \in \Omega_1\} \geq\sum_{i=1}^n f^-_1(\sigma_i), \qquad \#\{\sigma_i \in \Omega_2\} \leq\sum_{i=1}^n f^+_2(\sigma_i).
\end{align}
We also assume that the second derivatives of $f_1^-$ and $f_2^+$ satisfy 
$$|(f^-_1)''(r)| \lesssim l_n^{-2},  \qquad  |(f^+_2)''(r)| \lesssim l_n^{-2},\qquad \qquad  r\in \R,$$
which lead to the following estimates for the $L^\infty$ and  $L^{1}$ norms of $\Delta f_1^-$ and $\Delta f_2^+$, \ie
\begin{align}\label{deltaf}
\|\Delta f\|_\infty \lesssim l_n^{-2} \lesssim n\log n,\qquad \|\Delta f \|_{1} \lesssim l_n^{-1} \lesssim n^{1/2}(\log n)^{1/2}, \qquad f=f_1^-,~f_2^+.
\end{align}
Using (\ref{lower_upper}) and a similar argument with the Markov inequality as in~\cite[Section 3]{maxRe}, the following estimates will directly imply Theorem~\ref{main}~(\cf \cite[Eq. (3.8)-(3.9)]{maxRe}):
\begin{align}
	&\E\Big[\sum_{i=1}^n f^{+}_2(\sigma_i)\Big]=o(1);\qquad  \E\Big[\sum_{i=1}^n f^{-}_1(\sigma_i)\Big] \geq C,\label{main_part_1}\\
	&\E\Big|\sum_{i=1}^n f^{-}_1(\sigma_i)-\E\Big[\sum_{i=1}^n f^{-}_1(\sigma_i)\Big] \Big|=o(1) \Big( \E\Big[\sum_{i=1}^n f^{-}_1(\sigma_i)\Big] \Big).\label{main_part_2}
\end{align}

We start with considering the complex Ginibre ensemble $X$ and use the explicit kernel formula for the eigenvalues $\{\sigma_i\}_{i=1}^n$ of $X$ to prove the following lemma. The proof of this lemma is presented in Appendix~\ref{app:Gincal}.
\begin{lemma}
\label{lemma_Gin}
	Consider the Ginibre ensemble, then we have
	\begin{align}\label{main_part_1_Gin}
		\E^{\mathrm{Gin}}\Big[\sum_{i=1}^n f^{+}_2(\sigma_i)\Big]=o(1);\qquad  \E^{\mathrm{Gin}}\Big[\sum_{i=1}^n f^{-}_1(\sigma_i)\Big] \geq C,
	\end{align}
	for some constant $C>0$, and the concentration result in variance sense
	\begin{align}\label{main_part_2_Gin}
		\V^{\mathrm{Gin}}\Big[\sum_{i=1}^n f^{-}_1(\sigma_i) 
		\Big]=o(1) \Big( \E^{\mathrm{Gin}}\Big[\sum_{i=1}^n f^{-}_1(\sigma_i)\Big] \Big)^2.
	\end{align}
\end{lemma}

We next perform a GFT analysis to extend the Ginibre estimates to i.i.d. cases. 
To deal with the non-Hermitian eigenvalues of $X$, we use Girko's formula in (\ref{linear_stat})  with $T=n^{100}$ and $f=f^-_1$ or $f=f^+_2$
(in the rest of the paper we use the convention that we write $f$ for either $f_1^{-}$
or $f_2^+$)
\begin{align}\label{expectation_eq}
	\sum_{i=1}^n f(\sigma_i)=&-\frac{1}{4 \pi}  \int_{\C} \Delta f(z) \big(\int_{0}^{\eta_0}+\int_{\eta_0}^{T}\big)  \Im\Tr G^{z}(\ii \eta) \dd \eta \dd^2 z+O_\prec(n^{-100})\nonumber\\
=:&I_{0}^{\eta_0}(f)+I_{\eta_0}^{T}(f)+O_\prec(n^{-100}), \qquad \qquad \mbox{with}\quad	\eta_0:=n^{-1+\epsilon},
\end{align}
for a sufficiently small $\epsilon>0$. Note that the last term in (\ref{linear_stat}) with $T=n^{100}$ is bounded by $n^{-100}$ with high probability, as in \cite[Eq. (3.14)-(3.15)]{maxRe}. 
Choosing the truncation level at $n^{-3/4-\epsilon}$ would be natural
since $n^{-3/4}$ is the scale of fluctuations of the smallest eigenvalue $\lambda_1^z$, 
 but to estimate the variance of the more critical $I_{0}^{\eta_0}(f)$ we would need to prove a decorrelation bound on the tails
 of $\lambda_1^z$ and $\lambda_1^{z'}$ (see (\ref{truth}) below) for any $|z-z'| \gg n^{-1/2}$, 
 which is not known even for the Ginibre ensemble. So here we choose a much smaller truncation level $\eta_0$
  to compensate the loss in a weaker version of (\ref{truth}) that holds only for $|z-z'| \gg n^{-\gamma}$ with some small $\gamma>0$.

  Lacking the joint eigenvalue distribution of $H^{z}$ and $H^{z'}$ even in Ginibre cases to estimate the variances of the two $\eta$-integrals in~(\ref{expectation_eq}), we need to first estimate them using Girko's formula 'backward' together with Lemma~\ref{lemma_Gin}. We then extend the corresponding Ginibre estimates for these $\eta$-integrals to the i.i.d. cases using several GFT arguments.
The main new ingredient to perform the GFT analysis is the $|z-z'|$-\emph{decorrelation effect}
which is critically used in several steps of our proof. Additionally, we also need 
an accurate lower tail  bound on the smallest singular value $\lambda_1^z$
for the Ginibre ensemble
	\begin{equation}\label{tail1}
		\mathbf{P}^{\mathrm{Gin}}\left(\lambda_1^z\le y\delta^{3/2}\right)
		\lesssim y^2 (n\delta^2)^{4/3} e^{-n\delta^2(1+\OO(\delta))/2}, \qquad \delta:=|z|^2-1, 
	\end{equation}
 that was proven in \cite{maxRe} and is recalled here in Proposition~\ref{prop:tail}. 
We now informally explain the usage of these inputs.
\begin{enumerate}

	\item {\bf The $|z-z'|$-decorrelation effect}  asserts
	that $G^z$ and $G^{z'}$ becomes largely independent
	when $|z-z'|$ is somewhat large. This effect comes in several places.  First,
	the control over the typical large eigenvalues of $H^z$ has already been exploited in \cite[Theorem 5.2]{CES19}
	using a  local law for $G^{z}AG^{z'}$ with a slight $|z-z'|$ improvement\footnote{ 
	This improvement was
	substantially enhanced recently in~\cite{Zigzag} but here we do not need it.}  over the $z=z'$ case.  
		Second, here we will need a similar effect also for the small eigenvalues.
	 More precisely, we need 
	 that the smallest eigenvalues $\lambda_1^{z}$ and $\lambda_1^{z'}$ are almost independent.  
	 One expects that for any $|z-z'| \gg n^{-1/2}$, 
	\begin{align}\label{truth}
		\P\big( |\lambda_1^z| \leq \eta, ~ |\lambda_1^{z'}| \leq \eta\big) \lesssim \P\big( |\lambda_1^z| \leq \eta\big)\P\big( |\lambda_1^{z'}| \leq \eta\big),
	\end{align}
	but it is hard to prove this strong form of independence in the entire regime $|z-z'| \gg n^{-1/2}$.  However, 
thanks to  the small truncation level in (\ref{expectation_eq}), we only need to gain a bit from the $|z-z'|$-decorrelation effect, 
	so it suffices to prove \eqref{truth} only for $|z-z'| \gg n^{-\gamma}$ and for $ \eta = n^{-3/4-\alpha}$ with some small constants $\alpha,\gamma>0$; this is 
	 Proposition \ref{lemma_new_1} for the small eigenvalues.  
	 Third, the $|z-z'|$-decorrelation effect is also used to prove the independence of  $\langle G^{z}(\ii\eta)\rangle$
	  and $\langle G^{z'}(\ii\eta)\rangle$,
	  at least when $|z-z'|\ge n^{-\gamma}$ and $\eta = n^{-3/4-\alpha}$
	    (Proposition \ref{lemma_new_2}). 
	    Note that both Proposition \ref{lemma_new_1}
	  and Proposition \ref{lemma_new_2} are proven by analyzing weakly correlated DBMs;
	   pure resolvent methods are not sufficient for their proof.
	  Finally, the $|z-z'|$-decorrelation effect is 
	  also used in the improved estimates for 
	  various products of the resolvents on the level $\eta_0=n^{-1+\epsilon}$~(\eg (\ref{estimate7})-(\ref{estimate6})) that come up in the GFT argument; here we use that the 
	   stability factor of some self-consistent equation for the resolvent products in \eqref{stable_0} behaves nicely when $|z-z'| \gg n^{-\gamma}$.

\item {\bf The exponential factor} $e^{-n\delta^2/2}$ in (\ref{tail1}) 
expressing the effect that  $z$ is outside the unit disk. 
Note that $\gamma_n$ hence $L_n$ in \eqref{omega_2}
are chosen exactly such that this exponential factor  compensates
the volume factor loss in $\int |\Delta f|$. This exponential factor was first
used only for the Ginibre ensemble in \cite{maxRe} to study the precise location of $\max\Re\mathrm{Spec}(X)$;
 here we need it for generic i.i.d. matrices as well which will be proved using another GFT argument.

	\item {\bf The level repulsion} between the two smallest eigenvalues
	 $\lambda^z_1$ and $\lambda^z_{-1}=-\lambda_1^z$,
	 \ie $y^2$ factor in the tail estimate in 
	(\ref{tail1}).  This level repulsion factor was first
	 used in \cite{CES21} 
	 to prove the edge universality and
	 we will use it here, too.

	\end{enumerate}

Before we turn to the actual proof of the main Theorem~\ref{main}, we outline its 
 three main steps and  we give some heuristic ideas about them. These three steps will be precisely stated and proved in the following Section~\ref{sec:step1} through Section ~\ref{sec:GFT}, respectively. 

\begin{enumerate}

    \item[Step 1.]  This step analyses the Ginibre ensemble and proves bounds on the variances
    of the small and large $\eta$-regimes. We first prove  (in Lemma \ref{lemma_step1} below) that the small $\eta$-integral satisfies
    \begin{align}\label{eq_some}
    	I_{0}^{\eta_0}(f)
    	=&\frac{n\eta_0^2}{2\wt\eta}\int_{\C} \Delta f(z) \<G^{z}(\ii \wt\eta)\> \dd^2 z+\mathcal{E}, 
	\qquad \E|\mathcal{E}|^2=o(1).
    \end{align}
    where $\eta_0=n^{-1+\epsilon}$, $\wt\eta=n^{-3/4-\alpha}$ with sufficiently small $\epsilon,\alpha$ and $\alpha>\epsilon$, 
    and the error term $o(1)$ holds true in the second moment sense. 
    
    For a brief sketch of the proof of (\ref{eq_some}): in the spectral decomposition of the resolvent
    we split the eigenvalues of $H^z$ into two parts truncated at the level $\wt \eta=n^{-3/4-\alpha}$. Firstly, using the tail bound for the smallest (in absolute value) eigenvalues $\lambda_1^z,\lambda_1^{z'}$, 
    with two different $z,z'$ in Proposition \ref{lemma_new_1},
     the small eigenvalues below the level $\wt \eta$ will contribute $o(1)$ in the second moment sense. 
    Here we also exploit the factor $e^{-n\delta^2/2}$ in the tail bound that is already known for Ginibre matrices.
    Then, for the remaining large eigenvalues above $\wt \eta$, we can replace the spectral parameter of the resolvent $\eta$ with the
    much larger level $\wt\eta$ using Proposition \ref{lemma2} for the Ginibre ensemble. In particular, this replacement does not rely on the decorrelation effect coming from $|z-z'|$ being fairly large.

    Once the spectral parameter $\eta$ in $I_{0}^{\eta_0}(f)$ has been increased to $\wt\eta$ in (\ref{eq_some}), 
    we can use Proposition~\ref{lemma_new_2} (manifesting a slight
     $|z-z'|$-decorrelation effect, \ie $\<G^{z}\>$ and $\<G^{z'}\>$ are weakly independent) to show that the right side of (\ref{eq_some}) satisfies
    \begin{align}\label{Evar_3}
    	\V^{\mathrm{Gin}} \Big[\frac{n\eta_0^2}{2\wt\eta}\int_{\C} \Delta f(z)  \<G^z(\ii \wt\eta)\> \dd^2 z \Big]=o(1),
    \end{align}
as stated more precisely in Lemma \ref{lemma_step12}.
Therefore, using (\ref{eq_some}) and (\ref{Evar_3}), we obtain that
    \begin{align}
    	&\V^{\mathrm{Gin}} \Big[I_{0}^{\eta_0}(f)\Big]=o(1).
	\label{Evar_22}
    \end{align} 

 With the variance bound of $I_{0}^{\eta_0}(f)$ in~(\ref{Evar_22}) for $f=f^{-}_1$, we use 
 Girko's formula in (\ref{expectation_eq}) 'backward' together with
  (\ref{main_part_2_Gin})~(the variance estimate of the left side of (\ref{expectation_eq}))
   to get corresponding variance bound for the other part $I_{\eta_0}^{T}(f^{-}_1)$, \ie
	\begin{align}\label{some_Gini_2}
	\V^{\mathrm{Gin}} \Big[I_{\eta_0}^{T}(f^{-}_1) \Big]=
	o(1)\Big(\E^{\mathrm{Gin}}\Big[\sum_{i=1}^n f^{-}_1(\sigma_i)\Big]\Big)^2,
\end{align}
where $\E^{\mathrm{Gin}}\Big[\sum_{i=1}^n f^{-}_1(\sigma_i)\Big] \geq C$ 
with some constant $C>0$ given in (\ref{main_part_1_Gin}).

\item[Step 2.]	In this step we transfer information from Ginibre matrices to generic i.i.d. matrices
 in the small $\eta$-regime.
We use a similar argument~(Lemma~\ref{lemma_step2}) as in Step 1 to show that, for $f= f_1^-$ or~$f_2^+$,
       \begin{align}\label{step_2_G}
	I_{0}^{\eta_0}(f) = &\frac{n\eta_0^2}{2\wt\eta}\int_{\C} \Delta f(z) \<G^{z}(\ii \wt\eta)\>\dd^2 z+
	\mathcal{E}, \qquad \E|\mathcal{E}|=o(1),
\end{align} 
where $\eta_0=n^{-1+\epsilon}$, $\wt\eta=n^{-3/4-\alpha}$ with sufficiently small $\epsilon,\alpha$ and 
$\alpha>\epsilon$. Notice
the difference compared with~\eqref{eq_some}:
 now the error term is controlled only in the first absolute moment sense. 
  The extension of 
 \eqref{step_2_G} to the  second moment sense, as available for Ginibre in~\eqref{eq_some},
would  require much more effort, in particular, we would need to prove the analogue of Proposition \ref{lemma_new_1} for i.i.d. matrices, 
but our proof circumvents this.
 To prove~\eqref{step_2_G}, the key step is to use a GFT and Gronwall argument~\cite{EX22} to  transfer
 the tail bound of the smallest $\lambda^z$ with the  exponential factor $e^{-n\delta^2/2}$  from Ginibre to  i.i.d. matrices
 (Proposition \ref{prop1}).
 
Since the spectral parameter of the Green function  has been increased from 
$\eta_0=n^{-1+\epsilon}$  to $\wt\eta=n^{-3/4-\alpha}$ thanks to (\ref{step_2_G}),
we can use a standard  iterative GFT argument to compare
$\E[ \<G^{z}(\ii \wt\eta)\>]$ with its Ginibre counterpart. This is done in Proposition~\ref{prop2_old} which is stated more generally for  any $\eta \ge n^{-1+\epsilon}$ and its proof requires
$1/\epsilon$-many steps in the bootstrap, but for the current application we need it only for $\wt\eta=n^{-3/4-\alpha}$
with a few bootstrap steps since $1/(n\wt\eta)$, the gain in 
each step, is relatively small.  
Using this GFT comparison we get 
 \begin{align}\label{step_2_ex}
\E \Big[I_{0}^{\eta_0}(f)\Big] = &\frac{n\eta_0^2}{2\wt\eta}\int_{\C} \Delta f(z) \E\Big[\<G^{z}(\ii \wt\eta)\>\Big] \dd^2 z+o(1)\nonumber\\
 = &\frac{n\eta_0^2}{2\wt\eta}\int_{\C} \Delta f(z) \E^{\mathrm{Gin}}\Big[\<G^{z}(\ii \wt\eta)\>\Big] \dd^2 z+o(1)=\E^{\mathrm{Gin}} \Big[I_{0}^{\eta_0}(f)\Big]+o(1),
\end{align} 
where 
in the last step we used (\ref{eq_some}) or \eqref{step_2_G}. This is presented as the first statement of Lemma~\ref{lemma_step2_2} in a more precise way.

To prove the concentration result of the small $\eta$-component
  in (\ref{main_part_2}), we need to estimate the first absolute moment of $I^{\eta_0}_{0}-\E[I^{\eta_0}_{0}]$. Using \eqref{step_2_G} we have
	\begin{align}\label{step2_middle1}
	\E\Big|I^{\eta_0}_{0}(f)-\E[I^{\eta_0}_{0}(f)]\Big|=&\frac{n\eta_0^2}{2\wt\eta} \E\Big|(1-\E)\int_{\C} \Delta f(z) \<G^{z}(\ii \wt\eta)\>\dd^2 z\Big|+o(1)\nonumber\\
	\leq &\frac{n\eta_0^2}{2\wt\eta} \sqrt{\V \Big[\int_{\C} \Delta f(z) \<G^{z}(\ii \wt\eta)\>\dd^2 z \Big]} +o(1)  \nonumber\\
	=& \frac{n\eta_0^2}{2\wt\eta} \sqrt{\V^{\mathrm{Gin}} \Big[\int_{\C} \Delta f(z) \<G^{z}(\ii \wt\eta)\>\dd^2 z \Big]+o(1)}+o(1), \end{align}
where we used Cauchy-Schwarz inequality in the second line and
in the last line  we used standard  iterative GFT  for $\V\big[ \<G^{z}(\ii \wt\eta)\>\big]$,
 see Proposition~\ref{prop2_old},  (again, only for  $\wt\eta=n^{-3/4-\alpha}$).
 Finally, using (\ref{Evar_3}) we have
 \begin{align}\label{step2_middle}
	\E\Big|I^{\eta_0}_{0}(f)-\E[I^{\eta_0}_{0}(f)]\Big|=o(1);
\end{align}	
 see the second statement of Lemma \ref{lemma_step2_2} for a more precise form.

	\item[Step 3.] In this last step we consider the remaing large $\eta$-integral $I_{\eta_0}^{T}$ for in i.i.d.~ cases and
	 show (Proposition \ref{gft}) that, for $f= f_1^-$ or $f=f_2^+$,
	\begin{align}
	&\E\Big[I_{\eta_0}^{T}(f)\Big] = \E^{\mathrm{Gin}}\Big[I_{\eta_0}^{T}(f)\Big]+o(1),\label{genvar0}\\
	 &\V\Big[I_{\eta_0}^{T}(f)\Big] = \V^{\mathrm{Gin}}\Big[I_{\eta_0}^{T}(f)\Big]+o(1). \label{genvar}
	\end{align} 
This is a more delicate iterative GFT than the ones used in previous papers, e.g. \cite[Proposition 3.8]{maxRe}
or in Proposition~\ref{prop2_old} below, not only because 
we are  operating down to the most involved $\eta_0=n^{-1+\epsilon}$ level, but,
more critically,  also because the large $\eta$-regime is sensitive to larger eigenvalues. In this typical part of
the spectrum  the earlier exponential factor $e^{-n\delta^2/2}$ is not
present, so the compensation for the $\int |\Delta f|\sim n^{1/2}$ loss  has to come 
from a very precise comparison between Ginibre and i.i.d. ensembles.

Compared to the previous GFTs done in \cite[Proposition 3.8]{maxRe}, behind
the proof of  \eqref{genvar0}--\eqref{genvar} in Proposition \ref{gft} there are mainly two new refinements (see Proposition \ref{key_lemma}): 
\begin{itemize}
\item  for the third order terms with distinct summation indices, we need to iteratively expand these terms up to a sufficient precision using the so-called unmatched index~(see Definition~\ref{def:unmatch_form}), and explicitly identify the remaining leading deterministic terms, which will all vanish against $\Delta f(z)$ after the $z$-integrations. Here we really need $1/\epsilon$ iteration steps, in contrast to fewer steps needed in Step 2.
\item for the restricted third order terms with index coincidence as well as all the fourth order terms, again we identify the precise leading terms that vanish after the $z$-integrations with improved error terms; here we need to gain an extra smallness from the $|z-z'|$-decorrelation effect for the improved error terms. 
\end{itemize}
\end{enumerate}

Armed with the results of these three steps 
we are now ready to prove Theorem~\ref{main}.
\begin{proof}[Proof of Theorem~\ref{main}]
As in~\cite[Section 3]{maxRe}, the proof of Theorem~\ref{main} has been reduced to showing (\ref{main_part_1})-(\ref{main_part_2}). Firstly using (\ref{expectation_eq}), (\ref{step_2_ex}) and (\ref{genvar0}) for the expectation estimate, we have
\begin{align}\label{expectation_f}
	\E\Big[\sum_{i=1}^n f(\sigma_i)\Big]=&\E[I_{\eta_0}^{T}(f)]+\E[I_{0}^{\eta_0}(f)]+O_\prec(n^{-100})\nonumber\\
=&\E^{\mathrm{Gin}}[I_{\eta_0}^{T}(f)]+\E^{\mathrm{Gin}}[I_{0}^{\eta_0}(f)]+o(1)=\E^{\mathrm{Gin}}\Big[\sum_{i=1}^n f(\sigma_i)\Big]+o(1).
\end{align}
The expectation estimates in (\ref{main_part_1}) then follows directly from the Ginibre estimates in \eqref{main_part_1_Gin}.

We next prove the the concentration estimate in (\ref{main_part_2}). Recalling  the decomposition of the linear statistics in (\ref{expectation_eq}) and using (\ref{step2_middle}) for $f= f_1^-$, we have
\begin{align}\label{mainterm}
	\Big|(1-\E)\Big[\sum_{i=1}^n f (\sigma_i)\Big]\Big|=
	&\E\Big|I_{\eta_0}^{T}(f )-\E[I_{\eta_0}^{T}(f )]\Big| +o(1) \leq \sqrt{\V[I_{\eta_0}^{T}(f_1^{-} )]}+o(1). 
\end{align} 
Using the GFT result in (\ref{genvar}) for $f=f_1^-$ and (\ref{some_Gini_2}) for the Ginibre ensemble, we have
\begin{align}\label{var_f}
	\Big|(1-\E)\Big[\sum_{i=1}^n f^{-}_1(\sigma_i)\Big]\Big| \leq&  \sqrt{\V^{\mathrm{Gin}}[I_{\eta_0}^{T}(f^{-}_1)]+o(1)}+o(1)\nonumber\\
	=&o(1)\Big(\E^{\mathrm{Gin}}\Big[\sum_{i=1}^n f^{-}_1(\sigma_i)\Big]\Big)=o(1)\Big(\E\Big[\sum_{i=1}^n f^{-}_1(\sigma_i)\Big]\Big),
\end{align}
where we also used (\ref{expectation_f}) and that $\E\Big[\sum_{i=1}^n f^{-}_1(\sigma_i)\Big] \geq C$ for some constant $C>0$. Hence we completed the proof of Theorem~\ref{main}.
	
\end{proof}

\bigskip
 
We conclude this sketch of the proof strategy by commenting on the relation to our newer paper~\cite{Gumbel}
where the Gumbel universality was proven more than a year after the completion of the current work.
Both proofs start with Girko's formula with specifically designed cutoff
functions $f^\pm$ locating the largest eigenvalue, 
but the actual analysis is conceptually different in the two papers.
First, in~\cite{Gumbel} we did not split
the $\eta$-integration into two regimes at some $\eta_0$ as in~\eqref{expectation_eq},
instead, we effectively separated the microscopic and mesoscopic  regimes by monitoring the
smallest singular value $\lambda_1^z$ and using it as an effective small-scale cutoff in 
the typical regime where $\lambda_1^z\ge n^{-3/4-\epsilon}$. However, this event needed to be controlled
dynamically along the Ornstein-Uhlenbeck  flow similar to~\eqref{flow}, adding an extra complication.
  Second,  in~\cite{Gumbel} we
performed two integrations by parts (in $z$) in~\eqref{linear_stat}. This moved the complication 
due to $\Delta_z f$ from the test function to the resolvent. For the mesoscopic regime
it required to analyse the local law for $\Delta_z G^z = G^z F G^z F^* G^z + \ldots $ with 
$F = \begin{pmatrix} 0 & 0 \cr 1 & 0 \end{pmatrix}$ instead of merely using the single resolvent local law, Theorem~\ref{local_thmw}.
This considerably simplified the iterative GFT analysis (which comprises large part of the current work)
at the cost of proving more sophisticated multi-resolvent local laws where the off-diagonality of $F$
is exploited. Third, to control  $\Delta_z G^z$ in the microscopic regime  we needed the decorrelation 
estimate~\eqref{truth} in the entire range $|z-z'|\gg n^{-1/2}$. This was done
with DBM methods analogous to Section~\ref{sec:DBM}, but it  needed the essential
independence of $G^z$ and $G^{z'}$  for any $|z-z'|\gg n^{-1/2}$.  In turn, such independence
required a local law for $\langle G^z G^{z'}\rangle$ with optimal $|z-z'|$-decay which needed a 
new strategy in~\cite{Gumbel}, the {\it characteristic flow method}. Finally, while in the current paper we estimate
pieces of Girko's formula either in expectation or variance sense; to identify Gumbel distribution we needed
general test functions in~\cite{Gumbel}. It is fair to say that the approach in the current paper is more 
elementary, heavily relying on the robust GFT proof strategy, while~\cite{Gumbel} is more sophisticated
using additional ideas that are specific to this particular problem.  We also stress that several key results
from the current paper, such as the local law Theorem~\ref{local_thmw} and the DBM analysis in Section \ref{sec:DBM}, are directly 
used in~\cite{Gumbel}.

\section{Step 1. Ginibre ensemble: Small $\eta$ integral over $[0,n^{-1+\epsilon}]$}\label{sec:step1}

In this section, we consider the Ginibre ensemble and aim to replace the small spectral parameter 
$\eta\in[0,n^{-1+\epsilon}]$ of the resolvent in $I_{0}^{\eta_0}(f)$
with a large level $\wt \eta$ slightly below $n^{-3/4}$.
Recall that throughout the paper $f= f_1^-$ or $f=f_2^+$ given in (\ref{function1})-(\ref{function2}).
\begin{lemma}\label{lemma_step1}
 	Fix $\tau>0$. Set $\eta_0=n^{-1+\epsilon}$ and $\wt \eta=n^{-3/4-\alpha}$ with sufficiently small $\epsilon,\alpha$ and $\alpha >\epsilon>\tau/2$. Then we have 
	    \begin{align}\label{step_1_goal}
	\E^{\mathrm{Gin}}\Big|	I_{0}^{\eta_0}(f)-\int_{\C} \Delta f(z) \int^{\eta_0}_{0} \sum_{i}\frac{\eta}{(\lambda^z_i)^2+\wt \eta^2} \dd \eta \dd^2 z\Big|^2=O\Big((\log n)^C n^{-4(\alpha-\epsilon)}\Big).
	\end{align}
\end{lemma}

Before we give the proof, we recall from~\cite{maxRe} the following precise tail bound estimate 
for the smallest eigenvalue (in modulus) of $H^{z}$ from (\ref{initial}), with $X$ being the complex Ginibre ensemble.
\begin{proposition}[Proposition 2.7 \cite{maxRe}]\label{prop:tail}
	Fix\footnote{Here we use a different convention compared to~\cite{CES20,CES22a,CES22b}, i.e. we now define $\delta$ so that $\delta>0$ for $|z|>1$.} $\delta:=|z|^2-1$ with $n^{-1/2}\ll\delta\ll 1$  and let  $\lambda_1^z$ be the smallest singular value of $X-z$, 
	where $X$ is a  complex Ginibre matrix.
	Then there exists a constant $C>0$, independent of $n$ and $\delta$, such that for any $y\le C/(n\delta^2)$
	we have the following lower tail bound 
	\begin{equation}\label{tail}
		\mathbf{P}^{\mathrm{Gin}}\left(\lambda_1^z\le y\delta^{3/2}\right)\lesssim y^2 (n\delta^2)^{4/3} e^{-n\delta^2(1+\OO(\delta))/2}.
	\end{equation}
\end{proposition}
Recall the functions $f^{-}_1$ and $f^{+}_2$ given in (\ref{function1})-(\ref{function2}). 
For any $z \in \mathrm{supp}(f^{-}_1) \cup \mathrm{supp}(f^{+}_2)$, we have 
\begin{align}\label{delta_regime}
 \sqrt{\frac{\log n}{n}} (1-o(1)) \leq \delta=|z|^2-1 \lesssim \frac{n^{\tau}}{\sqrt{n}},
\end{align} which implies that  $e^{-n\delta^2/2} \lesssim n^{-1/2}$. Fixing any small $\epsilon'>\tau/2$,  from Proposition \ref{prop:tail} we have
\begin{align}\label{tail_bound}
	\mathbf{P}^{\mathrm{Gin}}\left(\lambda_1^z\le \eta \right)\lesssim n \eta^2, \quad\quad \mbox{for any}\;\;
	 \eta\leq n^{-3/4-\epsilon'} \;\; \mbox{and} \;\;  z \in \mathrm{supp}(f^{-}_1) \cup \mathrm{supp}(f^{+}_2).
\end{align} 
Furthermore, by (\ref{tail_bound}), we have the following estimates for the resolvent (see Proposition \ref{lemma2} below for the precise statement)
\begin{align}\label{moment_resolvent}
	\E^{\mathrm{Gin}} \Big[\big( \Im \<G^{z}(\ii \eta)\> \big)^k \Big]=O_\prec\Big((\sqrt{n}\eta)^k+\frac{n^{-1}}{(n\eta)^{k-2}}\Big), \qquad \forall k\geq 1.
\end{align}
This proposition will be formulated and proven directly for general i.i.d. matrices in the next section  relying only on the tail bound for $\lambda_1^z$ 
and rigidity of eigenvalues near the origin. Its proof does not use  any other  comparison
with Ginibre ensemble, hence the argument is not circular.

 We remark that for $z\in \mathrm{supp}(f^{+}_2)$, the above estimates should be smaller 
than for  $z\in \mathrm{supp}(f^{-}_1)$ 
since the support of $f^+_2$ is even farther away from from  the unit disk, see
(\ref{function2}).
 Nevertheless our proof will not rely on this improvement at all. Hence we will not distiguish between $f_1^-$ and $f_2^+$ and the following estimates are valid for any $f=f_1^-$ or $f_2^+$.

Now we are ready to prove Lemma \ref{lemma_step1}.  
\begin{proof}[Proof of Lemma \ref{lemma_step1}]
	Recall that the eigenvalues of $H^{z}$ in (\ref{initial}) are given by $\{\pm \lambda_j^z\}_{j \in \llbracket 1,n \rrbracket}$ 
	with non-decreasingly labelled $\lambda_j^z \in \R^+$ for $j \in \llbracket 1,n \rrbracket$.  
 Using similar arguments as in the proof of \cite[Lemma 3.4]{maxRe} (see also \cite[Eq. (5.34)-(5.35)]{AEK18a}), there exists a suffciently large $l>0$ such that the tiny $\eta$-integral over $[0,n^{-l})$ is negligible in the sense of any finite moment, \ie 
\begin{align}\label{tiny_int}
	\int_{0}^{n^{-l}}\Im \Tr G^{z}(\ii \eta) \dd \eta =\frac{1}{2} \Big( \sum_{|\lambda_i| \lesssim n^{-l}} +\sum_{|\lambda_i| \gtrsim n^{-l}} \Big) \log \Big(1+\frac{n^{-2l}}{(\lambda^z_j)^2}\Big)=O(n^{-10}),
\end{align}
where $O(n^{-10})$ is an irrelevant error term small enough for our proof. Note that the above also holds true for any i.i.d. matrix $X$ satisfying Assumption \ref{ass:mainass} with a sufficiently large $l>0$ depending on $\alpha,\beta$ in \eqref{assumption_b}.  Combining (\ref{tiny_int}) with (\ref{deltaf}), we have 
$$I_{0}^{\eta_0}(f)=\int_{\C} \Delta f(z)  \int_{n^{-l}}^{\eta_0}\Im \Tr G^{z}(\ii \eta) \dd \eta \dd^2 z+O(n^{-8}),$$
where the last error term holds true in the finite moment sense.  In the spectral decomposition of  $\Im \<G^{z}(\ii \eta)\>$, we then separate the eigenvalues into two parts with the cut-off level at $\wt \eta=n^{-3/4-\alpha}$,
\begin{align}\label{small+big}
	I_{0}^{\eta_0}(f)
	&=\int_{\C} \Delta f(z)  \int_{n^{-l}}^{\eta_0} \Big( \sum_{\lambda^z_j \leq \wt \eta} +\sum_{\lambda_j^z \geq \wt \eta}\Big)\frac{\eta}{(\lambda^z_j)^2+\eta^2} \dd \eta \dd^2 z+O(n^{-8})\nonumber\\
	&=:I_{small}+I_{big}+O(n^{-8}).
\end{align}

We first estimate the second moment of $I_{small}$ in (\ref{small+big}) by splitting the $z$-integrals into two parts, \ie
\begin{align}\label{E_1+E_2}
	\E^{\mathrm{Gin}}|I_{small}|^2=&\Big( \iint_{|z-z'|\leq n^{-\gamma}} +\iint_{|z-z'|\geq n^{-\gamma}} \Big) \Delta f(z)\Delta f(z') \times\nonumber\\
	&\qquad \qquad\qquad\E^{\mathrm{Gin}} \Big[ \int_{n^{-l}}^{\eta_0} \sum_{\lambda^{z}_j \leq \wt \eta}  \frac{\eta \dd \eta }{(\lambda^{z}_j)^2+\eta^2} \int_{n^{-l}}^{\eta_0} \sum_{\lambda^{z'}_j \leq \wt \eta} \frac{\eta' \dd \eta'}{(\lambda^{z'}_j)^2+\eta'^2} \Big]\dd^2 z \dd^2 z'\nonumber\\
	=:&E_{|z-z'|\leq n^{-\gamma}}+E_{|z-z'|\geq n^{-\gamma}},
\end{align}
with a small $\gamma>0$ to be chosen. For the regime $|z-z'|\leq n^{-\gamma}$ in (\ref{E_1+E_2}), using the Cauchy-Schwarz inequality we obtain
\begin{align}\label{eq_step1}
	E_{|z-z'|\leq n^{-\gamma}} \lesssim  \eta_0 \int |\Delta f(z')| \dd^2 z' \int_{|z-z'| \leq n^{-\gamma}} |\Delta f(z)|  \int_{n^{-l}}^{\eta_0}	\E^{\mathrm{Gin}}\Big[	\Big(\sum_{\lambda^z_j \in [0,\wt \eta]}\frac{\eta}{(\lambda^z_j)^2+\eta^2} \Big)^2 \Big]\dd \eta \dd^2 z.
\end{align}
We divide the interval $[0,\wt \eta]$ into triadic partitions $\bigcup_{k=0}^{O(\log n)} [3^{k-1}\eta, 3^{k}\eta]$ for any $n^{-l} \leq \eta \leq \eta_0$.  With a slight abuse of notation, 
we define the first interval for  $k=0$ to be $[3^{k-1}\eta, 3^{k}\eta]: = [0, \eta]$. 
From the rigidity estimates of eigenvalues in Corollary \ref{cor:rigidity}, fixing a small $\alpha>0$,  for any  $-Cn^{-1/2} \leq |z|-1\le c$, we have
\begin{align}\label{rigidity0}
	\#\big\{ j\in \llbracket 1, n \rrbracket: \lambda_j^z \leq n^{-3/4-\alpha} \big\} \leq n^{\xi},
\end{align}
for any small $\xi>0$, with a very high probability.  Thus using (\ref{rigidity0}) with $0<\xi<\epsilon$ we obtain that
\begin{align}\label{dy}
	\E^{\mathrm{Gin}}\Big[\Big(\sum_{\lambda^z_j \leq \wt \eta}&\frac{\eta}{(\lambda^z_j)^2+\eta^2} \Big)^2\Big]
	\lesssim \log n \sum_{k=0}^{O(\log n)}\E^{\mathrm{Gin}}
	\Big[ \sum_{\lambda^z_j \in [3^{k-1}\eta, 3^{k}\eta] }\frac{\eta^2}{((\lambda^z_j)^2+\eta^2)^2}\Big]\nonumber\\
	\lesssim& n^{\xi} \log n \sum_{k=0}^{O(\log n)} \frac{\eta^2}{(3^{k-1}\eta)^4} \P^{\mathrm{Gin}} (\lambda^z_1 \leq 3^{k}\eta)	
	\lesssim  (\log n)^C n^{1+\epsilon},
\end{align}
where we used the tail bound in (\ref{tail_bound}). Plugging (\ref{dy}) in (\ref{eq_step1}) in combination with the 
norm bounds of $\Delta f$ in (\ref{deltaf}), we gain a small factor $n^{-\gamma}$ from the volume of the $z$-integration, \ie 
\begin{align}\label{large_agg}
	E_{|z-z'|\leq n^{-\gamma}} \lesssim (\log n)^C n^{1-\gamma}  n^{1+\epsilon} \eta_0^2 = (\log n)^C n^{-\gamma+3\epsilon},
\end{align}
which is negligible choosing $\gamma>3\epsilon$.

\bigskip

We next consider the complementary regime $|z-z'|\geq n^{-\gamma}$ in (\ref{E_1+E_2}), where
we will  gain a small factor from the $|z-z'|$-decorrelation effect using the following proposition
whose proof is deferred to Section~\ref{sec:DBM}.
\begin{proposition}\label{lemma_new_1}
Fix any sufficiently small $\gamma,\alpha,\tau>0$, then for any $z,z'$ with $-n^{-1/2+\tau}\le |z|^2-1\le n^{-1/2+\tau}$, $|z-z'|\ge n^{-\gamma}$ it holds
	\begin{align}\label{lambdatail}
		\P^{\mathrm{Gin}}\Big( |\lambda_1^z| \leq n^{-3/4-\alpha}, ~ |\lambda_1^{z'}| \leq n^{-3/4-\alpha}\Big) 
		\lesssim \Big[\P^{\mathrm{Gin}}\big( |\lambda_1^z| \leq 10n^{-3/4-\alpha}\big)\Big]^2+ n^{-100}.
	\end{align}	
\end{proposition}
We comment that the last error term $n^{-100}$ is irrelevant but enough to proceed with our proof. 
The same irrelevant error term will also be used in Proposition~\ref{prop1}--\ref{lemma2} later
 and more explanations can be found below these propositions.

 Then, using (\ref{lambdatail}) together with the tail bound (\ref{tail_bound}) for a single $z$, the rigidity estimate in (\ref{rigidity0}), and the $L^1$-norm bound from~\eqref{deltaf}, we have
\begin{align}\label{small_agg}
	E_{|z-z'|\geq n^{-\gamma}} \lesssim &\iint_{|z-z'|\geq n^{-\gamma}} |\Delta f(z)||\Delta f(z')| \Big(\int_{n^{-l}}^{\eta_0} \frac{\dd \eta}{\eta}\Big)^2 n^{2\xi}  \P^{\mathrm{Gin}}\Big( |\lambda_1^z| \leq \wt \eta, ~ |\lambda_1^{z'}| \leq \wt \eta\Big) \dd^2 z \dd^2 z' \nonumber\\
	\lesssim& (\log n)^C n^{1+2\xi}\Big(\P^{\mathrm{Gin}}\big( |\lambda_1^z| \leq \wt \eta\big)\Big)^2 \lesssim (\log n)^C
	 n^{-4\alpha+2\xi},
\end{align}
for any small $\xi>0$. Therefore, combining (\ref{large_agg}) and (\ref{small_agg}), we conclude that
\begin{align}\label{step_1_small}
	\E^{\mathrm{Gin}}|I_{small}|^2=O\Big((\log n)^C \big(n^{-4\alpha+2\xi}+n^{-\gamma+3\epsilon}\big) \Big).
\end{align}

Next we estimate $I_{big}$ in (\ref{small+big}). By \eqref{step_1_small}, to prove \eqref{lemma_step1} it is sufficient to compute
\begin{align}\label{step_1_split}
	\E^{\mathrm{Gin}}\Big|	I_{big}-\int_{\C} \Delta f(z)& \int^{\eta_0}_{0} \sum_{i}\frac{\eta}{(\lambda^z_i)^2+\wt \eta^2} \dd \eta \dd^2 z\Big|^2 \leq \E^{\mathrm{Gin}}\Big| \int_{\C} \Delta f(z) \int_{0}^{n^{-l}} \sum_{i}\frac{\eta}{(\lambda^z_i)^2+\wt\eta^2} \dd \eta \dd^2 z\Big|^2
	\nonumber\\
	&+\E^{\mathrm{Gin}}\Big|\int_{\C} \Delta f(z) \int^{\eta_0}_{n^{-l}} 	\Big(\sum_{|\lambda_i^{z}| \leq \wt \eta} \frac{\eta}{(\lambda^z_i)^2+\wt \eta^2} \Big) \dd \eta \dd^2 z\Big|^2 \nonumber\\
	&+\E^{\mathrm{Gin}}\Big|\int_{\C} \Delta f(z) \int^{\eta_0}_{n^{-l}}  \sum_{|\lambda_i^{z}| \geq \wt \eta}  \frac{ \eta(\wt\eta^2-\eta^2)}{((\lambda^z_i)^2+\wt\eta^2)((\lambda^z_i)^2+\eta^2)} \dd \eta \dd^2 z \Big|^2,
\end{align}
with a sufficiently large $l>0$ chosen as in (\ref{tiny_int}) and  $\wt \eta=n^{-3/4-\alpha}$. Notice that the first part in (\ref{step_1_split}) with the tiny $\eta$ integral over $[0,n^{-l})$ is
 negligible for sufficiently large $l$. In addition, the second part in (\ref{step_1_split}) from small eigenvalues $|\lambda_i^{z}| \leq \wt \eta$ can be estimated using the tail bound in (\ref{tail_bound}) with an even better error term $O\big((\log n)^C n^{-\frac{1}{2}+4\epsilon+2\alpha}\big)$ since $\wt\eta=n^{-3/4-\alpha}$ is already large. So we only focus on the third part, i.e. we now consider regime $|\lambda_i^{z}| \geq \wt \eta$ in (\ref{step_1_split}). Recalling that $\eta_0=n^{-1+\epsilon} \ll \wt \eta=n^{-3/4-\alpha}$ with sufficiently small $\epsilon,\alpha$, we have
\begin{align}
	&\E^{\mathrm{Gin}}\Big|\int_{\C} \Delta f(z) \int^{\eta_0}_{n^{-l}}  \sum_{|\lambda_i^{z}| \geq \wt \eta}  \frac{ \eta(\wt\eta^2-\eta^2)}{((\lambda^z_i)^2+\wt\eta^2)((\lambda^z_i)^2+\eta^2)} \dd \eta \dd^2 z \Big|^2\nonumber\\
	\lesssim & (n\eta_0)^4  \iint |\Delta f(z)||\Delta f(z')|   \E^{\mathrm{Gin}}\Big[\big(\Im \< G^{z}(\ii \wt \eta) \> \big)^2\big(\Im \< G^{z'}(\ii \wt \eta)\>\big)^2\Big]\dd^2 z\dd^2 z' \nonumber\\ 
	\lesssim& (n\eta_0)^4 \|\Delta f\|_1^2   \E^{\mathrm{Gin}}\Big[\big(\Im \< G^{z}(\ii \wt \eta)\>\big)^4\Big] \lesssim (\log n)^C n^{-4\alpha+4\epsilon},
\end{align}
where in the last line we used the Cauchy-Schwarz inequality, the norm bound~\eqref{deltaf}, and
 the estimate for the resolvent in (\ref{moment_resolvent}) for $k=4$. Thus from (\ref{step_1_split}) we obtain that
\begin{align}\label{step_1_big}
	\E^{\mathrm{Gin}}\Big|	I_{big}-\int_{\C} \Delta f(z) \int^{\eta_0}_{0} \sum_{i}\frac{\eta}{(\lambda^z_i)^2+\wt\eta^2} \dd \eta \dd^2 z\Big|^2=O\big((\log n)^C n^{-4\alpha+4\epsilon}\big).
\end{align}
Combining (\ref{step_1_small}) and (\ref{step_1_big}), we thus conclude
 $$\E^{\mathrm{Gin}}\Big|	I-\int_{\C} \Delta f(z) \int^{\eta_0}_{0} \sum_{i}\frac{\eta}{(\lambda^z_i)^2+\wt\eta^2} \dd \eta \dd^2 z\Big|^2=O\Big((\log n)^C \big(n^{-4\alpha+2\xi}+n^{-\gamma+3\epsilon}+n^{-4\alpha+4\epsilon}\big) \Big).$$
 Choosing sufficiently small $\xi,\gamma,\alpha$ with $\xi<\epsilon$, $\gamma>4\alpha$, and $\alpha>\epsilon$, we have finished the proof of Lemma~\ref{lemma_step1}.
\end{proof}

Next, we will use Lemma \ref{lemma_step1} to show the following.
\begin{lemma}\label{lemma_step12}
	For the Ginibre ensemble	 we have
    \begin{align}
	  	&\E^{\mathrm{Gin}} \Big[I_{0}^{\eta_0}(f)\Big]=\frac{n\eta_0^2}{2\wt \eta}\int_{\C} \Delta f(z) \E^{\mathrm{Gin}} \Big[\<G^{z}(\ii \wt\eta)\> \Big]\dd^2 z+O(n^{-\kappa}),\label{eq_step11}\\
	&\V^{\mathrm{Gin}} \Big[I_{0}^{\eta_0}(f)\Big]\lesssim \V^{\mathrm{Gin}} \Big[\frac{n\eta_0^2}{2\wt \eta}\int_{\C} \Delta f(z) \<G^{z}(\ii \wt\eta)\> \Big]\dd^2 z+O(n^{-\kappa})=O(n^{-\kappa}),\label{eq_step12}
\end{align} 
recalling that $\eta_0=n^{-1+\epsilon}$, $\wt\eta=n^{-3/4 -\alpha}$,
 with any sufficiently small $\epsilon,\alpha$ and $\alpha>\epsilon$.  Here $\kappa>0$ is
 a small positive number, depending on $\alpha, \epsilon$, whose precise value is irrelevant.
\end{lemma}

\begin{proof}[Proof of Lemma \ref{lemma_step12}]
	The expectation estimate in (\ref{eq_step11}) follows  from (\ref{step_1_goal}) by a direct computation. Similarly, the first inequality in the variance estimate (\ref{eq_step12}) follows from (\ref{step_1_goal}) and Cauchy-Schwarz inequality. So we focus on proving the second estimate in (\ref{eq_step12}). It then suffices to study 
\begin{align}\label{V_1+V_2}
\V^{\mathrm{Gin}} \Big[\frac{n\eta_0^2}{2\wt \eta}\int_{\C} \Delta f(z) \<G^{z}(\ii \wt\eta)\> \Big]&\dd^2 z	=\frac{n^2\eta_0^4}{4\wt\eta^2} \Big( \iint_{|z-z'|\leq n^{-\gamma}} + \iint_{|z-z'|\geq n^{-\gamma}} \Big)  \Delta f(z)  \Delta f(z')  \times \nonumber\\
	&  \Big( 	\E^{\mathrm{Gin}} \Big[\<G^{z}(\ii \wt\eta)\> \<G^{z'}(\ii \wt\eta)\>\Big] -\E^{\mathrm{Gin}} \Big[\<G^{z}(\ii \wt\eta)\>\Big] \E^{\mathrm{Gin}} \Big[\<G^{z'}(\ii \wt\eta)\>\Big] \Big)  \dd^2 z \dd^2 z'\nonumber\\
	 =:&V_{|z-z'|\leq n^{-\gamma}}+V_{|z-z'|\geq n^{-\gamma}},
\end{align}
with a small $\gamma >0$ to be chosen. For the regime $|z-z'|\leq n^{-\gamma}$ in (\ref{V_1+V_2}),  we gain a little from the volume factor, \ie
\begin{align}\label{V_1}
	V_{|z-z'|\leq n^{-\gamma}}\lesssim& \frac{n^2\eta_0^4}{4\wt\eta^2} \int |\Delta f(z')| \dd^2 z' \int_{|z-z'|\leq n^{-\gamma}} |\Delta f(z)|
	\E^{\mathrm{Gin}} \Big[\<G^{z}(\ii \wt\eta)\>^2\Big] \dd^2 z\lesssim (\log n)^C n^{-\gamma+4\epsilon},
\end{align}
which is negligible choosing $\gamma>4\epsilon$, where we used the Cauchy-Schwarz inequality and the estimate for the resolvent in (\ref{moment_resolvent}) for $k=2$.

Next we estimate $V_{|z-z'|\geq n^{-\gamma}}$ in (\ref{V_1+V_2}). In this regime the additional smallness comes from the $|z-z'|$-decorrelation effect state in the following proposition.	

	\begin{proposition}\label{lemma_new_2}
	For any small $\gamma, \tau>0$ there exists $\widehat{\omega}>0$ such that for any small $\alpha>0$, and for $\wt\eta=n^{-3/4-\alpha}$, $-n^{-1/2+\tau}\le |z|^2-1\le n^{-1/2+\tau}$, $|z-z'| \geq n^{-\gamma}$ it holds 	
		\begin{align}
		\label{eq:almostlast}
			\Bigg|\E \Big[\<G^{z}(\ii \wt\eta)\> \<G^{z'}(\ii \wt\eta)\>\Big] -\E \Big[\<G^{z}(\ii \wt\eta)\>\Big] \E \Big[\<G^{z'}(\ii \wt\eta)\>\Big] \Bigg|\lesssim n^{-1/2-\alpha-\widehat{\omega}/2}.
		\end{align}
	\end{proposition}
	 We will use this proposition only for Ginibre matrices (i.e. with $\E$ being replaced by $\E^{\mathrm{Gin}}$), however we stated it for general i.i.d matrices since the proof in this more general case is completely analogous. The proof of Proposition~\ref{lemma_new_2} is presented in Appendix~\ref{app:decres}. Then we have 
	\begin{align}\label{V_2}
		V_{|z-z'|\geq n^{-\gamma}} \lesssim \frac{n^2\eta_0^4}{4\wt\eta^2} \|\Delta f\|^2_1 \times n^{-1/2-\alpha-\widehat{\omega}/2} \lesssim
		 (\log n)^C n^{-\alpha+4\epsilon+2\alpha-\widehat{\omega}/2},
	\end{align}
which is negligible for $\epsilon,\alpha\le \widehat{\omega}/100$. Putting (\ref{V_1}) and (\ref{V_2}) together, from (\ref{V_1+V_2}) we obtain
\begin{align}\label{small_v}
\V^{\mathrm{Gin}} \Big[\frac{n\eta_0^2}{2\wt \eta}\int_{\C} \Delta f(z) \<G^{z}(\ii \wt\eta)\> \Big]&\dd^2 z=O\Big((\log n)^C \big(n^{-\gamma+4\epsilon} +n^{4\epsilon+\alpha-\widehat{\omega}/2}\big)\Big).
\end{align}
Choosing sufficiently small $\epsilon,\alpha,\gamma$ with $\epsilon,\alpha\le \widehat{\omega}/100$ and $\gamma>6\epsilon$, we have proved (\ref{eq_step12}) and thus finished the proof of Lemma \ref{lemma_step12}.
\end{proof}

\section{Step 2. i.i.d. ensemble: Integral over  the small $\eta$ regime $[0,n^{-1+\epsilon}]$}\label{sec:step2}

In this section, we aim to extend the Ginibre estimate for the small $\eta$-integral $I_0^{\eta_0}$ in Step 1. to the i.i.d. case
 and prove the analogue of Lemma \ref{lemma_step1},  but only in the first absolute moment sense.
\begin{lemma}\label{lemma_step2}
 Fix $\tau>0$.  Set $\eta_0=n^{-1+\epsilon}$ and $\wt \eta=n^{-3/4-\alpha}$ with sufficiently small $\epsilon,\alpha>0$ and $\alpha >\epsilon > \tau/2$. Then we have 
	\begin{align}\label{step_2_goal}
		\E\Big|	I_{0}^{\eta_0}(f)-\int_{\C} \Delta f(z) \int^{\eta_0}_{0} \sum_{i}\frac{\eta}{(\lambda^z_i)^2+\wt\eta^2} \dd \eta \dd^2 z\Big|=O\Big((\log n)^C n^{-2(\alpha-\epsilon)}\Big).
	\end{align}
\end{lemma}
We remark that for i.i.d. matrices, it is enough to estimate the first absolute moment in (\ref{step_2_goal}) instead of the the second moment. This simplifies our proof since estimating the second moment would require
 an analogue of Proposition \ref{lemma_new_1} for i.i.d. cases which would require much more effort. To prove Lemma \ref{lemma_step2}, we first state the following lower tail estimate for the smallest singular value of $X-z$ for generic i.i.d. matrices $X$.  
\begin{proposition}\label{prop1} 
	Consider general complex i.i.d. $X$ and let  $\lambda_1^z$ be the smallest singular value of $X-z$.
	 Fix any small $\epsilon_1, \epsilon_2,\tau>0$ with $\epsilon_2\geq \tau/2$. For any $E$ with $n^{-1+\epsilon_1}\leq E 
	\leq n^{-3/4-\epsilon_2}$, and any $\delta =|z|^2-1$ with $n^{-1/2} \ll \delta \leq n^{-1/2+\tau}$, we have
	\begin{align}\label{tau_result}
		\P\big(|\lambda_1^z| \leq E \big) \lesssim n^{3/2} E^2 e^{-n \delta^2/2}+n^{-100}.
	\end{align}
\end{proposition}

As a corollary of Proposition \ref{prop1}, we have the following estimates for the resolvent. The proofs of Proposition \ref{prop1} and \ref{lemma2} are postponed to Appendix \ref{sec:proof_tail}.
\begin{proposition}\label{lemma2} Fix any small $\epsilon_1, \epsilon_2,\tau>0$ with $\epsilon_2\geq \tau/2$.
	For any $n^{-1+\epsilon_1} \leq \eta \leq n^{-3/4-\epsilon_2}$ and any $n^{-1/2} \ll \delta =|z|^2-1\leq n^{-1/2+\tau}$, we have
	\begin{align}\label{key2}
		\E\big[ \big(\Im \<G^z(\ii \eta)\> \big)^{k} \big] =O_\prec\Big( (\sqrt{n}\eta)^k
		+\frac{n^{-1/2}}{(n\eta)^{k-2}}e^{-n \delta^2/2}+n^{-100}\Big), \qquad  \forall k\geq 1.
	\end{align}
\end{proposition}
 We remark that the above expectation estimates improve significantly over the 
a priori bounds from the local law, especially at the lower level $\eta=n^{-1+\epsilon_1}$. 
 Notice that for the Ginibre ensemble, the tail bound (\ref{tau_result}) directly follows from  
 Proposition \ref{prop:tail}~(without the irrelavent error term $n^{-100}$). We then use a 
 GFT and Gronwall argument similar to~\cite{EX22} to extend the Ginibre estimate to i.i.d. matrices. 
 In contrast to Proposition \ref{prop:tail},  we have an irrelevant error term $n^{-100}$ in (\ref{tau_result})
  for the i.i.d. cases. This is simply because our GFT method in \cite{EX22} yields polynomially 
  small error terms with a large power which cannot be incorporated into the exponential factor
$e^{-n\delta^2/2}$ when $\delta \gg n^{-1/2}(\log n)^{1/2}$. Nevertheless this weaker version is 
enough for us, because in the proof we only use the upper bound $e^{-n\delta^2/2} \lesssim n^{-1/2}$ 
(from (\ref{delta_regime})) and the irrelevant $n^{-100}$ can be absorbed into it. More precisely, from 
Proposition~\ref{prop1} and \ref{lemma2}, for any $n^{-1+\epsilon_1}\leq \eta\leq n^{-3/4-\epsilon_2}$ with $\epsilon_2\geq \tau/2$ and $z \in \mathrm{supp}(f^{-}_1) \cup \mathrm{supp}(f^{+}_2)$, 
we have  
\begin{align}\label{tail_bound_iid}
		\mathbf{P}\left(\lambda_1^z\le \eta \right)\lesssim n \eta^2,\qquad 
	 \E \Big[\big( \Im \<G^{z}(\ii \eta)\> \big)^k \Big]=O_\prec \Big((\sqrt{n}\eta)^k+\frac{n^{-1}}{(n\eta)^{k-2}}+n^{-100}\Big).
\end{align}
Again, all the estimates corresponding to $z \in \mathrm{supp}(f^{+}_2)$ are in principle smaller using (\ref{function2}) and (\ref{tau_result}) owing to the exponential factor, but our proof will not rely on this fact.

Given with (\ref{tail_bound_iid}) for i.i.d. cases, the proof of Lemma \ref{lemma_step2} is similar to that of Lemma \ref{lemma_step1} so we omit it for brevity.  In other words, Lemma \ref{lemma_step2} states that  also for the i.i.d. case we have 
\begin{align}
\label{eq_some_step2}
	&I_{0}^{\eta_0}(f)=\frac{n\eta_0^2}{2\wt\eta}\int_{\C} \Delta f(z) \<G^{z}(\ii \wt\eta)\> \dd^2 z+\mathcal{E}, \qquad \E|\mathcal{E}|=O(n^{-\kappa})
\end{align}
with some small $\kappa>0$. Note that using (\ref{deltaf}) and (\ref{key2}), we obtain an a priori bound for $\E|I_{0}^{\eta_0}(f)|$:
\begin{align}\label{naive_bound}
	\E|I_{0}^{\eta_0}(f)| =O_\prec((\log n)^{1/2} n^{2} \eta_0^2), \quad \qquad \eta_0=n^{-1+\epsilon},
\end{align}
which cannot be neglected. The formula (\ref{eq_some_step2}) expresses $I_{0}^{\eta_0}(f)$,
the contribution of the small $\eta$ regime, in terms of resolvents at a larger $\wt \eta$ level
for which we can perform standard iterative
 GFTs between the Ginibre and the i.i.d. ensembles. 
 We will need this only 
  for spectral parameter $\wt\eta=n^{-3/4-\alpha}$ but 
  in the next proposition we state and prove it for any $\eta\ge n^{-1+\epsilon}$
  since the essence of the iterative proof is the same ($\eta= n^{-1+\epsilon}$ requires $1/\epsilon$ iteration
  but this is not a major complication compared with four iterations needed for $\wt\eta= n^{-3/4-\alpha}$).
   We remark that the GFT analysis here is much easier  than the one required in Proposition~\ref{gft} below which will have essential new elements and where we will give all details. The proof of the following proposition is then postponed to Appendix~\ref{sec:proof_tail}.

\begin{proposition}\label{prop2_old}  Fix any small $\epsilon_1, \epsilon_2>0$. For any $1- Cn^{-1/2}\le |z|, |z'| \leq 1+c$ and $n^{-1+\epsilon_1}\leq\eta 
	\leq n^{-3/4-\epsilon_2}$, we have
	\begin{align}
		&\Big|	\E\big[ \<G^{z}(\ii \eta)\big]-	\E^{\mathrm{Gin}}\big[ \<G^{z}(\ii \eta)\big]  \Big|
		\prec\frac{1}{n^{5/2}\eta^2}+\frac{1}{n^5\eta^5}+n^{-1},\label{difference}\\
		\Big|	\Cov\big[ \<G^{z}(\ii \eta)&,\<G^{z'}(\ii \eta)\big]-	\Cov^{\mathrm{Gin}}\big[ \<G^{z}(\ii \eta),\<G^{z'}(\ii \eta)\big]  \Big|
		\prec \frac{1}{n^{7/2}\eta^3}+\frac{1}{n^6\eta^6}+n^{-1}.\label{difference_var}
	\end{align}
\end{proposition}

Using Lemma \ref{lemma_step2}, Proposition \ref{prop2_old}, and Lemma \ref{lemma_step12} for the Ginibre ensemble, we obtain the following:
\begin{lemma}\label{lemma_step2_2}
	For  i.i.d. matrices and $\eta_0= n^{-1+\epsilon}$ with a sufficiently small $\epsilon>0$, we have
		\begin{align}\label{eq_step21}
		&\E\Big[I_{0}^{\eta_0}(f)\Big]=\E^{\mathrm{Gin}}\Big[I_{0}^{\eta_0}(f)\Big]+O(n^{-\kappa}),\qquad \quad \E\Big|I^{\eta_0}_{0}(f)-\E\Big[I^{\eta_0}_{0}(f)\Big]\Big|=O(n^{-\kappa}),
	\end{align}
	for some small $\kappa>0$.
\end{lemma}

\begin{proof}[Proof of Lemma \ref{lemma_step2_2}]

	Using (\ref{eq_some_step2}), (\ref{difference}) and the bound for the $L^1$--norm of $\Delta f$ in (\ref{deltaf}), recalling
	that $\eta_0=n^{-1+\epsilon}$ and $\wt\eta= n^{-3/4-\alpha}$ with some sufficiently small $\epsilon,\alpha$ and $\alpha>\epsilon$, we obtain
	\begin{align}\label{eq_step2_ex}
		\E[I_{0}^{\eta_0}(f)]=&\frac{n\eta_0^2}{2\wt\eta}\int_{\C} \Delta f(z) \E\big[\<G^{z}(\ii \wt\eta)\>\big] \dd^2 z+O(n^{-\kappa})\nonumber\\
		=&\frac{n\eta_0^2}{2\wt\eta}\int_{\C} \Delta f(z) \E^{\mathrm{Gin}}\big[\<G^{z}(\ii \wt\eta)\>\big] \dd^2 z+O(n^{-\kappa})+O_\prec(n^{-3/4+2\epsilon+3\alpha})\nonumber\\
		=&\E^{\mathrm{Gin}}[I_{0}^{\eta_0}(f)]+O(n^{-\kappa}),
	\end{align}
where in the last line we also used (\ref{eq_step11}) for the Ginibre ensemble and that $\kappa$, $\epsilon$ and $\alpha$ are sufficiently small.
 This proves the first estimate in (\ref{eq_step21}). For the second estimate in~\eqref{eq_step21}, using (\ref{eq_some_step2}) and the Cauchy-Schwarz inequality, we have
	\begin{align}\label{step_2_some}
		\E\Big|I^{\eta_0}_{0}(f)-\E[I^{\eta_0}_{0}(f)]\Big|=&\frac{n\eta_0^2}{2\wt\eta} \E\Big|(1-\E)\int_{\C} \Delta f(z) \<G^{z}(\ii \wt\eta)\>\dd^2 z\Big|+O(n^{-\kappa})\nonumber\\
		\leq & \sqrt{ \frac{n^2\eta_0^4}{4\wt\eta^2} \V \Big[\int_{\C} \Delta f(z) \<G^{z}(\ii \wt\eta)\>\dd^2 z \Big]}+O(n^{-\kappa}).
	\end{align}
Using the $L^1$ norm of $\Delta f$ in (\ref{deltaf}) and the GFT variance estimate (\ref{difference_var}), recalling that $\eta_0=n^{-1+\epsilon}$ and $\wt\eta=n^{-3/4-\alpha}$ with sufficiently small $\epsilon,\alpha$ and $\alpha>\epsilon$, we have 
\begin{align}\label{ff_gft}
	\frac{n^2\eta_0^4}{4\wt\eta^2}\V \Big[\int_{\C} \Delta f(z) \<G^{z}(\ii \wt\eta)\>\dd^2 z \Big] =&\frac{n^2\eta_0^4}{4\wt\eta^2}	\V^{\mathrm{Gin}} \Big[\int_{\C} \Delta f(z) \<G^{z}(\ii \wt\eta)\>\dd^2 z \Big]+O_\prec(n^{-1/2+4\epsilon+2\alpha})\nonumber\\
	=&O(n^{-\kappa})+O_\prec(n^{-1/2+4\epsilon+2\alpha}),
\end{align}
where we used (\ref{eq_step12}) in the last line. Plugging (\ref{ff_gft}) in (\ref{step_2_some}), we finished the proof of (\ref{eq_step21}).
\end{proof}

\section{Step 3. i.i.d. ensemble: Large $\eta$ integral over $[n^{-1+\epsilon},T]$}
\label{sec:GFT}

In this section we focus on the large $\eta$-integral $I_{\eta_0}^{T}(f)$ in (\ref{expectation_eq}) with $\eta_0=n^{-1+\epsilon}$, and we prove the GFTs for the expectation and variance, respectively.

\begin{proposition}\label{gft}
	 Fix $\tau,\epsilon>0$ with $\epsilon>\tau/2$ and let $\eta_0=n^{-1+\epsilon}$. Then we have
	\begin{align}
		\Big|	\big(\E-\E^{\mathrm{Gin}}\big)   \int_{\C} \Delta f(z)  \int_{\eta_0}^T \Im \Tr G^{z}(\ii \eta) \dd \eta \Big|&=O_\prec(n^{-1/2-\epsilon}),\label{exp_gft}\\
		\Big|	\big(\V-\V^{\mathrm{Gin}}\big)   \int_{\C} \Delta f(z)  \int_{\eta_0}^T \Im \Tr G^{z}(\ii \eta) \dd \eta \Big|&=O_\prec(n^{-\epsilon/4}).\label{variance_gft}
	\end{align}
\end{proposition}
 A similar GFT statement was proved in \cite[Proposition 3.8]{maxRe} for the purpose
of  estimating the rightmost eigenvalue, 
where the corresponding functions $f$ contributed an extra factor $\int |\Delta f| \sim n^{1/4}$ and the lower limit of  the $\eta$-integrals was slightly below the intermediate level $n^{-7/8}$. 
In the present paper we lose more in $\int |\Delta f| \sim n^{1/2}$ from (\ref{deltaf}) 
 and the error term
 in the local law is very bad for our small $\eta_0=n^{-1+\epsilon}$; we therefore need to develop a much finer  GFT analysis.
 
Before we enter the details, we explain the new ingredients in our proof compared to \cite{maxRe}.
Focusing only on the more involved   (\ref{variance_gft}), the variance difference is
given explicitly in terms of the third and higher order cumulants of the matrix entries  in 
(\ref{deri_var})--\eqref{pq_order_terms} below.  Similar terms emerged in  \cite{maxRe}, but 
our new estimate on them (given in Proposition \ref{key_lemma} below) involves the following novelties:
\begin{itemize}
	\item[(i)] \emph{Third order terms in
	(\ref{pq_order_terms}) with distinct indices $a \neq \ud{B}$}: these are 
	unmatched terms (see Definition~\ref{def:unmatch_form} in Appendix~\ref{appendix_A}) 
	and were bounded by $O_\prec(n^{-1})$ in \cite{maxRe}. 	
Now we need to identify the leading error term $O_\prec(n^{-1})$
	precisely which vanishes against $\Delta f(z)$ after $z$-integrations and we 
	hence improve the error $O(n^{-1})$ slightly to $O(n^{-1-\epsilon})$ in order 
	to compensate the loss in $\int |\Delta f| \sim n^{1/2}$; see \eg (\ref{third_aB_goal_0}) below.
	
	\item[(ii)] \emph{Third order terms in~\eqref{pq_order_terms} with index coincidence $a=\ud{B}$}: 
	 these matched terms were estimated trivially in~\cite{maxRe} using the local law on 
	the intermediate $\eta$-level $n^{-7/8}$. 
	Now on a much smaller $\eta$-level slightly above $n^{-1}$, we exploit an
	additional cumulant expansion together with Proposition \ref{lemma2} to obtain improved error estimates; again 
	we find the explicit leading terms will vanish.
Most importantly, with a better stability factor~(\ref{stable_0}) for $|z-z'|\geq n^{-\gamma}$
 in some self-consistent equation, we gain an extra smallness from the $|z-z'|$-decorrelation 
 effect as in Lemma \ref{lemma4}.

	\item[(iii)] \emph{Higher order terms in~\eqref{pq_order_terms}}:
	 while estimated trivially in~\cite{maxRe} by the local law, 
	here we again need to gain the $|z-z'|$-decorrelation effect (for the fourth order) 
	and use iterative cumulant expansions
	 (for the fifth order). Since we gain more from higher order cumulants, this part is somewhat easier than (i)-(ii).
		
\end{itemize}

We now introduce some notations which we will use throughout this section.
 The same notations have been used in \cite{maxRe}; for reader's convenience we recall them here.

\begin{notation}
	\label{not:index} 
	We use lower case
	letters to denote the indices taking values in $\llbracket 1, n \rrbracket$ and upper case letters to denote
	the indices taking values in $\llbracket n+1, 2n \rrbracket$. We also use calligraphic letters 
	$\mathfrak{u}, \mathfrak{v}$ to denote the indices ranging fully from $1$ to $2n$.
	
	For any index $\mathfrak{v} \in \llbracket 1, 2n \rrbracket$, the conjugate of $\mathfrak{v}$, denoted 
	by $\mathrm{conj}(\mathfrak{v})\in \llbracket 1, 2n \rrbracket$, is defined by the relation $|\mathrm{conj}(\mathfrak{v})-\mathfrak{v}|=n$. In particular, for an index $ a \in \llbracket 1, n \rrbracket$, 
	we define its index conjugate $\mathrm{conj}(a)=\bar a:=a+n$,
	and  for an index $ B \in \llbracket n+1, 2n \rrbracket$ we define its index conjugate $\mathrm{conj}(B)=\ud B:=B-n$. 
	With a slight abuse of terminology,  we say that two indices {\it coincide} if either they are equal or one is equal to 
	the conjugate of the other one. For instance, we say $a\in \llbracket 1, n \rrbracket$ coincides with the 
	index $B\in \llbracket n+1, 2n \rrbracket$ if $a=\ud{B}$ (or equivalently $B=\bar a$). 
	We also say that a collection of indices are {\it distinct} if there is no index coincidence among them (in the sense explained above). 
	
	Moreover, we often use generic letters  $x$ and $y$ to denote the row and the column
	 index of a Green function entry $G_{xy}$. 
	 In this context the lower case letters $x,y$ do not indicate
	 that they take values in $\llbracket 1, n \rrbracket$; later we will assign actual
	summation indices, \eg $a,B$ or their index conjugates to them. The assignment is
	 denoted by the symbol $\equiv$, for example $x_i \equiv a$, $y_i \equiv \ud{B}$ means that 
	the generic Green function entry $G_{xy}$ is replaced with the actual 
	$G_{a\ud{B}}$.

	\bigskip

\end{notation}

We will prove Proposition \ref{gft} via a continuous interpolating flow. 
Though we present the proof for $X$ being a complex-valued matrix for simplicity, the same result holds for the real case; see Remark \ref{handwave_real}.
Given the initial ensemble $H^{z}$ 
in (\ref{initial}), we consider the following Ornstein-Uhlenbeck matrix flow
\begin{align}\label{flow}
	\dd H^z_t=-\frac{1}{2} (H^z_t+Z)\dd t+\frac{1}{\sqrt{n}} \dd \mathcal{B}_t, \quad Z:=\begin{pmatrix}
		0  &  zI  \\
		\overline{z}I   & 0
	\end{pmatrix}, \quad 
	\mathcal{B}_t:=\begin{pmatrix}
		0  &  B_t  \\
		B^*_t   & 0
	\end{pmatrix}
\end{align}
 with initial condition $H^z_{t=0}:= H^z$, 
where $B_t$ is an $n \times n$ matrix with i.i.d. standard complex valued Brownian motion entries.
The matrix flow $H_t^z$ interpolates between the initial matrix $H^{z}$ in (\ref{initial}) at $t=0$ 
and a Hermitized matrix as in (\ref{initial}) with $X$ being replaced with an independent complex Ginibre ensemble
at $t=\infty$.

The Green function of the time dependent matrix $H^{z}_t$, denoted by $G^{z}_t$, satisfies the following time-dependent local law~(\cf Theorem \ref{local_thmw} with $w=\ii \eta$): for any $t \in \R^+$, any $||z|-1| \leq c$ and $n^{-1+\epsilon_1} \leq \eta\leq n^{-3/4-\epsilon_2}$, 
\begin{align}\label{localg}
	\sup_{t\ge 0}\max_{1\leq \mathfrak{v},\mathfrak{u} \leq 2n} \Big\{\big| \big(G^{z}_t(\ii \eta) \big)_{\mathfrak{uv}}-\big(M^{z}(\ii \eta) \big)_{\mathfrak{uv}} \big| \Big\} \prec\Psi:=\frac{1}{n\eta}+ \sqrt{\frac{\rho}{n\eta}}.
\end{align}
holds uniformly, since the flow in (\ref{flow}) is stochastically H\"{o}lder continuous in time as discussed above \cite[Eq. (4.3)]{maxRe}. 
Here $M^{z}=M^{z}(\ii \eta)$ is the deterministic $(2n)\times (2n)$ block-constant
matrix from (\ref{Mmatrix}) on the imaginary axis, \ie
\begin{align}\label{Mmatrix_0}
	M^{z}=\begin{pmatrix}
		m^{z}(\ii \eta)  &  \mathfrak{m}^{z}(\ii \eta)  \\
		\overline{\mathfrak{m}^{z}}(\ii \eta)   & m^{z}(\ii \eta)
	\end{pmatrix}, \qquad \mathfrak{m}^{z}(\ii \eta):=zu^z(\ii \eta),\qquad u^z(\ii \eta)=\frac{\Im {m}^z(\ii \eta)}{\eta+\Im {m}^z (\ii \eta)}.
\end{align}
From (\ref{m_function}), $m^{z}(\ii \eta)$ is pure imaginary, $u^{z}(\ii \eta)$ is real, and both functions are radial functions depending on $|z|$. In our main proof we only need the above estimates for
 $z \in \mathrm{supp}(f^{-}_1) \cup \mathrm{supp}(f^{+}_2)$ from (\ref{function1})-(\ref{function2}) 
and at the level $\eta_0=n^{-1+\epsilon}$ and thus $\Psi=n^{-\epsilon}$. From (\ref{rho}), we have the following upper bounds
\begin{align}\label{localm}
	|m^{z}(\ii \eta)| \lesssim \sqrt{n}\eta, \qquad  |u^{z}(\ii \eta) |\lesssim 1, \qquad |\mathfrak{m}^{z}(\ii \eta)| \lesssim 1.
\end{align}
Note that the two diagonal blocks in (\ref{Mmatrix_0}) are small, while the two off-diagonal blocks are typically order one. Without specific mentioning, all the estimates in this section hold true 
uniformly for any $t\ge 0$, $\eta=\eta_0=n^{-1+\epsilon}$, and for any  $z \in \mathrm{supp}(f^{-}_1) \cup \mathrm{supp}(f^{+}_2)$. For notational simplicity we often drop the dependence on the parameters $t$, $\eta_0$ and $z$.

\begin{proof}[Proof of Proposition \ref{gft}]
	We present only the proof of the
	more involved variance estimate in (\ref{variance_gft}). 
	The proof of (\ref{exp_gft}) is much easier with an even better error term $O_\prec(n^{-1/2-\epsilon})$ due to less overestimate of $\int |\Delta f|$ in (\ref{deltaf}) for the expectation, so we omit its proof for brevity. 
As a starting point we will rely  on \cite[Section 5]{maxRe}.

	Recall the matrix interpolating flow in $H^z_t$ (\ref{flow}) with complex-valued $X$ and its resolvent $G_t^z$. 
	As in \cite[Section 5.2]{maxRe} we introduce  the following short-hand notations, $j=1,2$
	\begin{align}\label{F}
		\wh{ \F_t^{z_j}} :=\F_t^{z_j}-\E[\F_t^{z_j}]=-\ii \int_{\eta_0}^T \Big( \Tr G_t^{z_j}(\ii \eta)
		-\E\big[ \Tr G_t^{z_j}(\ii \eta)\big]\Big) \dd \eta \prec 1,
	\end{align}
	where the last estimate follows from the local law in (\ref{localg}).
We aim to prove that
\begin{align}\label{derivative_estimate}
	\Big|\int_{\C} \int_{\C} \Delta f(z_1){ \Delta f(z_2) }  \frac{\dd	\E\big[ \wh{ \F_t^{z_1}} \wh{\F_t^{z_2}} \big]}{\dd t}\dd^2 z_1 \dd^2 z_2\Big|=O_\prec(n^{-\epsilon/4}).
\end{align}
Once we proved (\ref{derivative_estimate}), integrating it over $t \in [0, t_0]$ with $t_0:=800\log n$ we obtain 
\begin{align}\label{0tot_0}
	\Big| \int_{\C} \int_{\C} \Delta f(z_1){ \Delta f(z_2) }  \Big(\E\big[ \wh{ \F_0^{z_1}} \wh{\F_0^{z_2}} \big]-\E\big[ \wh{ \F_{t_0}^{z_1}} \wh{\F_{t_0}^{z_2}} \big] \Big)\dd^2 z_1 \dd^2 z_2 \Big| =\OO_{\prec}\big((\log n)n^{-\epsilon/4}\big). 
\end{align}
Note that $H_t^{z}$ in (\ref{flow}) is given as in (\ref{initial}) with $X$ being replaced with the time dependent matrix
$$X_t \stackrel{{\rm d}}{=} e^{-\frac{t}{2}} X+\sqrt{1-e^{-t}} \mathrm{Gin} (\mathbb{C}), 
\qquad t\ge 0,$$
where $X_{\infty} \stackrel{{\rm d}}{=}  \mathrm{Gin} (\mathbb{C})$ is the complex 
Ginibre ensemble which is independent of $X$. Then we have
\begin{align}\label{approxxxx}
	\|G^{z}_{t_0}(\ii \eta)-G^{z}_{\infty}(\ii \eta)\| \leq \|G^{z}_{t_0}\| \|G^{z}_{\infty}\| \|X_{t_0}-X_{\infty}\| 
	\prec n^{-398},
\end{align}
where we used that $\|G^z(\ii \eta)\| \leq \eta^{-1}\le n$ and that $|x_{ij}| \prec n^{-1/2}$ from the moment assumption in \eqref{eq:hmb}. Using the $L^1$ bound of $\Delta f$ in (\ref{deltaf}) and that $T=n^{100}$, we have
\begin{align}\label{t_0toinf}
	\Big| \int_{\C} \int_{\C} \Delta f(z_1){ \Delta f(z_2) }  \Big(\E^{\mathrm{Gin}}\big[ \wh{ \F^{z_1}} \wh{\F^{z_2}} \big]-\E\big[ \wh{ \F_{t_0}^{z_1}} \wh{\F_{t_0}^{z_2}} \big] \Big)\dd^2 z_1 \dd^2 z_2 \Big|=O_\prec(n^{-195}).
\end{align}	
Combining (\ref{0tot_0}) with (\ref{t_0toinf}) we hence finished the proof of Proposition \ref{gft}.
 
In the rest of proof, we focus on proving the key estimate (\ref{derivative_estimate}). Set 
 \begin{align}\label{Wmatrix}
 	W = W_t:=H^z_t+Z=\begin{pmatrix}
 		0  &  X_t  \\
 		X_t^*   & 0
 	\end{pmatrix}.
 \end{align}
 Then  $W_t$ satisfies the usual matrix OU flow:
 $$
 \dd H^z_t = \dd W_t=-\frac{1}{2}W_t\dd t+\frac{1}{\sqrt{n}} \dd \mathcal{B}_t.
 $$ 
	Applying Ito's formula to $\wh{ \F_t^{z_1}} \wh{\F_t^{z_2}}$ in (\ref{derivative_estimate})
	and performing the cumulant expansions on the expectation, we observe the precise cancellations of the second order terms with $p+q+1=2$ and obtain that (see also \cite[Eq (5.16)]{maxRe})
\begin{align}\label{deri_var}
	\frac{\dd	\E\big[ \wh{ \F_t^{z_1}} \wh{\F_t^{z_2}} \big]}{\dd t}=-\frac{1}{2}\sum_{p+q+1=3}^{K_0} \frac{1}{p!q!} \Big( \mathcal{K}^{z_1,z_2}_{p+1,q}+{\wt{\mathcal{K}}}^{z_1,z_2}_{q,p+1} \Big)+O_\prec(n^{-\frac{K_0}{2}+2})
\end{align}
where we define for simplicity
\begin{align}\label{pq_order_terms}
   \mathcal{K}^{z_1,z_2}_{p+1,q}:=&  \frac{c^{(p+1,q)}}{n^{\frac{p+q+1}{2}}}  \sum_{a,B} \E \left[  \frac{\partial^{p+q+1} \wh{\F_t^{z_1}} \wh{\F_t^{z_2}}}{\partial w_{aB}^{p+1} \partial \overline{w_{aB}}^{q} }\right],
\end{align}
and ${\wt{\mathcal{K}}}^{z_1,z_2}_{q,p+1}$ is the same with $p+1$ and $q$ interchanged, with $c^{(p,q)}$ the $(p,q)$-cumulants of the normalized complex-valued i.i.d. entries $\sqrt{n} w_{aB}$~(we omit their dependence on $t$) that are uniformly bounded from (\ref{eq:hmb}). Here we
truncate the cumulant expansions at a sufficiently large $K_0$-th order, say $K_0=100$, using the local law in
(\ref{localg}) and the finite moment condition in \eqref{eq:hmb}. 
To compute each $\mathcal{K}^{z_1,z_2}_{p+1,q}$ in (\ref{pq_order_terms}), we recall the following differentiation rules from~\cite[Eqs. (4.13), (5.8)]{maxRe} for any $1\leq \mathfrak{u},\mathfrak{v} \leq 2n$
	\begin{align}\label{rule_12}
		\frac{\partial G^{z}_{\mathfrak{uv}}}{\partial w_{aB}}
		=-G^{z}_{\mathfrak{u}a}G^{z}_{B\mathfrak{v}}, \qquad \quad
		\frac{\partial \wh{\F_t^{z}}}{\partial w_{aB}}=-\gz_{Ba}(\ii \eta_0)+\gz_{Ba}(\ii T)=-\gz_{Ba}(\ii \eta_0)+\OO(n^{-100}),
	\end{align}
where the latter rule can be proved using the former one and that 
$(G^2)(\ii \eta)=-\ii \frac{\dd G(\ii \eta)}{\dd \eta}$ and the deterministic norm bound $\|G(\ii T)\|\leq T^{-1}$ with $T=n^{100}$. A similar differentiation rule for $\partial/\partial \overline{w_{aB}} $ holds with $a$ and $B$ interchagned.

Using the differentiation rules in (\ref{rule_12}), each term $\mathcal{K}^{z_1,z_2}_{p+1,q}$ 
	in (\ref{pq_order_terms}) is a linear combination of products of $p+q+1$ Green function entries
	(either $\ga(\ii \eta_0)$ or $\gb(\ii \eta_0)$ up to an error $\OO_\prec(n^{-100})$) with  possible factors
	 $\wh{\F_t^{z_1}}$ or $\wh{\F_t^{z_2}}$ in front, \ie these are expressions of the following general form
	\begin{align}\label{some_form_F}
		\frac{1}{n^{\frac{p+q+1}{2}}}  \sum_{a,B}  \E \Big[  (\wh{\F_t^{z^{(0)}}})^{\alpha_0} \prod_{i=1}^{p+q+1} G^{z^{(i)}}_{x_i,y_i} (\ii \eta_0)\Big],
	\end{align}
with $\alpha_0=0,1$, where $z^{(i)}$ stands for either $z_1$ or $z_2$, and $x_i, y_i$ denote generic row and column indices of $G^{z^{(i)}}$, respectively, to which we assign actual
summation indices $a,B$ based on (\ref{pq_order_terms})--(\ref{rule_12}). The specific assignments in (\ref{some_form_F}) all have the
 following properties: 
\begin{align}\label{pq_number}
	\#\{x_i \equiv a\}=\#\{y_i \equiv B\}=q, \qquad \#\{ x_i \equiv B\}=\#\{ y_i \equiv a\}=p+1.
\end{align}
From the local law in (\ref{localg}), (\ref{localm}), 
and that that $|\F_t^{z}| \prec 1$ from~\eqref{F}, we have the following a priori bound
	\begin{align}\label{naive_g_two}
		|\mathcal{K}^{z_1,z_2}_{p+1,q}|=O_{\prec}\big(n^{-\frac{p+q-3}{2}}(\Psi^{p+q+1}+n^{-1})\big),\qquad \Psi=n^{-\epsilon},
	\end{align}
where the error term $n^{-1}$  corresponds to the cases with an index coincidence $a=\ud{B}$. In particular for $6 \leq p+q+1\leq K_0$, using (\ref{naive_g_two}) and the $L^1$ norm of $\Delta f$ in (\ref{deltaf}) we 
 already have  the direct upper bound
	\begin{align}\label{var5}
		\int_{\C} \int_{\C} |\Delta f(z_1)|| \Delta f(z_2)| \sum_{p+q+1 = 6}^{K_0} \Big|\mathcal{K}^{z_1,z_2}_{p+1,q}  \Big| \dd^2 z_1 \dd^2 z_2=O_\prec\big((\log n) n^{-6\epsilon}\big),
	\end{align}
	so these higher order cumulant terms need no further refined estimates. 
	
Precise estimates on the remaining terms $\mathcal{K}^{z_1,z_2}_{p+1,q}$ in (\ref{pq_order_terms}) with
 $3\leq p+q+1\leq 5$ 
are more delicate. We need to find the leading order deterministic terms to $\mathcal{K}^{z_1,z_2}$ and
show that while they are not negligible, their $(z_1, z_2)$-integrals against $\Delta f(z_1){ \Delta f(z_2) }$ are vanishing.
This shows that the final contribution of $\mathcal{K}^{z_1,z_2}$ is smaller than it naively looks like. 
The following lemma states this fact precisely:  
\begin{proposition}\label{key_lemma}
	There exist bounded  deterministic functions depending on $p, q$,
	 denoted by  $\mathcal{M}_{p+1,q}(z_1,z_2)$,  
	satisfying the following integral condition
	\begin{align}\label{delta_int}
		\int_{\C} \int_{\C} \Delta f(z_1){ \Delta f(z_2) }\Big( \mathcal{M}_{p+1,q}(z_1,z_2) \Big) \dd^2 z_1 \dd^2 z_2=0,
	\end{align}
	 such that
\begin{align}\label{third_fourth_result}
	\iint  |\Delta f(z_1)|  | \Delta f(z_2)|  \sum_{p+q+1=3}^5 \Big| \mathcal{K}^{z_1,z_2}_{p+1,q}-\mathcal{M}_{p+1,q}(z_1,z_2)  \Big|\dd^2 z_1 \dd^2 z_2 =O_\prec(n^{-\epsilon/4}),
\end{align}
\end{proposition}
Using (\ref{var5})-(\ref{third_fourth_result}), we conclude (\ref{derivative_estimate}) from (\ref{deri_var}) and thus finish the proof of Proposition \ref{gft}.

\end{proof}

 Next we prove Proposition \ref{key_lemma} in the following four subsections by estimating the third to fifth order terms $\mathcal{K}^{z_1,z_2}_{p+1,q}$ given in~\eqref{pq_order_terms} for $p+q+1=3,4,5$ respectively. The most involved one is the third order terms and we split the discussion into two cases: restricted summations with distinct indices $a\neq \ud{B}$ and with index coincidence $a=\ud{B}$, denoted by $\mathcal{K}^{z_1,z_2}_{p+1,q} \Big|_{a \neq \ud{B}}$ and $\mathcal{K}^{z_1,z_2}_{p+1,q} \Big|_{a=\ud{B}}$ respectively. The fourth and fifth order terms can be estimated similarly and more easily since we gain more from higher order cumulants. The proofs of some technical lemmas will be deferred to the Appendix~\ref{appendix_A} and \ref{sec:some_lemma}, but
 we explain the main ideas behind them.

\subsection{Third order terms with $p+q+1=3$ and with distinct indices $a\neq \ud{B}$
 in~\eqref{third_fourth_result}}\label{subsec:part1} 
By direct computations using (\ref{rule_12}), the third order terms $\mathcal{K}^{z_1,z_2}_{p+1,q}$ given in (\ref{pq_order_terms}) are linear combinations of the following terms (plus analogous terms 
when interchanging $z_1$ with $z_2$, or interchanging $a$ with $B$)
\begin{align}\label{third}
	&\frac{\sqrt{n}}{n^{2}}\sum_{a,B}\E \Big[ \wh{ \F_t^{z_2}} G^{z_1}_{aa} G^{z_1}_{BB}G^{z_1}_{aB} \Big], \qquad \frac{\sqrt{n}}{n^{2}} \sum_{a,B} \E \Big[ \ga_{aB} \gb_{aa}\gb_{BB}\Big],\nonumber\\
	&\frac{\sqrt{n}}{n^{2}} \sum_{a,B}\E \Big[ \wh{ \F_t^{z_2}} G^{z_1}_{aB} G^{z_1}_{aB}G^{z_1}_{aB} \Big], \quad 
	\frac{\sqrt{n}}{n^{2}} \sum_{a,B} \E \Big[ \ga_{aB}\gb_{aB} \gb_{aB} \Big], \quad 
	\frac{\sqrt{n}}{n^{2}} \sum_{a,B} \E \Big[ \ga_{aB} \gb_{Ba} \gb_{Ba} \Big],
\end{align}
where  $G^{z_1} =G^{z_1}_{t}(\ii \eta_0)$, $G^{z_2} =G^{z_2}_{t}(\ii \eta_0)$ with $t\in \R^+$, $\eta_0=n^{-1+\epsilon}$ and $z_1,z_2 \in \mathrm{supp}(f^{-}_1) \cup \mathrm{supp}(f^{+}_2)$ from (\ref{function1})-(\ref{function2}).

We first split the above summations over $a$ and $B$ into two parts: the restricted summation with distinct 
indices $a\neq \ud{B}$ and the remaining summation with the index coincidence $a=\ud{B}$. Note that the latter
yields off-diagonal resolvent terms that are large, typically order one, see~\eqref{localg}--\eqref{localm}. 
 For example, the last term in (\ref{third}) can be split into two parts, \ie
\begin{align}\label{split_third}
	\frac{\sqrt{n}}{n^{2}}\sum_{a, B}\E \Big[ \ga_{aB} &\gb_{Ba} \gb_{Ba} \Big]=\frac{\sqrt{n}}{n^{2}} \sum_{a \neq \ud{B}}\E \Big[ \wh{\ga_{aB}} \wh{\gb_{Ba}} \wh{\gb_{Ba}} \Big]+\frac{\sqrt{n}}{n^{2}} \sum_{a}\E \Big[ \ga_{a \bar a} \gb_{\bar a a} \gb_{ \bar a a} \Big],
\end{align}
where we define the shifted Green function by
\begin{align}\label{shifted}
	\wh{\gz}=\wh{\gz}(\ii \eta):=\gz(\ii \eta)-M^{z}(\ii \eta), 
\end{align}
with $M^{z}=M^{z}(\ii \eta)$ given in (\ref{Mmatrix_0})-(\ref{localm}). 
In particular for $a \neq \ud{B}$, $\gz_{aB}=\wh{\gz_{aB}}=O_\prec(\Psi)$ 
with $\Psi=n^{-\epsilon}$ from the local law in (\ref{localg}). This subsection 
is devoted to estimating the third order terms in (\ref{third}) with restricted 
summations $a \neq \ud{B}$, \eg the first part in (\ref{split_third}). The remaining 
summations with the index coincidence $a=\ud B$, \eg the second part in (\ref{split_third})
 will be estimated in the next subsection using a different approach.

Note that both the index $a$
 and $B$
 are 
assigned odd number of times as a row/column index of Green function entries in
 the first part in (\ref{split_third}), \ie in
 the third order terms with
the restricted summations  $a \neq \ud{B}$ for the products of Green function entries in (\ref{third}).
  We will call these terms \emph{unmatched} (without the additional factor $\sqrt{n}$) 
 with \emph{unmatched indices} $a$ and $B$.  More generally, 
 an unmatched term, denoted by $P^o_d$ for any degree $d\in \N$,
  is an averaged product of $d$ shifted Green function entries with unmatched indices and with a
   possible $\wh \F$ prefactor. The concept of unmatched terms
    were defined informally in \cite[Section 5]{maxRe}. 
    For completeness we also give their formal definition in Appendix~\ref{appendix_A} (see  Definition \ref{def:unmatch_form}). Here we just show a few examples: 
\begin{align}
	&d=3: \qquad \frac{1}{n^{2}}\sum_{a\neq \ud{B}}\E \Big[  \wh{G^{z_1}_{aB}} \wh{G^{z_2}_{aB}} \wh{G^{z_2}_{aB}} \Big], \qquad \qquad \frac{1}{n^{2}}\sum_{a\neq \ud{B}}\E \Big[ \wh{ \F_t^{z_2}} \wh{G^{z_1}_{aa}} \wh{G^{z_1}_{BB}} \wh{G^{z_1}_{aB}} \Big], \label{threeG}\\
	&d=4: \qquad \frac{1}{n^{2}}\sum_{a\neq \ud{B}}\E \Big[  \wh{G^{z_1}_{aB}} \wh{G^{z_2}_{aB}} \wh{G^{z_2}_{aB}} \wh{G^{z_2}_{Ba}}\Big], \qquad \frac{1}{n^{2}}\sum_{a\neq \ud{B}}\E \Big[ \wh{ \F_t^{z_2}} \wh{G^{z_1}_{aa}} \wh{G^{z_1}_{BB}} \wh{G^{z_1}_{aB}} \wh{G^{z_1}_{aB}}  \Big]\label{fourG}
\end{align}
are all unmatched terms with unmatched indices $a$ and $B$.  For example,  
 the second line (\ref{fourG}) stems from the fourth order terms $\mathcal{K}^{z_1,z_2}_{p+1,q}$ with
  $q\neq 2$ that will be estimated similarly later. We also give some examples of matched terms: 
$$\frac{1}{n}\sum_{a}\E \Big[ \wh{ \F_t^{z_2}} \wh{G^{z_1}_{a \bar a}} \wh{G^{z_1}_{a \bar a}} \wh{G^{z_1}_{a \bar a}}  \Big], \qquad \frac{1}{n^{2}}\sum_{a\neq \ud{B}}\E \Big[  \wh{G^{z_1}_{aB}} \wh{G^{z_1}_{Ba}} \wh{G^{z_2}_{aB}} \wh{G^{z_2}_{Ba}} \Big],$$
which will be estimated differently in the next two subsections, corresponding to the third order terms with the index coincidence $a=\ud B$ and the fourth order term $\mathcal{K}^{z_1,z_2}_{2,2}$, respectively.

The naive estimate for any unmatched term of degree $d$, denoted by $P^o_d$ is $O_\prec(\Psi^d)=O_\prec(n^{-d\epsilon})$ 
using simply (\ref{localg}).
However, it 
  can be improved significantly  to
  \begin{align}\label{P_d_o_before}
	|\E[P^o_d]|=O_\prec(n^{-3/2}), \qquad \qquad d\geq 1,
\end{align}
  by performing an iterative cumulant expansions on the unmatched indices. 
  This was shown in 
   \cite[Proposition 4.5]{maxRe} if  the $\mathcal{F}$ factors were not present.
 In  Proposition \ref{prop:unmatched} in the Appendix~\ref{appendix_A} we give the proof for the general case
 extending the argument from \cite{maxRe}.  
 
   Hence, using~\eqref{P_d_o_before}, 
   the third order terms in (\ref{third}) with $a \neq \ud{B}$, \eg the first part in (\ref{split_third}) can be bounded by
\begin{align}\label{old}
	\Big|\frac{\sqrt{n}}{n^{2}} \sum_{a \neq \ud{B}}\E \Big[ \wh{\ga_{aB}} \wh{\gb_{Ba}} \wh{\gb_{Ba}} \Big] \Big|=O_\prec(n^{-1}).
\end{align}
Combining with the $L^{1}$ norm of $\Delta f$, we have
\begin{align}\label{third_before}
	\int_{\C} \int_{\C} |\Delta f(z_1)|| \Delta f(z_2)| \left|\sum_{p+q+1 =3} \mathcal{K}^{z_1,z_2}_{p+1,q}  \Big|_{a\neq \ud{B}} \right|\dd^2 z_1 \dd^2 z_2=O_\prec(1).
\end{align}
However the above estimate is barely not enough to prove (\ref{third_fourth_result}) for $p+q+1=3$. To gain a little improvement, we will show a more refined estimate than (\ref{P_d_o_before})
 in Proposition \ref{prop:unmatched} in Appendix~\ref{appendix_A} \ie
\begin{align}\label{P_d_o_improve}
	|\E[P^o_d]|=O_\prec(n^{-3/2-\epsilon}), \qquad \qquad d\geq 4.
\end{align}
Compared to \cite{maxRe} the real novelty in the current proof is to use iterative expansions and the improved estimate in~(\ref{P_d_o_improve}) 
to identify the leading deterministic terms that contribute $O_\prec(1)$ in (\ref{third_before}) before taking the absolute values inside the integral 
and show that they vanish after $z$-integrations. The error term has degree at least four
and hence can be estimated by the improved bound (\ref{P_d_o_improve}).
  For instance, the first part in (\ref{split_third}) can be bounded by, \cf \eqref{old}
	\begin{align}\label{third_aB_goal_0}
	\frac{\sqrt{n}}{n^{2}} \sum_{a \neq \ud{B}}\E \Big[ \wh{\ga_{aB}} \wh{\gb_{Ba}} \wh{\gb_{Ba}} \Big]=\frac{C}{n}  \big(\m^{z_1}\big)^2 (\overline{\m^{z_2}})^4 +O_\prec(n^{-1-\epsilon}),
\end{align}
for some numerical constant $C\in \R$. Since $\m^{z}=zu^{z}$ from (\ref{Mmatrix_0}) and 
both $f(z)$ and $u^{z}$ are radial functions in $|z|$,  then we have
\begin{align}\label{pq_m_int_0}
	\int_{\C} \Delta f(z) \big(\m^{z}\big)^p  (\overline{\m^{z}})^q \dd^2 z=0, \qquad \mbox{ unless $p=q$}.
\end{align}
Thus the leading term in (\ref{third_aB_goal_0}) of size $n^{-1}$ satisfies the integral condition in (\ref{delta_int}). In general, we have the following lemma for all the third order terms with $a\neq \ud{B}$. The proof details are found in Appendix~\ref{appendix_A}.

\begin{lemma}\label{lemma_third_order}
	There exists bounded deterministic functions, denoted by $\mathcal{M}^{(1)}_{p+1,q}(z_1,z_2)$ with $p+q+1=3$ satisfying the integral condition in (\ref{delta_int}) such that
	\begin{align}\label{result_third}
		\iint |\Delta f(z_1)|| \Delta f(z_2)| \left|\sum_{p+q+1=3} \Big(\mathcal{K}^{z_1,z_2}_{p+1,q}\Big|_{a\neq \ud{B}}  -\mathcal{M}^{(1)}_{p+1,q}(z_1,z_2) \Big)\right| \dd^2 z_1 \dd^2 z_2 =O_\prec(n^{-\epsilon}).
	\end{align}
\end{lemma}

\subsection{Third order terms with index coincidence $a=\ud{B}$  in~\eqref{third_fourth_result}}\label{subsec:part2}
In this subsection, we study the third order terms computed in (\ref{third}) with the index coincidence $a=\ud{B}$. More precisely, they are the following matched terms 
(plus their versions interchanging $z_1$ with $z_2$, or interchanging $a$ with $\overline a$)
\begin{align}
	&\frac{1}{n^{3/2}}\sum_{a}\E \Big[ \wh{ \F_t^{z_2}} G^{z_1}_{aa} G^{z_1}_{\bar a \bar a}G^{z_1}_{a \bar a} \Big], \qquad \frac{1}{n^{3/2}} \sum_{a} \E \Big[ \ga_{a\bar a} \gb_{aa}\gb_{\bar a\bar a}\Big],\label{third_diagonal_line1}\\
	&\frac{1}{n^{3/2}} \sum_{a}\E \Big[ \wh{ \F_t^{z_2}} G^{z_1}_{a \bar a} G^{z_1}_{a \bar a}G^{z_1}_{a \bar a} \Big], \quad 
	\frac{1}{n^{3/2}} \sum_{a} \E \Big[ \ga_{a\bar a}\gb_{a\bar a} \gb_{a\bar a} \Big], \quad 
	\frac{1}{n^{3/2}} \sum_{a} \E \Big[ \ga_{a\bar a} \gb_{\bar a a} \gb_{\bar a a} \Big]\label{third_diagonal_line2}.
\end{align}

To estimate the terms in (\ref{third_diagonal_line1}), we state the following lemma
asserting that the diagonal elements of $\Im G^z$ are essentially bounded by their average $\langle \Im G^z\rangle$.
Its fairly routine  proof relies on the complete delocalization of the eigenvectors in (\ref{eigenvector}); the details are deferred to 
 Appendix~\ref{sec:some_lemma}.
\begin{lemma}\label{lemma:aaBB}
	Fix small $\epsilon_1,\epsilon_2>0$. For any $-Cn^{-1/2} \leq |z|-1 \leq c$ and any $n^{-1+\epsilon_1} \leq \eta \leq n^{-3/4-\epsilon_2}$, the following estimates
	\begin{align}\label{aaBB}
		\Im G^z_{aa}(\ii \eta)=O_\prec\big(\Im \<G(\ii \eta)\>+\eta\big), \qquad \Im G^z_{\bar a \bar a}(\ii \eta)=O_\prec\big(\Im \<G(\ii \eta)\>+\eta\big). 
	\end{align}
hold true for any $a \in \llbracket 1,n \rrbracket$.
\end{lemma}
Therefore, using (\ref{aaBB}) the first term in (\ref{third_diagonal_line1}) is bounded by
\begin{align}
	\Big|\frac{1}{n^{3/2}}\sum_{a}\E \Big[ \wh{ \F_t^{z_2}} G^{z_1}_{aa} G^{z_1}_{\bar a \bar a}G^{z_1}_{a \bar a} \Big]\Big| \prec \frac{1}{\sqrt{n}}
	 \E [|	\<G^{z_1}\>|^2]=O_{\prec}(n^{-3/2+2\epsilon}),
\end{align}
where we also used \eqref{F}, \eqref{Gvv} and Proposition \ref{lemma2} for $k=2$. Using the $L^{1}$ norm of $\Delta f$ in (\ref{deltaf}), we~have
\begin{align}\label{estimate2}
	\frac{1}{n^{3/2}} \iint |\Delta f(z_1) \Delta f(z_2)| \Big|\sum_{a}\E \left[ \wh{ \F_t^{z_2}} G^{z_1}_{aa} G^{z_1}_{\bar a \bar a}G^{z_1}_{a \bar a} \Big] \right| \dd^2 z_1 \dd^2 z_1=O_{\prec}(n^{-1/2+2\epsilon}).
\end{align}
A similar bound applies to the second term in (\ref{third_diagonal_line1}). Note that in these
two terms we had two diagonal resolvent elements which are small.

In the remaining terms in (\ref{third_diagonal_line2}) all factors are off-diagonal Green function entries that are large, so we need to
use further cumulant expansion and again identify the leading term plus 
the improved expectation estimates of resolvents from Proposition \ref{lemma2}.  The result is summarized in
the following lemma which will be proven 
 in Appendix~\ref{sec:some_lemma}. Notice that the estimates below are much better than the naive ones directly obtained from the local law in~(\ref{localg})--\eqref{localm}.

\begin{lemma}\label{lemma4}
	For any $z_1,z_2 \in \mathrm{supp}(f^{-}_1) \cup \mathrm{supp}(f^{+}_2)$ and $\eta=n^{-1+\epsilon}$, we have
	\begin{align}
			&\frac{1}{n} \sum_{a} \E \Big[\wh{\F^{z_2}_t} \ga_{a \bar a} \ga_{a \bar a}  \ga_{a \bar a} \Big]=\OO_\prec(n^{-1/2}),\quad \frac{1}{n} \sum_{a} \E \Big[ \ga_{a \bar a} \gb_{ a \bar a} \gb_{ a \bar a} \Big]=\m^{z_1} ({\m^{z_2}})^2 +\OO_\prec(n^{-1/2-\epsilon}),\label{estimate4}\\
		&\frac{1}{n} \sum_{a} \E \Big[ \ga_{a \bar a} \gb_{\bar a a} \gb_{\bar a a} \Big]=\m^{z_1} (\overline{\m^{z_2}})^2 +\OO_\prec(n^{-1/2}).\label{estimate5}
	\end{align} 
Further if $|z_1-z_2| \geq n^{-\gamma}$ for any $\gamma\in (0, \frac{1}{2})$, the above error terms $\OO_\prec(n^{-1/2})$ can be improved to
	\begin{align}
		&\frac{1}{n} \sum_{a} \E \Big[\wh{\F^{z_2}_t} \ga_{a \bar a} \ga_{a \bar a}  \ga_{a \bar a} \Big]=
		\OO_\prec\Big( \frac{n^{-1/2-\epsilon}}{|z_1-z_2|}\Big) \le
		\OO_\prec(n^{-1/2-\epsilon+\gamma}), \label{estimate7}\\
		& \frac{1}{n} \sum_{a} \E \Big[ \ga_{a \bar a} \gb_{\bar a a} \gb_{\bar a a} \Big]= 
		\m^{z_1} (\overline{\m^{z_2}})^2 +\OO_\prec\Big( \frac{n^{-1/2-\epsilon}}{|z_1-z_2|}\Big) 
		=
		 \m^{z_1} (\overline{\m^{z_2}})^2 +\OO_\prec(n^{-1/2-\epsilon+\gamma}). \label{estimate6}
	\end{align}
The above estimates in (\ref{estimate4})-(\ref{estimate6}) also hold true with $a$ interchanged with $\bar a$.
\end{lemma}
 
Notice that (\ref{estimate4})-(\ref{estimate5})  
are already enough (with $\OO_\prec(n^{-1/2-\epsilon})$) or just barely not enough 
(with $\OO_\prec(n^{-1/2})$) 
to prove (\ref{third_fourth_result}) for $p+q+1=3$. 
Moreover, in cases when (\ref{estimate4})-(\ref{estimate5})
gives only $\OO_\prec(n^{-1/2})$
we gain a little extra smallness from the $|z-z'|$-decorrelation effect in the improved estimates \eqref{estimate7}--\eqref{estimate6}.
This gain  will rely on a certain self-consistent equation for the resolvent products with 
a stability factor $1-\m^{z_1}\overline{\m^{z_2}}$. Simple calculus using (\ref{Mmatrix_0}) and (\ref{rho}) shows that, for any $z_1,z_2 \in \mathrm{supp}(f^{-}_1) \cup \mathrm{supp}(f^{+}_2)$ and $|z_1-z_2| \geq n^{-\gamma}$ with $\gamma\in (0,1/2)$,
\begin{align}\label{stable_0}
|1-\m^{z_1} \overline{\m^{z_2}}| \gtrsim |1-z_1 \overline{z_2}| \gtrsim |z_1-z_2| \geq n^{-\gamma}
\end{align} 
\ie the self-consistent equation  is rather stable for $z_1$ and $z_2$ being far away.

\begin{remark}\label{remark:aa}
Note that the estimates for the third order terms with $a=\ud{B}$ given in~\eqref{estimate2}--\eqref{estimate5}
are slightly depending on how many (small) diagonal and (large) off-diagonal elements it contains. 
First,  recall that diagonal resolvent elements $G_{aa}$ or $G_{\bar a \bar a}$ are generically
smaller than the off-diagonal ones $G_{a \bar a}$ and $G_{\bar a a}$ since the corresponding leading
deterministic term for  $G_{aa}$ is $m\sim \eta/(|z|^2-1)\lesssim n^{-1/2+\epsilon}$, while for
$G_{a \bar a}$ we have $\m \sim 1$. This explains the lack of the deterministic leading term in~\eqref{estimate2}.

Second, notice that the error term  in the second estimate in~\eqref{estimate4} is slightly better than 
in \eqref{estimate5} due to the location of the $a$ and $\bar a$ indices: the key point is 
that  once a row index can only be paired with conjugated column indices (instead of the identical ones),
 then we
pick up a small diagonal leading term $m$ along the cumulant expansion, see~\eqref{aa}--\eqref{aaexpl} 
in Appendix~\ref{appendix_A}. 
 For example,
 in the second estimate in~\eqref{estimate4}
 the row index $a$ of the first resolvent is always paired with the column indices $\bar a$
 of the other two, while in~\eqref{estimate5} no such pairing is possible irrespective of
  which index we try to expand. However, this gain may disappear when a factor $\wh{\F}$
 is present (compare the errors in the first and second estimates in~\eqref{estimate4}),  since  in this case
 along the  cumulant expansion one has to differentiate $\wh{\F}$ as well and we may not
 pick up a small diagonal $G$ term (the leading deterministic term vanishes when $\wh{\F}$ is present
 since $\E \wh{\F}=0$). 

 Finally, in the estimates~\eqref{estimate7}-\eqref{estimate6} we remedy these problems 
 and still obtain the additional $n^{-\epsilon}$
 factor, but using the additional $|z_1-z_2|$-decorrelation effect, i.e  these bounds hold 
 only for  $|z_1-z_2|\ge n^{-\gamma}$ and at the irrelevant expense of $n^{\gamma}$.

\end{remark}

Now we are ready to estimate the terms in (\ref{third_diagonal_line2}). Choosing $\gamma=\epsilon/2$, using (\ref{estimate5}) for the regime $|z_1-z_2| \leq n^{-\gamma}$ and (\ref{estimate6}) for the remaining part $|z_1-z_2| \geq n^{-\gamma}$ together with (\ref{deltaf}), the last term in (\ref{third_diagonal_line2}) can be bounded by
\begin{align}\label{diagonal_goal}
	& \iint |\Delta f(z_1)||\Delta f(z_2) |  \left|\frac{1}{n^{3/2}} \sum_{a} \E \Big[ \ga_{a \bar a}\gb_{\bar a a} \gb_{ \bar a a} \Big] -\frac{1}{\sqrt{n}}\m^{z_1} (\overline{\m^{z_2}})^2 \right|\dd^2 z_1 \dd^2 z_1
	=O_\prec(n^{-\epsilon/4}),
\end{align}
where the deterministic function $\m^{z_1} (\overline{\m^{z_2}})^2$ satisfies the integral condition (\ref{delta_int}) using (\ref{pq_m_int_0}). 
The remaining two terms in (\ref{third_diagonal_line2}) can be estimated similarly as in (\ref{diagonal_goal}) using instead (\ref{estimate4}) and (\ref{estimate7}).
Hence, from the estimates as in (\ref{estimate2}) and (\ref{diagonal_goal}), there exists deterministic bounded functions, denoted by $\mathcal{M}^{(2)}_{p+1,q}(z_1,z_2)$ with $p+q+1=3$ satisfying the integral condition in (\ref{delta_int}) such that
\begin{align}\label{result_aa}
	&\iint |\Delta f(z_1)| |\Delta f(z_2)| \left|\sum_{p+q+1=3} \Big( \mathcal{K}^{z_1,z_2}_{p+1,q} \Big|_{a=\ud{B}} -\mathcal{M}^{(2)}_{p+1,q}(z_1,z_2)  \Big)\right| \dd^2 z_1 \dd^2 z_2=O_\prec(n^{-\epsilon/4}).
\end{align}

\subsection{Fourth order terms with $p+q+1=4$  in~\eqref{third_fourth_result}}\label{subsec:part3}

We will first look at the fourth order term with $p=1,q=2$, \ie $\K^{z_1,z_2}_{2,2}$ given in (\ref{pq_order_terms}).  By direct computations, $\K^{z_1,z_2}_{2,2}$ is a linear combination of the following matched terms (plus their versions when interchanging $z_1$ with $z_2$ or interchanging $a$ with $B$)
\begin{align}
	&\frac{1}{n^2} \sum_{a,B}\E\Big[ \wh{\F^{z_2}_t} (\ga_{BB}\ga_{aa})^2\Big] , \qquad \frac{1}{n^2} \sum_{a, B}\E\Big[ \ga_{aa} \ga_{BB} \gb_{aa} \gb_{BB}\Big],\label{fourth_line1}\\
	&\frac{1}{n^2} \sum_{a, B}\E\Big[ \wh{ \F^{z_2}_t} \ga_{aa}\ga_{BB}\ga_{aB}\ga_{Ba}\Big], \quad 	\frac{1}{n^2} \sum_{a, B}\E\Big[ \ga_{aB} \gb_{Ba} \gb_{aa} \gb_{BB}\Big], \label{fourth_line2}\\
	&\frac{1}{n^2} \sum_{a, B}\E\Big[ \ga_{aB} \ga_{aB} \gb_{Ba} \gb_{Ba}\Big]\label{fourth_line3}.
\end{align}
whose naive sizes are given by $O_\prec(n^{-4\epsilon})$ from the local law in (\ref{localg}).

Using \eqref{F}, (\ref{aaBB}) and Proposition \ref{lemma2} with $k=4$, the first term in (\ref{fourth_line1}) is bounded by
\begin{align}\label{naive_2}
	\Big|\frac{1}{n^2} \sum_{a,B}\E\Big[ \wh{\F^{z_2}_t} (\ga_{BB}\ga_{aa})^2\Big]\Big| \prec \E\Big[\Big( \<G^{z_1}(\ii \eta)\>+\eta \Big)^4\Big]= \OO_\prec(n^{-1-2\epsilon}).
\end{align}
The same upper bound applies to the second term in (\ref{fourth_line1}) reducing it to the first
by using the Cauchy-Schwarz inequality. Moreover, using \eqref{F},  (\ref{aaBB}), the {\it Ward identity}
\begin{align}\label{ward}
	\sum_{\mathfrak{v}=1}^{2n}|G^z_{\mathfrak{u}\mathfrak{v}}|^2=\frac{\Im G^z_{\mathfrak{u}\mathfrak{u}}}{\eta}, \qquad \quad\mathfrak{u}\in \llbracket 1,2n \rrbracket
\end{align}
and Proposition \ref{lemma2} with $k=2$, the first term in (\ref{fourth_line2}) can be bounded by
\begin{align}\label{naive3}
	\Big| \frac{1}{n^2} \sum_{a, B}\E\Big[ \wh{ \F^{z_2}_t} \ga_{aa}\ga_{BB}\ga_{aB}\ga_{Ba}\Big] \Big| \prec \frac{1}{n\eta}\E\Big[\Big( \<G^{z_1}(\ii \eta)\>+\eta \Big)^3\Big]=\OO_\prec(n^{-1-\epsilon}),
\end{align}
and a similar upper bound also applies to the second term in (\ref{fourth_line2}).

 The most delicate is the last term in (\ref{fourth_line3}). As explained in Remark \ref{remark:aa}, the terms in (\ref{naive_2}) and (\ref{naive3}) are smaller than (\ref{fourth_line3}) because the leading term of $G_{aa}$ (or $G_{BB}$) is $m^{z} \sim \sqrt{n}\eta$ is small, while the leading term of $G_{aB}$ with $a=\ud{B}$ is $\m^{z} \sim 1$.
A precise estimate for the last term in (\ref{fourth_line3}) is stated below which will be proved in Appendix~\ref{sec:some_lemma} 
 using again that a little extra smallness is gained from the $|z-z'|$-decorrelation effect, 
 as in Lemma~\ref{lemma4}.

\begin{lemma}\label{lemma5}
	For any $z_1,z_2 \in \mathrm{supp}(f^{-}_1) \cup \mathrm{supp}(f^{+}_2)$ and $\eta=n^{-1+\epsilon}$, we have
	\begin{align}
		&\Big|\frac{1}{n^2} \sum_{a, B}\E\Big[ \ga_{aB} \ga_{aB} \gb_{Ba} \gb_{Ba}\Big]\Big| =\OO_\prec(n^{-1}).\label{estimate_1}
	\end{align}
Further, 
 if $|z_1-z_2|\ge n^{-\gamma}$ for some $\gamma\in (0, \frac{1}{2})$, then we have
\begin{align}
	&\frac{1}{n^2} \sum_{a, B}\E\Big[ \ga_{aB} \ga_{aB} \gb_{Ba} \gb_{Ba}\Big] 
	= \frac{1}{n} (\m^{z_1})^2 (\overline{\m^{z_2}})^2+\OO_\prec(n^{-1-\epsilon+\gamma}),\label{estimate_2}
\end{align} 
that identifies the leading term in the above estimate when $\gamma$ is small.

\end{lemma}

Note that from (\ref{pq_m_int_0}) the leading deterministic term in (\ref{estimate_2}) satisfies 
the integral condition in \eqref{delta_int}. Therefore, we conclude from Lemma \ref{lemma5} with $\gamma=\epsilon/2$ and (\ref{deltaf}) that
	\begin{align}\label{result_fourth}
		&\iint |\Delta f(z_1)|| \Delta f(z_2)| \Big| \mathcal{K}^{z_1,z_2}_{2,2}-\frac{C}{n}(\m^{z_1})^2 (\overline{\m^{z_2}})^2  \Big| \dd^2 z_1 \dd^2 z_2=O_\prec(n^{-\epsilon/4}),
	\end{align}
for some numerical constant $C \in \R$.

All the other fourth order terms $\mathcal{K}^{z_1,z_2}_{p+1,q}$ with $p+q+1=4$, $q\neq 2$ can be 
estimated similarly as the third order terms in Section \ref{subsec:part1}-\ref{subsec:part2}. Since we gain an additional $n^{-1/2}$ prefactor from the fourth order cumulants, we only sketch the proof that is much easier. For $p+q+1=4,~q\neq 2$, using (\ref{pq_number}) and 
Definition \ref{def:unmatch_form} in Appendix~\ref{appendix_A}, the restricted summations with $a \neq \ud{B}$, similarly 
to (\ref{fourG}) are also unmatched terms which can be bounded by $O_\prec(n^{-3/2})$ using (\ref{P_d_o_before}). 
The remaining summations with $a=\ud {B}$ yield some deterministic functions, \ie products of $\m^{z_1}$ and 
$\m^{z_2}$ up to an error term $O_\prec(n^{-1-\epsilon})$ using the local law (\ref{localg}) trivially. Notice that these fourth order 
terms satisfy the index assignment condition (\ref{pq_number}) with $p+q+1=4$, $q\neq 2$. Setting $a=\ud{B}$, 
this implies that at least one $\m^{z}$ factor (from $\gz_{aB}$) cannot be paired with $\overline{\m^{z}}$
 (from $\gz_{Ba}$). Hence these deterministic functions, denoted by $\mathcal{M}_{p+1,q}(z_1,z_2)$ with $p+q+1=4,q\neq 2$ 
 vanish after the $z$-integrations using (\ref{pq_m_int_0}). Therefore using the $L^1$ norm of $\Delta f$ 
 in (\ref{deltaf}), we obtain that
\begin{align}\label{result_four_other}
	\iint |\Delta f(z_1)||\Delta f(z_2)| \left|\sum_{p+q+1=4,q\neq 2} \Big(\mathcal{K}^{z_1,z_2}_{p+1,q}  -\mathcal{M}_{p+1,q}(z_1,z_2) \Big) \right| \dd^2 z_1 \dd^2 z_2 =O_\prec(n^{-\epsilon}),
\end{align}
with $\mathcal{M}_{p+1,q}(z_1,z_2)$ satisfying the integral condition in (\ref{delta_int}).

\subsection{Fifth order terms with $p+q+1=5$  in~\eqref{third_fourth_result}}
Recall that the fifth order terms $\mathcal{K}^{z_1,z_2}_{p+1,q} $ in (\ref{pq_order_terms}) are linear combinations
 of terms in (\ref{some_form_F}) with $p+q+1=5$ and satisfying the assignment condition (\ref{pq_number}). Since
both  indices $a$ and $B$ are assigned five times as the row/column index of Green function entries, 
 the restricted summations with $a\neq \ud{B}$ are unmatched 
terms with an additional factor $n^{-1/2}$ gaining from the fifth order cumulants. Using (\ref{P_d_o_before}),
 the restricted summations with $a\neq \ud{B}$ can be bounded by $O_\prec(n^{-2})$. In addition, the 
 remaining summations with $a=\ud{B}$ can be bounded by $O_\prec(n^{-3/2})$ using the local law in
  \eqref{localg}--\eqref{localm} naively. Therefore, using the $L^1$ bound for $\Delta f$ in (\ref{deltaf}), 
we have the simple estimate
\begin{align}\label{fifth_estimate}
	\iint |\Delta f(z_1)| |\Delta f(z_2)| \Big| \sum_{p+q+1=5} \mathcal{K}^{z_1,z_2}_{p+1,q} \Big|  \dd^2 z_1 \dd^2 z_2 =O_\prec(n^{-1/2}).
\end{align}

\bigskip

To sum up  the above four subsections, we have obtained the precise estimates for 
$\mathcal{K}^{z_1,z_2}_{p+1,q}$ in (\ref{third_fourth_result}) with $p+q+1=3,4,5$ respectively, 
\ie (\ref{result_third}) for the third order terms with $a\neq \ud{B}$, (\ref{result_aa}) for the 
third order terms with $a=\ud{B}$, (\ref{result_fourth})-(\ref{result_four_other}) for the fourth 
order terms, as well as (\ref{fifth_estimate}) for the fifth order terms. Hence we have 
concluded the proof of Proposition \ref{key_lemma}.

\section{Weakly correlated Dyson Brownian motions at the cusp}
\label{sec:DBM}

Consider the matrix flow
\begin{equation}
\label{eq:diffmat}
\mathrm{d}X_t=\frac{\mathrm{d} B_t}{\sqrt{n}}, \qquad X_0=X_{\mathrm{in}},
\end{equation}
for some i.i.d. matrix $X_{\mathrm{in}}$ with complex entries as initial condition.
 Here $B_t\in C^{n\times n}$ is a \emph{matrix valued standard complex Brownian 
motion}, i.e. $(B_t)_{ij}$ are a family of i.i.d. standard complex Brownian motions.  
It is easy to see that
\[
X_t\stackrel{\mathrm{d}}{=}X_0+\sqrt{t} U,
\]
with $U$ being a complex Ginibre matrix independent of $X_0$.
We denote the singular values of $X_t-z$ by $\lambda_i^z(t)$, indexed in increasing order,
 and let $\mathbf{u}_i^z(t), \mathbf{v}_i^z(t)$  be
the corresponding left and right singular vectors normalized so that $\lVert \mathbf{u}_i^z(t)\rVert^2=
\lVert \mathbf{v}_i^z(t)\rVert^2=1/2$.

 
 The singular values $\lambda_i^{z_l}(t)$ of $X_t-z_l$, for $l=1,2$, are the solution of the following \emph{Dyson Brownian motion (DBM)} (see~\cite[Appendix B]{CES19}:
\begin{equation}
\label{eq:lamDBM}
\mathrm{d}\lambda_i^{z_l}(t)=\frac{\mathrm{d} b_i^{z_l}(t)}{\sqrt{2n}}+\frac{1}{2n}\sum_{j\ne i}\frac{1}{\lambda_i^{z_l}(t)-\lambda_j^{z_l}(t)}\, \mathrm{d}t.
\end{equation}
Here $\lambda_{-i}^{z_l}=-\lambda_i^{z_l}$, for $i\in [[1,n]]$, and a similar symmetry holds for the driving Brownian 
motions, i.e. $b_{-i}^{z_l}(t)=-b_{i}^{z_l}(t)$, which ensures  that $\lambda_{-i}^{z_l}(t)=-\lambda_i^{z_l}(t)$ 
holds not only initially but at any later time as well. The driving martingales in \eqref{eq:lamDBM} are Brownian motions only for fixed $z_l$, but not jointly for different $z_1,z_2$; more precisely, their correlation is given by
\begin{equation}
\label{eq:DBMcor}
\mathrm{d}\big[b_i^{z_1}(t), b_j^{z_2}(s)\big]=4\mathrm{Re}\big[\langle \mathbf{u}_i^{z_1}, \mathbf{u}_j^{z_2}\rangle\langle  \mathbf{v}_j^{z_2}, \mathbf{v}_i^{z_1}\rangle \big]\, \mathrm{d} t.
\end{equation}
Note that in this section we use $z_1,z_2$ instead of $z,z'$ as in some previous sections of this paper. 
We made this choice so that the notation is the same as in~\cite{CES19,real_CLT},
 to which we often refer within this section.

 \begin{remark}
In \eqref{eq:diffmat} we evolve the initial condition $X_0=X_{\mathrm{in}}$ via the Brownian motion flow.
 Another possible choice would have been to choose again the Ornstein-Uhlenbeck flow (OU) as in~\eqref{flow}:
\begin{equation}
\mathrm{d}X_t=\frac{\mathrm{d} B_t}{\sqrt{n}}-\frac{1}{2}X_t\mathrm{d}t, \qquad X_0=X_{\mathrm{in}},
\end{equation}
that is often used in DBM analysis with the advantage that the first and second moments of $X_t$
are unchanged. In particular, the self consistent density of states of the Hermitization $H^z$ remains invariant. 
However, this choice would have implied the DBM flow
\begin{equation}\label{OUDBM}
\mathrm{d}\lambda_i^{z_l}(t)=\frac{\mathrm{d} b_i^{z_l}(t)}{\sqrt{2n}}+\frac{1}{2n}\sum_{j\ne i}\frac{1}{\lambda_i^{z_l}(t)-\lambda_j^{z_l}(t)}\, \mathrm{d}t-
2\lambda_i^{z_l}(t)\mathrm{Re}[\langle\mathbf{u}_i^{z_l}(t),\mathbf{v}_i^{z_l}(t) \rangle]\, \mathrm{d}t.
\end{equation}
 This flow is harder to analyze than \eqref{eq:lamDBM}, which is obtained from~\eqref{eq:diffmat},
  since the additional last term in~\eqref{OUDBM}
 requires information about the singular vector overlap
   $\langle\mathbf{u}_i^z(t),\mathbf{v}_i^z(t) \rangle$ as well. 
   For the same reason in~\cite[Sections 6--8]{CEKS19} 
 we chose an evolution as in \eqref{eq:diffmat} rather than the OU-like flow~\cite[Eq. (3.2)]{CEKS19}, which would have produced a DBM-like flow depending on eigenvectors as well.  The price for this convenience
 is that we had to analyse how the self-consistent density of the singular values evolve.
 \end{remark}

 Next, consider two independent complex Ginibre matrices $X^{(l)}$, and denote by $X_t^{(l)}$ their evolution under the flow
\begin{equation}
\label{eq:Gindiff}
\mathrm{d}X_t^{(l)}=\frac{\mathrm{d}B_t^{(l)}}{\sqrt{n}},\qquad X_0^{(l)}=X^{(l)},
\end{equation}
with $B_t^{(l)}$ two independent matrix valued complex Brownian motions (defined similarly to $B_t$ in \eqref{eq:diffmat}). Let $\mu^{(l)}(t)$ be the singular values of $X_t^{(l)}-z_l$, then $\mu^{(l)}(t)$ evolve as 
\begin{equation}
\label{eq:muDBM}
\mathrm{d}\mu_i^{(l)}(t)=\frac{\mathrm{d} \beta_i^{(l)}(t)}{\sqrt{2n}}+\frac{1}{2n}\sum_{j\ne i}\frac{1}{\mu_i^{(l)}(t)-\mu_j^{(l)}(t)}\, \mathrm{d}t.
\end{equation}
 In particular,  the family $\{ \beta_i^{(l)}:\, i\in [[1,n]],\, l=1,2\}$ is a $2n$-dimensional   standard Brownian motion 
and $\beta_{-i}^{(l)}=-\beta_i^{(l)}$. 

Here we consider only the case when $X$ has complex entries, 
the proof in the real case follows similar steps but it is  technically more involved, since the product of
 singular vector overlaps 
 $\langle \mathbf{u}_i^z,\mathbf{u}_j^{\overline{z}}\rangle \langle \mathbf{v}_j^{\overline{z}},\mathbf{v}_i^z\rangle$
  influences the dynamics in \eqref{eq:lamDBM} in a non--trivial way even for fixed $z$. For the sake of brevity and
   clarity of the presentation we will not say more here about the real case, but the proof of Theorem~\ref{theo:ind}
    below is completely analogous once all the references to~\cite{CES19} are replaced with the 
    corresponding version in~\cite{real_CLT}. 

The main result of this section is  the following theorem, asserting that the small singular values of $X_t-z_1$,
and $X_t-z_2$,
are very close to two independent processes 
 if $t$ is somewhat large. In particular, it shows the 
asymptotical independence of  these singular values, which will then
 readily imply Proposition~\ref{lemma_new_1}
 and Proposition~\ref{lemma_new_2}.
\begin{theorem}
\label{theo:ind}
Fix $\tau,C>0$, pick $z_1,z_2\in\mathbb{C}$ such that $-Cn^{-1/2+\tau}\le |z_l|^2-1\le C n^{-1/2+\tau}$, and let $\lambda_i^{z_l}(t)$, $\mu_i^{(l)}(t)$ be the solutions of \eqref{eq:lamDBM} and \eqref{eq:muDBM}, respectively. Furthermore, for any small $\omega_E\ge 10 \omega_c\ge 10 \omega_1>0$, assume that 
\begin{equation}\label{ovass}
|\langle \mathbf{u}_i^{z_1}(t),\mathbf{u}_j^{z_2}(t)\rangle|^2+|\langle \mathbf{v}_i^{z_1}(t),\mathbf{v}_j^{z_2}(t)\rangle|^2\le n^{-\omega_E}, \qquad 
\mbox{ $|i|,|j|\le n^{\omega_c}$, \quad $0\le t\le t_1$},
\end{equation}
 with $t_1:=n^{-1/2+\omega_1}$. 
 Then there exist $\omega,\widehat{\omega}>0$, with $\widehat{\omega}\le\omega/10\le \omega_1/100$, such that
\begin{equation}
\label{eq:mainb}
\big|\lambda_i^{z_l}(t_1)-\mu_i^{(l)}(t_1)\big|\le n^{-3/4-\omega}, \qquad |i|\le n^{\widehat{\omega}}, \, l\in [2],
\end{equation}
with very high probability (in the joint probability space of $\lambda_i^{z_1},\lambda_i^{z_2}$).
\end{theorem}

\begin{remark}
We stated this result for $-Cn^{-1/2+\tau}\le |z_l|^2-1\le C n^{-1/2+\tau}$ since the analysis in~\cite[Sections 6--8]{CEKS19} is performed in an analogous regime; however a similar proof also holds for $-Cn^{-c}\le |z_l |^2-1\le C n^{-c}$, for some small fixed $c>0$. We omit these details for brevity. The same remark applies to Propostion~\ref{lemma_new_1}.
\end{remark}

We now conclude the proof of Proposition~\ref{lemma_new_1} and then present the proof of Theorem~\ref{theo:ind}.

\begin{proof}[Proof of Proposition~\ref{lemma_new_1}]
For consistency of notation within this section we use $z_1=z$, $z_2=z'$, with $z,z'$ from the statement of Proposition~\ref{lemma_new_1}.

We now consider the OU flow
\begin{equation}
\mathrm{d}\widehat{X}_t=\frac{\mathrm{d} \widehat{B}_t}{\sqrt{n}} -\frac{1}{2}\widehat{X}_t \mathrm{d} t,
 \qquad \widehat X_0=X,
\end{equation}
with initial condition $X$ being the Ginibre matrix for which we want to prove \eqref{lambdatail}, and $\widehat{B}_t$ 
being a standard complex matrix valued Brownian motion defined as
 $B_t$ in \eqref{eq:diffmat}. It is easy to see that
\[
\widehat{X}_t\stackrel{\mathrm{d}}{=}e^{-t/2}X+\sqrt{1-e^{-t}}U,
\]
with $U$ being a complex Ginibre matrix independent of $X$. In particular, note that if $\widehat{X}_0=X$ 
is a Ginibre matrix so is $\widehat{X}_t$ for any $t\ge0$. Next, we define $\check{X}_{t_1}:=e^{-t_1/2}X$, with
a fixed  $t_1$ obtained  in Theorem~\ref{theo:ind}, and obtain
\begin{equation}
\widehat{X}_{t_1}\stackrel{\mathrm{d}}{=}\check{X}_{t_1}+\sqrt{ct_1}U,
\end{equation}
with $c=c(t_1)=(1-e^{-t_1})/t_1=1+O(t_1)$. Then, considering the flow \eqref{eq:diffmat} with initial condition $X_0=\check{X}_{t_1}$, we get
\begin{equation}
\widehat{X}_{t_1}\stackrel{\mathrm{d}}{=}X_{ct_1}.
\end{equation}
In particular, since $\widehat{X}_{t_1}$ is distributed as a Ginibre matrix we conclude that
\begin{equation}
\label{eq:indistr}
\big(\lambda_i^{z_1},\lambda_i^{z_2}\big)_{i\in [[1,n]]}\stackrel{\mathrm{d}}{=}\big(\lambda_i^{z_1}(ct_1),\lambda_i^{z_2}(ct_1)\big)_{i\in [[1,n]]},
\end{equation}
with $\lambda_i^{z_l}$ and  $\lambda_i^{z_l}(ct_1)$ being the singular values 
of $X-z_l$ and $X_{ct_1}-z_l$, respectively, where $X_{ct_1}$ is the flow given by (\ref{eq:diffmat}) with the initial Ginibre condition at time $ct_1$. We remark that \eqref{eq:indistr} holds only 
at the precise time $ct_1$. In particular, by \eqref{eq:indistr} it follows that
\begin{equation}
\label{eq:s1}
\P^{\mathrm{Gin}}\Big( |\lambda_1^{z_1}| \leq n^{-3/4-\alpha}, ~ |\lambda_1^{z_2}| \leq n^{-3/4-\alpha}\Big)=\P^{\mathrm{Gin}}\Big( |\lambda_1^{z_1}(ct_1)| \leq n^{-3/4-\alpha}, ~ |\lambda_1^{z_2}(ct_1)| \leq n^{-3/4-\alpha}\Big).
\end{equation}

We now apply Theorem~\ref{theo:ind} to show the asymptotical independence of the singular values 
$\lambda_i^{z_1}(ct_1)$, $\lambda_i^{z_2}(ct_1)$ for small indices. We thus start showing that the assumptions of Theorem~\ref{theo:ind} are fulfilled. By~\cite[Theorem 5.2]{CES19}, for any $|z_1|+|z_2|\le C$, we have,
\begin{equation}
\label{eq:evbtim}
|\langle \mathbf{u}_i^{z_1}(t),\mathbf{u}_j^{z_2}(t)\rangle|^2+|\langle \mathbf{v}_i^{z_1}(t),\mathbf{v}_j^{z_2}(t)\rangle|^2\prec n^{-\omega_E},  \qquad \quad |i|,|j|\le n^{\omega_c},
\end{equation}
for some small fixed $\omega_E\ge 10\omega_c>0$, simultaneously in $t\in [0, t_1]$.
  The proof of this bound is exactly the same as~\cite[Lemma 7.9]{CES19}, since $|z_1-z_2|\ge n^{-\gamma}$ and the local law in~\cite[Theorem 5.2]{CES19} holds uniformly in $|z_1|+|z_2|\le C$, for some $C>0$, in particular it holds also in the cusp regime. The fact that \eqref{eq:evbtim} holds simultaneously in $t$ follows by a standard grid argument.

Finally, by applying \eqref{eq:mainb} in the first and last inequality 
 (\eqref{eq:mainb} also holds at time $ct_1$ instead of $t_1$  since $c\approx 1$),  we conclude
\begin{equation}
\label{eq:s2}
\begin{split}
&\P^{\mathrm{Gin}}\Big( |\lambda_1^{z_1}(ct_1)| \leq n^{-3/4-\alpha}, ~ |\lambda_1^{z_2}(ct_1)| \leq n^{-3/4-\alpha}\Big) \\
&\quad\le \P^{\mathrm{Gin}}\Big( |\mu_1^{(1)}(ct_1)| \leq n^{-3/4-\alpha}+n^{-3/4-\omega}, ~|\mu_1^{(2)}(ct_1)| \leq n^{-3/4-\alpha}+n^{-3/4-\omega} \Big) +n^{-100}\\
&\quad= \P^{\mathrm{Gin}}\Big( |\mu_1^{(1)}(ct_1)| \leq n^{-3/4-\alpha}+n^{-3/4-\omega}\Big)^2+n^{-100} \\
&\quad \le \P^{\mathrm{Gin}}\Big( |\lambda_1^z(ct_1)| \leq n^{-3/4-\alpha}+2n^{-3/4-\omega}\Big)^2+n^{-100} \\
&\quad = \P^{\mathrm{Gin}}\Big( |\lambda_1^z| \leq n^{-3/4-\alpha}+2n^{-3/4-\omega}\Big)^2+n^{-100},
\end{split}
\end{equation}
where the irrelevant error term $n^{-100}$ comes from the fact
that both~(\ref{eq:evbtim}) and~(\ref{eq:mainb}) hold true with a very high probability, 
say larger than $1-n^{-100}$. The first equality in \eqref{eq:s2} follows from $\mu_i^{(1)}(t_1)$, $\mu_i^{(2)}(t_1)$ 
 being fully independent, the second equality follows from \eqref{eq:indistr}. This bound, together 
with \eqref{eq:s1} and  choosing $\alpha<\omega$ 
concludes the proof of Proposition~\ref{lemma_new_1}.
\end{proof}

\begin{proof}[Proof of Theorem~\ref{theo:ind}]

The proof of this theorem is fairly similar to the proof of~\cite[Lemmas 7.6--7.7]{CES19} (which are proven using~\cite[Proposition 7.14]{CES19}) but in the cusp ($|z|\approx 1$) instead of the bulk regime ($|z|<1$). For the proof of~\cite[Proposition 7.14]{CES19} we relied on the homogenization theory of the Dyson Brownian motion (DBM) developed in~\cite{BEYY16,LSY12,B21} (see also~\cite{CL19} for its adaptation to singular values) for the analysis of a single DBM in the Hermitian setting. The main novelty in \cite[Proposition 7.14]{CES19} was to extend this idea to analyze 
several weakly dependent DBMs of the form~\eqref{eq:lamDBM}  driven by correlated  Brownian motions
with correlation given in \eqref{eq:DBMcor}. As was already mentioned above, the analysis in~\cite{CES19} was in the bulk regime,  we now explain how to extend this approach to weakly correlated DBMs in the cusp regime. 
 Instead of using homogenization theory, as in the bulk regime \cite{CES19}, we now
  rely on the strong local ergodicity of the DBM in the cusp regime of usual Hermitian matrices
   (i.e. no $z$-dependence), which was proven in~\cite[Proposition 7.1]{CEKS19} using energy methods. In particular, 
   the adaptation of the analysis in~\cite{CEKS19} to the weakly correlated case follows similar steps
    to the adaptation of the Hermitian homogenization theory to~\cite[Section 7]{CES19}.
     To avoid tedious uninformative computations and keep the presentation short, we only explain the 
     main steps of the proof pointing out the minor differences compared to~\cite{CEKS19,CES19}.

The high probability bound in \eqref{eq:mainb} is proven following several steps that we now explain:

\begin{enumerate}

\item[(1)] In order to prove bounds like \eqref{eq:mainb} for standard Hermitian DBMs,
one uses the coupling method that was first introduced in \cite{BEYY16} and  later
in a more convenient continuous interpolation form in~\cite{LSY12}. In this approach, one
 studies the interpolating process
\begin{equation}
\label{eq:rDBM}
\mathrm{d}r_i^{(l)}(t,\alpha)=\alpha \frac{\mathrm{d} b_i^{z_l}(t)}{\sqrt{2n}}+(1-\alpha)\frac{\mathrm{d} \beta_i^{(l)}(t)}{\sqrt{2n}}+\frac{1}{2n}\sum_{j\ne i}\frac{1}{r_i^{(l)}(t,\alpha)-r_j^{(l)}(t,\alpha)}\, \mathrm{d}t,
\end{equation}
for any $\alpha\in [0,1]$ 
with initial conditions
 $r^{(l)}(0, \alpha):=\alpha \lambda_i^{z_l}(0)+(1-\alpha)\mu_i^{(l)}(0)$. 
 Note that $r^{(l)}(t,0)=\mu^{(l)}(t)$ and $r^{(l)}(t,1)=\lambda^{z_l}(t)$ for any
  $t\ge0$, so $r^{(l)}(t, \alpha)$ indeed interpolates between the $\mu^{(l)}$ and $\lambda^{z_l}$
  processes.
  In particular, the key observation is that 
\[
 \lambda_i^{z_l}(t)-\mu_i^{(l)}(t)=\int_0^1 \partial_\alpha r_i^{(l)}(t,\alpha)\,\mathrm{d}\alpha,
\]
hence an high probability bound on $| \partial_\alpha r_i^{(l)}(t,\alpha)|$ corresponds to a bound on $| \lambda_i^{z_l}(t)-\mu_i^{(l)}(t)|$. We  remark that this interpolating process $r_i^{(l)}(t,\alpha)$
   was denoted by $z_i(t,\alpha)$ in~\cite[Eq. (6.2)]{CEKS19}, here we do 
   not use this notation to avoid any confusion with the $z$-dependence of $\lambda_i^z(t)$.

In standard Hermitian DBM analysis the driving Brownian motions in \eqref{eq:rDBM} are exactly coupled,
 i.e. we have $b^{z_l}_i =\beta^{(l)}_i$. In this case differentiating~\eqref{eq:rDBM} in $\alpha$ 
yields a differential equation for $\partial_\alpha r_i^{(l)}$ without stochastic term; this is the conventional
situation for using the coupling method. 
In the current case, however, we are interested in the correlation of singular values for different $z_1,z_2$'s, hence the driving  Brownian motions $b^{z_1}_i$ and $b^{z_2}_i$
 in \eqref{eq:lamDBM} have a non--trivial correlation as in \eqref{eq:DBMcor}.
In particular, they cannot be exactly equal to two independent Brownian motions $\beta^{(1)}_i$ and $\beta^{(2)}_i$.
So we choose $b^{z_l}_i$ and $\beta^{(l)}_i$ to be close
but not identical, yielding an additional stochastic term in the DBM for $\partial_\alpha r_i^{(l)}$
whose estimate is explained in $(4)$ below.

\medskip

\item[(2)] A fundamental input in the analysis of \eqref{eq:rDBM} is an a priori bound on the distance of  the particles $r_i^{(l)}(t,\alpha)$ from their \emph{quantiles} (the so-called \emph{rigidity estimates}). 
The proof of rigidity estimates requires two steps: (i) shape analysis of the deterministic
 density approximating the particles $r_i^{(l)}(t,\alpha)$, (ii) rigidity at time $t=0$ is 
 preserved along the flows \eqref{eq:lamDBM}, \eqref{eq:muDBM}.

\smallskip

\begin{itemize}

\item[(i)] Let $\rho^{z_l}$, with $l=1,2$, be the density of states defined in \eqref{eq:scdos} and denote by $\rho_t^{z_l}$ be its evolution along the flow in~\cite[Eqs. (2.5)--(2.6)]{LSY12}, i.e. $\rho_t^{z_l}$ is limiting density of the $r_i^{(l)}(t,\alpha)$. Note that $\rho_t^{z_l}$ does not depend on $\alpha$ since $\lambda_i^{z_l}$ and $\mu_i^{(l)}$ have the same limiting deterministic density. The analysis of the shape of $\rho_t$ is analogous (actually much easier since $\rho_t$ is independent of $\alpha$) to~\cite[Section 4]{CEKS19}. The \emph{quantiles} (classical locations) $\gamma_i^{z_l}(t)$ of $\rho_t^{z_l}$ are defined implicitly by
\begin{equation}
\int_0^{\gamma_i^{z_l}(t)}\rho_t^{z_l}(x)\,\mathrm{d}x=\frac{i}{n},
\end{equation}
for $i\in [[1,n]]$ and $\gamma_{-i}^{z_l}(t)=-\gamma_i^{z_l}(t)$ (this reflects the symmetry of the spectrum of $H^{z_l}$).

\smallskip

\item[(ii)] We now briefly explain how the rigidity bound from \eqref{rigidity} is propagated along the flow. 
Since rigidity is a high probability bound which holds for a fixed $z_l$,
 we can follow verbatim the analysis in~\cite[Section 6]{CEKS19} proving the optimal rigidity bound
\begin{equation}
\label{eq:flowrigid}
\big|r_i^{(l)}(t,\alpha)-\gamma_i^{z_l}(t)\big|\prec \frac{1}{n^{3/4}|i|^{1/4}},
\end{equation}
for indices sufficiently close to zero, $|i|\le n^{\omega_\ell}$, for some small fixed $\omega_\ell\ge 10\omega_1$,
with very high probability for any fixed $t\ge 0$ and $\alpha\in [0,1]$. A weaker bound (i.e. as in \eqref{eq:flowrigid} but without $i$-dependence) also hold for $|i|\le n^{1-\delta}$ (see~\cite[Eq. (6.101)]{CEKS19}, for some small fixed $\delta>0$.  We remark that by a simple grid argument, together with an elementary H\"older continuity, the bound in \eqref{eq:flowrigid} also holds simultaneously in $z$, $t$, and $\alpha$.

\end{itemize}

\medskip

\item[(3)] The key input in the analysis of weakly correlated DBMs in~\cite{CES19} was to show that the correlation of the driving Brownian motions in \eqref{eq:lamDBM}, which is given \eqref{eq:DBMcor}, is small for indices close to zero. This is now ensured by the assumption
\begin{equation}\label{uu}
|\langle \mathbf{u}_i^{z_1}(t),\mathbf{u}_j^{z_2}(t)\rangle|^2+|\langle \mathbf{v}_i^{z_1}(t),\mathbf{v}_j^{z_2}(t)\rangle|^2\prec n^{-\omega_E}, \qquad\quad |i|,|j|\le n^{\omega_c}.
\end{equation}
We used exactly the same assumption in~\cite{CES19} (cf.~\cite[Lemma 7.9]{CES19}), but the main difference is that
now we know~\eqref{uu} in the cusp regime, i.e. when $|z_l|\ge 1-\tau$, for some small fixed $\tau>0$.
\medskip

\item[(4)] Given all these inputs, the proof of \eqref{eq:mainb} is analogous to~\cite[Section 7]{CEKS19}. 
The only difference is that in the current case the driving Brownian motions $\mathrm{d}b_i^{z_l}$ and
$\mathrm{d}\beta_i^{(l)}$ (for $l=1,2$) are not exactly coupled but this complication has been already handled 
in~\cite[Section 7.2.1]{CES19}. We thus need to estimate the additional term
 $\mathrm{d}b_i^{z_l}-\mathrm{d}\beta_i^{(l)}$ along the flow; this additional bound is completely analogous to the estimate of $\mathrm{d}\xi_{1,i}$ in~\cite[Eqs. (7.92)--(7.95)]{CES19}.
 \end{enumerate}
\end{proof}

\appendix

\section{Proof of Theorem~\ref{local_thmw}}\label{app:local_law}

The proof of Theorem~\ref{local_thmw} is a combination of the cusp local law for diagonally deformed Wigner type 
matrices~\cite[Theorem 2.5]{EKS20} and the local law for $H^z$ on the imaginary axis $\Re w=0$
given in~\cite[Theorem 5.2]{AEK19b}. 
We first use this argument to show Theorem~\ref{local_thmw} for any $\eta\ge n^\xi\eta_f$
where $\eta_f=\eta_f(E,z)$ is the local eigenvalue
spacing (or \emph{fluctuation scale}) near the energy $E$ and $\xi>0$ is an arbitrary small.
If $E\in \mbox{supp}(\rho^z)$, then the fluctuation scale is defined by
\begin{equation}\label{fluscale}
   \int_{-\eta_f}^{\eta_f} \rho^z(E +x)\,\mathrm{d}x =\frac{1}{2n}.
\end{equation}
In particular if $1-\tau \le |z|\le 1$, then 
$$\eta_f=\min\{(1-|z|^2)^{-1/2} n^{-1},E^{-1/3}n^{-1}\}.$$
For $z$ outside of the support of the circular law,  i.e. in the  regime $ 1 < |z| \leq 1+\tau$, we 
have a symmetric gap of size $\Delta\sim (|z|^2-1)_+^{3/2}$ around  the origin
in the self-consistent density of states $\rho^z$. For $E$ within this gap we define
(see~\cite[Eq. (2.7)]{EKS20} or~\cite[Eq. (5.2)]{AEK19b}),
$$\eta_f= \min\{n^{-3/4}, \Delta^{1/9} n^{-2/3} \},
$$
which, in fact,  is the fluctuation scale at the internal edge of $\rho$ at $E=\Delta/2$ via the definition~\eqref{fluscale}.

Then we extend the range from $\eta\geq n^\xi \eta_f$ down to $\eta=n^{-1}$~(even to any $\eta>0$) exactly 
as in~\cite[Proposition 1]{CES21} which was given for a smaller regime $||z|-1| \lesssim n^{-1/2}$ and $w$
 lying on the imaginary axis. This argument can be easily extended to our regime $| |z|-1 |\leq c$ 
 and $w$ in a small neighbourhood of the imaginary axis
with minor modifications. 
The proof is standard, relying on the monotonicity of 
$ \eta\to \eta \langle   \mathbf{x},  (G^{z}(E+\ii \eta)  \mathbf{x}\rangle$, so we omit it for brevity. 

The rest of this section explains how to combine~\cite{EKS20} and~\cite{AEK19b} to prove 
Theorem~\ref{local_thmw} for  $\eta\ge n^\xi\eta_f$. While all ingredients 
are present in these papers, unfortunately neither result can be directly cited
since~\cite{EKS20} analyses the cusp local law for general $\Re w$
but only for a model with flatness, while~\cite{AEK19b} handles exactly
our model $H^z$, but  restricted only to $\Re w=0$, exactly in the 
middle of the small gap in the density of states. 

Therefore, the argument primarily follows the proof of~\cite[Theorem 5.2]{AEK19b} 
from~\cite[Sections 4 and 5.1]{AEK19b} that itself heavily relies on~\cite{EKS20}.
Fortunately, the main formulas in \cite[Sections 4 and 5.1]{AEK19b}  are written 
in a canonical way using the density $\rho^z$ from~\eqref{rho_E} as a control parameter
and not relying on evaluating it on the imaginary axis. 
We need to recalculate them only when actual estimates are used.
A major simplification is that we consider the i.i.d. situation, where the 
self-energy operator~\eqref{S} is particularly simple, while~\cite{AEK19b} (as well as~\cite{EKS20})
is written in the much more complicated setup when the variances $s_{ij}:=\E |x_{ij}|^2$ depend on $i, j$.
A large part of~\cite{AEK19b} is devoted to meticulous estimates of the solution to~\eqref{dyson}
and its derived quantities -- in our case all these are given  explicitly using~\eqref{Mmatrix}.

For any fixed $z$ and $w$
we first recall the \emph{linear stability operator} $\mathcal{B} $ acting on $2n\times 2n$ matrices and given by
$$
\mathcal{B}[T] = T - M^z(w) \mathcal{S}[T] M^z(w),
$$
see~\cite[Eq. 2.10]{AEK19b} with $M^z(w)$ defined in~\eqref{Mmatrix}-\eqref{m_function} and 
$\mathcal{S}$ is from in~\eqref{S}.
  From now on we often omit the superscript $z$ and the argument $w$ 
  and write, \eg,
   $M=M^{z}(w)$, $G=G^{z}(w)$, $\rho=\rho^z(w)$, etc. for notational simplicity. 
    Recall that the stability operator expresses the leading linear relation between the quantity $G-M$ we
   are interested in and the key ``renormalized" fluctuating object $D=WG+\s[G]G$ that we can compute:
   $\mathcal{B}[G-M] \approx MD$, see later. 
   
    In our case, the stability operator as well as its adjoint
(with respect to the standard Hilbert-Schmidt scalar product on matrices) leaves the four dimensional
space of block constant matrices 
invariant and it acts trivially as the identity on the remaining $4n^2-4$ dimensional space
of block-traceless matrices. On this four dimensional space  $\mathcal{B} $ has a further eigenvalue 1
with multiplicity two and it has two nontrivial, potentially small eigenvalues $\beta$ and $\beta_*$
with algebraic and geometric multiplicity one. Explicit 
calculation gives 
\begin{equation}\label{betas}
	\beta_*= 1 + m^2 - |z|^2 u^2, \qquad \beta = 1 - m^2 - |z|^2 u^2,
\end{equation}
and the left and right eigenvectors are given by 
$$
\mathcal{B}[B_*] =\beta_* B_*, \quad \mathcal{B}[B] =\beta B, \quad  
\mathcal{B}^*[\wh B_*] =\overline{\beta_*} \wh B_*, \quad \mathcal{B}^*[\wh B] =\overline{\beta} \wh B
$$   
with
\begin{equation}\label{Bs}
	B_* = ME_-M=(m^2-|z|^2u^2)E_-, \quad  B=M^2=\begin{pmatrix} m^2 + |z|^2u^2  & -2zum \\ -2\bar z um & m^2 + |z|^2u^2
	\end{pmatrix}, \quad \wh B_* = E_-, \quad \wh B =I,
\end{equation}
where we used the notation from~\cite[Prop. 3.1]{AEK19b}.

Now we comment on the two nontrivial small eigenvalues.
On the one hand, the small eigenvalue $\beta$ stems from the cusp regime as discussed in~\cite[Sec. 3.1]{EKS20}. 
 Since $B$~in (\ref{Bs}) has identical 
diagonal terms \ie $\s[B]=\<B\>$, the current $\mathcal{B}$ acts exactly in the same way on $B$
as in  \cite{EKS20}, see also the explanation below~\eqref{rho_E}. In particular,
 $\beta$ in \eqref{betas} satisfies the same scaling relation as in~\cite[Eq. (3.7c)]{EKS20}, \ie for small $||z|-1|$ and $|w|$, we have
\begin{align}\label{betabeta}
	|\beta| =|\beta(w)| \sim \frac{\eta}{\rho}+\rho(\rho+|\sigma|), \qquad w=E+\ii \eta
\end{align}
with the parameter $\sigma$ given as in~\cite[Eq. (3.5a)]{EKS20}: 
\begin{align}\label{sigma}
	\sigma=\sigma(w):=\<(\mathrm{sgn}(\Re U) (\Im U/\rho))^3\>, \qquad 
	U:=\frac{(\Im M)^{-1/2} (\Re M)(\Im M)^{-1/2}+\ii}{|(\Im M)^{-1/2} (\Re M)(\Im M)^{-1/2}+\ii|}.
\end{align}

The quantities $\sigma, U$ also depend on $w$ and $z$ as $\beta, m, \rho, M$ do
but we usually omit this dependence from the notation, we indicate them only when the statement would
otherwise be ambiguous. 

We point out that the key parameter $\sigma$ measures  the distance from the cusp, it is zero exactly at the cusp 
point. In various estimates later this parameter will represent the extra gain or loss specific to the cusp
compared to the bulk regime and the regular edge regime. For example, by~\eqref{betabeta}  
we see that $\beta$ is order one in the bulk, it is order $\rho$ at the regular edge and 
it is order $\rho^2$ at the cusp. Since the smallest eigenvalue  of the stability operator
governs the behavior of $G-M$, we see that the cusp regime is the most difficult as $\beta$ is very small.
This has to be compensated by two steps. First, unlike
in a typical proof of a local law in the bulk or regular edge regime,
 the linear approximation $\mathcal{B}[G-M] \approx -MD$ is not sufficient, we also need 
 to compute the subleading term quadratic in $G-M$. Second, we need to estimate $MD$
 better, at least in the ``bad direction", i.e. tested against the eigenvector of $\mathcal{B}$ corresponding to the smallest eigenvalue.
 In the setup of a general cusp in~\cite{EKS20} this second step was especially complicated
 but in our current setup (as well as in the setup of~\cite{AEK19b}) the extra spectral symmetry 
 of $H^z$ simplifies the proof a lot.

We further discuss several quantitative results for the parameter $\sigma$. Note that both $\Re m(w)$ and $\sigma(w)$ are $1/3$-H\"older continuous;
this follows directly from the fact that $m$ solves a cubic equation, but it
can also be derived from the general shape analysis for $\rho$ from~\cite[Remark 7.3(i)]{AEK18b}
and using that $m= \< M \>$ is the Stieltjes transform of $\rho$ exactly as in the setup of~\cite{AEK18b}.
 We also have
 $\Re m(\ii \eta) \equiv 0$, 
$\sigma(\ii \eta)\equiv 0$ by symmetry, and for $|z|>1$ we
have $|\sigma(\Delta/2)| \sim \Delta^{1/3}$  from~\cite[Eq. (10.15)]{AEK18b},
in particular $\sigma(E)$ increases as $|E|^{1/3}$ away from zero on the real axis.
Moreover, for small $||z|-1|$ and $|w|$, we 
have the relation 
\begin{align}\label{Realsigma}
	 100  |\Re w|\le  |\Re m| \sim |\sigma|,
\end{align}
 that can be checked directly using (\ref{sigma}) and that $m$ is the Stieltjes transform of $\rho$ with an asymptotic behavior  given in~\eqref{rho_E}.

On the other hand, due to the zero-block structure of $H^z$,  
there is another small eigenvalue, $\beta_*$, and from~(\ref{Mmatrix}) and (\ref{m_function}), it is explicitly given by 
$$\beta_*= \frac{w}{w+m}, \quad \mbox{hence} \quad |\beta_*| \sim \frac{\eta+|E|}{\rho+|\sigma|}, 
\qquad w=E+\ii \eta.$$
This bad direction was not present in the flat models of \cite{AEK18b, EKS20}.
In our case it is, however,  harmless since $G-M$ is exactly orthogonal to the associated 
eigenvectors $B_*$, $\widehat{B}_*$ 
 (recall from \eqref{Bs} that $B_*,\widehat{B}_*$ are parallel with $E_-$). 
  Here we used the key symmetry of $H^z$
from its block structure implying
\begin{equation}\label{blocksym}
\<E_-, G\>=0, \qquad \<E_-, M\>=0.
\end{equation}
Note that in~\cite[Eq. (3.2)]{AEK19b} simpler relations for $\beta$ and $\beta_*$ were obtained
 with $E=0$, and $\sigma=0$ for $\Re w=0$.

Next, we follow the proofs in \cite[Sections 4-5]{AEK19b} where $H^{z}$ 
 was considered but only for $\Re w=0$. Since we consider a small neighbourhood of $\Re w=0$ and thus have slightly different scaling relation~(\ref{betabeta}), we need to repeat the proofs with new estimates and minor modifications.  We then arrive at the same cubic relation as in the cusp local law paper \cite{EKS20}; see Lemma \ref{lemma_cubic} below. Finally we perform the same bootstrap as in \cite{EKS20} to prove Theorem \ref{local_thmw} for $\eta\geq n^\xi \eta_f$.  In the following we explain these details.

The main tool is to find an approximate cubic equation for $\< G-M\>$ and then deduce the size of $\< G-M\>$
by solving it. The delicate part of the analysis is to select the right solution out of the three; this 
is done by a continuity argument by reducing $\eta=\Im w$. However, this step has been done in \cite{EKS20},
so here we just need to show that we arrive at the same cubic equation.  Setting $Y:= G-M$, the 
cubic equation for $\< Y\>$ is found 
by analysing a general quadratic matrix equation of the form $\mathcal{B}[Y] - \mathcal{A}[Y, Y] +X=0$.
This is the first part of Lemma A.1 of \cite{AEK19b}.
The structure of this quadratic matrix equation and 
the fact that $\mathcal{B}$ has only one relevant small eigenvalue $\beta$
 also imply that  $Y$ is essentially parallel with $B$, the 
right eigenvector of $\mathcal{B}$ to  $\beta$, i.e. $Y\approx \Theta B$
with  $\Theta:= \<\wh B, Y\>/\<\wh B,B\>$. The second  part of Lemma A.1 of \cite{AEK19b} precisely 
identifies the subleading term of this approximation. 

Now we explain the slight modifications of some key statements in~\cite{AEK19b, EKS20}
whose combination gives the proof of Theorem~\ref{local_thmw}.

Thanks to the symmetry relations~\eqref{blocksym}, 
 we can  use Lemma A.1 of \cite{AEK19b} with the choice $Y=G-M$ and $X=MD$ with $D:=WG+\s[G]G$, and
\begin{align}\label{AB}
	\A[S,T]=\frac{ \s[S] MT+\s[T]M S}{2}, \qquad \B[T]=T-M\s[T]M  
\end{align}
as in the proof of Proposition 4.1 in \cite{AEK19b}.  The analysis is easier in this special i.i.d. case
than in  \cite{AEK19b}
  since we have explicit formulas for all four eigenvectors  in (\ref{Bs}). By direct computations using (\ref{Bs}) we have
\begin{align}\label{BBs}
	\<\wh B,B\>=\<M^2\> \sim 1, \qquad \<\wh B_*,B_*\>=m^2-|z|^2u^2 \sim 1, \qquad \<E_-,B\>=0, \qquad \<E_-,B_*\> \sim 1,
\end{align}
where these relations hold for small $|w|$ and $||z|^2-1|$.
Moreover using (\ref{Bs}) and (\ref{AB}) we have
\begin{align}\label{special}
	\<\wh B, \A[B,B_*]\>=0, \qquad \<E_-,\B^{-1} \Q \A[B,B]\>=0,
\end{align}
where  $\Q$ is the spectral projection of $\B$ onto the all eigenvalues other than  
$\beta,\beta^*$. Since these are well separated from zero,  in particular we have $\|\B^{-1}\Q\| \lesssim 1$.
Using (\ref{Bs}), (\ref{betabeta}), (\ref{Realsigma}), (\ref{BBs}) and (\ref{special}), the $\mu$-coefficients in Lemma A.1 of \cite{AEK19b} can be computed explicitly, \ie
\begin{align}
	\mu_3=&2\< \wh B, \A[B,\B^{-1}\Q \A[B,B]\> \sim \<\A[M^2,M^3]\> \sim 1,\nonumber\\
	\mu_2=&\<I, \A[B,B]\> \sim \<M^3\> \sim \rho+|\sigma|, \nonumber\\
	\mu_1=&-\beta\<M^2\>-2\<\A[M^2,\B^{-1}\Q[MD]]\>,\nonumber\\ 
	\mu_0=&\<\A[\B^{-1}\Q[MD],\B^{-1}\Q[MD]]\>-\<MD\>.
\end{align}
Then  as in \cite[Eq. (4.7-4.8)]{AEK19b}, $\Theta=\<\wh B, G-M\>/\<\wh B,B\>\sim \<Y\>$
 satisfies the following cubic equation:
\begin{align}\label{cubic_eq}
	\mu_3 \Theta^3+\mu_2 \Theta^2-\beta \<M^2\> \Theta=&-\mu_0+\<R_2,D\> \Theta+O\Big(n^{-1/4K}|\Theta|^3+n^{62/K} \big(\|D\|_*^3+|\<R_1,D\>|^{3/2}\big)\Big),\nonumber\\
	=& -\mu_0+O\Big(n^{-1/4K}|\Theta|^3+n^{62/K} \big(\|D\|_*^3+|\<R_1,D\>|^{3/2}+|\<R_2,D\>|^{3/2}\big)\Big), 
\end{align}
for any $K>1$ as long as $\| G-M\|_* \le n^{-30/K}$ 
with  
\begin{align}\label{R1_R2}
	R_1:=M^*(\B^{-1}\Q)^*[E_-], \qquad R_2:=M^*(\B^{-1}\Q)^*\Big[\<(M^*)^2\> M^* +\<(M^*)^3\>\Big].
\end{align}
Here $\| \cdot\|_*$ is a specific norm on random matrices (depending on $K$ and two fixed deterministic
vectors $\mathbf{x}, \mathbf{y}$) introduced in Section 3.1 of \cite{EKS20},
see also \cite[Section 4.1]{AEK19b} whose exact definition is irrelevant for the current explanation. Note that the last line of (\ref{cubic_eq}) is obtained using that $|\<R_2,D\> \Theta| \leq n^{-1/4K}|\Theta|^3+n^{1/8K}|\<R_2,D\>|^{3/2}$ from Young’s 
inequality, and we may absorb $n^{-1/4K}|\Theta|^3$ into the left side of (\ref{cubic_eq}). 

As we mentioned above,
the solution $\Theta\sim \langle G-M\rangle$ of the cubic equation~\eqref{cubic_eq} governs the leading 
order behaviour of the entire matrix $G-M$ as $G-M\approx \Theta B$.
 The following proposition identifies the subleading terms in this approximation and we will comment on them
 after the statement. 
This is the analogue of \cite[Proposition 4.1]{AEK19b} but
with new estimates on the coefficients as given in \cite[Proposition 3.4]{EKS20}
(note also
that the error term became simpler due to~\eqref{special}):
\begin{proposition}\label{prop_cubic}  
There exists a small $\tau_*>0$ such that for any $||z|-1 |\leq \tau_*$ and $|w|\le \tau_*$, the following statement holds. Assuming $\| G-M\|_*+ \|D \|_*\le n^{-30/K}$ for some fixed $K>1$, 	
	we have the expansion
	\begin{align}\label{sub}
		G-M=\Theta B -\B^{-1} \Q[MD]+\Theta^2 \B^{-1} \Q [M \s[B] B]+E,
	\end{align}
	where the error matrix $E$ has the upper bound
	\begin{align}\label{E_star}
		\|E\|_* \lesssim n^{1/16K} \Big( |\Theta|^3 +|\Theta|  \|D\|_*+\|D\|^2_*+|\<R_1,D\>|\Big),
	\end{align}
	and the scalar	$\Theta=\<\wh B, G-M\>/\<\wh B,B\> = \<G-M\>/\<\wh B,B\>$ satisfies the cubic equation:
	\begin{align}\label{cubice}
		\Theta^3+\xi_2\Theta^2+\xi_1 \Theta=\epsilon_*.
	\end{align}
	Here the coefficients satisfy the scaling relations:
	\begin{align}\label{xi_12}
		 |\xi_1 |
		\sim |\beta| \sim
		\frac{\eta}{\rho}+\rho(\rho+|\sigma|), \qquad |\xi_2|
		\sim \rho+|\sigma|,
	\end{align}
	and the error term $\epsilon_*$ is bounded by
	\begin{align}\label{epsilon_star}
		|\epsilon_*| \lesssim n^{62/K}\big( \|D\|^3_*+|\<R_1,D\>|^{3/2}
		+|\<R_2,D\>|^{3/2} +|\<MD\>|+|\<M \s[\B^{-1}\Q[MD]]\B^{-1}\Q[MD]\>|\big),
	\end{align}
	where $R_1$ and $R_2$ given in (\ref{R1_R2}) are deterministic matrices with $\|R_1\|, \|R_2\| \lesssim 1$.  
	\qed
\end{proposition}

The subleading terms~\eqref{sub}--\eqref{E_star} are somewhat involved but they follow a simple
power-counting pattern. The key quantity is $X=MD$, the small source term in 
the quadratic matrix equation $\mathcal{B}[Y] - \mathcal{A}[Y, Y] +X=0$,
where we recall that $D=WG+\s[G]G$ is the ``renormalized" version of $WG$.
The added ``counter-term" $\s[G]G$ guarantees that $\E D$ vanishes up to the first and second
order in the cumulant expansion on high moments of $D$. The size of $D$ is measured
in the norm $\| D\|_*$ introduced below~(\ref{R1_R2}).  However, when $D$ is tested against any  bounded deterministic matrix $R$, then
$\langle R, D\rangle$ is one order better than naively expected, i.e. we have $|\langle R, D\rangle|\lesssim \| D\|_*^2$;
this ``fluctuation averaging" is a basic property of the fluctuating term $D$. Furthermore, in the cusp regime
$D$ tested specifically against a block constant 
(almost) off-diagonal matrix $R$ (like $M$) is one more order smaller,  roughly speaking 
$|\langle RD \rangle|\lesssim \| D\|_*^3$. These claims will be formalized in Theorem~\eqref{MD_bound} below.
 Thus every term in~\eqref{epsilon_star} is of order $\| D\|_*^3$
(the last term seems only quadratic in $D$,  but $\s$ contains an effective averaging
so $ \s[\B^{-1}\Q[MD]]\lesssim \| D\|_*^2$). Thus $\Theta$ satisfies a cubic equation~\eqref{cubice} with leading
coefficient 1 up to a precision of order $\| D\|_*^3$. Depending on the size of the other coefficients $\xi_1, \xi_2$
this determines the size of $\Theta$. For example in the bulk regime (or if $\eta\sim 1$) we have $\xi_1\sim 1$
thus $|\Theta|\sim \epsilon_*\sim \| D\|_*^3$ as the cubic and quadratic terms are negligible. 
Exactly at the cusp ($\eta =\sigma=\rho=0$) we have $\xi_1=\xi_2=0$ and thus $\Theta\sim (\epsilon_*)^{1/3}
\sim \| D\|_*$. Later we will perform a bootstrap argument in which we gradually reduce $\eta$
thus the equation changes from the ``bulk" behavior to the ``cusp" behavior (with an intermediate 
stage when the quadratic term $\xi_2\Theta^2$ becomes dominant).

Armed with these intuitions, we can now explain the subleading terms in~\eqref{sub} 
beyond the $\Theta B$ leading term which is of order $|\Theta|$. The first subleading term is of order $\| D\|_*$
but it lies in harmless spectral subspace of 
$\mathcal{B}$.  The second subleading term is of order $|\Theta|^2$. The error term~\eqref{E_star}
is of order $|\Theta|^3  + |\Theta|  \|D\|_*+\|D\|^2_*$, i.e. it is lower order in both small parameters.
It turns out that expanding the solution of $\mathcal{B}[Y] - \mathcal{A}[Y, Y] +X=0$ up to this order
is sufficient for our purposes.

Next, we estimate the error terms (\ref{E_star}) and (\ref{epsilon_star}) involving the 
fluctuation $D=WG+\s[G]G$ by a cumulant expansion on high moments of $D$. Similar estimates
using a sophisticated Feynman diagrammatic expansion appeared first in~\cite[Theorem 4.1]{EKS20} and
they   have been used to prove optimal local laws first  in the bulk
regime for very general random matrices with correlated entries in~\cite{EKS19} 
and later in the edge regime in~\cite{AEKS20}. Now we are in the cusp regime, where
the key point is to gain an additional small factor specific for the cusp.
This was done for general cusps in the context of Wigner-type matrices 
 in \cite[Theorem 3.7]{EKS20} and later specialized for $H^z$ 
 on the imaginary axis, $\Re w=0$, in \cite[Proposition 5.5]{AEK19b}.
In the following theorem 
 we will establish the analogue of \cite[Proposition 5.5]{AEK19b} slightly away from the imaginary axis
by explaining the necessary changes. Introduce the $L^p$ norm for random scalars $Z$ and 
random matrices $Y\in \C^{2n\times 2n}$ as follows:
$$
    \| Z\|_p: = ( \E |Z|^p)^{1/p}, \qquad \| Y\|_p: =\sup_{\mathbf{x}, \mathbf{y}}
    \frac{\| \langle\mathbf{x}, Y\mathbf{y}\rangle\|_p}{\| \mathbf{x}\| \| \mathbf{y}\|},
$$
where the supremum is over all deterministic vectors.
\begin{theorem}\label{MD_bound}
There is a small $\tau_*>0$ and a large constant $C>0$ such that for any $p\ge 1$, $\epsilon>0$,
$\big| |z|-1\big|\le \tau_*$, $|w|\le\tau_*$, for any deterministic vectors $\mathbf{x}, \mathbf{y}\in \C^{2n}$
and matrix $R\in \C^{2n\times 2n}$ we have
\begin{equation}\label{isoD}
\| ( \mathbf{x}, D \mathbf{y})\|_p\le_{\epsilon, p}  n^\epsilon \psi_q'   \big(1+ \| G\|_q\big)^C
 \Big( 1+ \frac{\| G\|_q}{\sqrt{n}}\Big)^{Cp}
\| \mathbf{x}\| \| \mathbf{y}\|,
\end{equation}
\begin{equation}\label{aveD}
\| \langle RD\rangle\|_p \le_{\epsilon, p}  n^\epsilon \big[ \psi_q' \big]^2  \big(1+ \| G\|_q\big)^C
 \Big( 1+ \frac{\| G\|_q}{\sqrt{n}}\Big)^{Cp} \| R\|.
\end{equation}
Moreover, if $R$ is block constant off-diagonal matrix, then we have the improved estimate
\begin{equation}\label{aveDimp}
\| \langle RD\rangle\|_p \le_{\epsilon, p}  n^\epsilon\sigma_q \big[ \psi +\psi_q' \big]^2  \big(1+ \| G\|_q\big)^C
 \Big( 1+ \frac{\| G\|_q}{\sqrt{n}}\Big)^{Cp} \| R\|,
\end{equation}
where we defined the control parameters
$$
 \psi = \sqrt{\frac{\rho}{n\eta}}, \quad \psi_q' = \sqrt{\frac{\| \Im G \|_q}{n\eta}}, 
 \quad \psi_q''= \| G-M\|_q, \quad \sigma_q=\rho+|\sigma|+\psi+\sqrt{\frac{\eta}{\rho}}+\psi'_q+\psi''_q
 $$
 with $q=Cp^2/\epsilon$.
\end{theorem}
The error terms are somewhat complicated, but for simplicity the reader
 can think of $\psi\sim \psi_q'\sim\psi''_q$
and consider $\psi$ as the main parameter in the power counting. Moreover 
the terms containing $\| G\|_q$ can be ignored as they are roughly order one.
Thus the isotropic bound~\eqref{isoD} on $D$ 
is of order $\psi$, while the averaged bound~\eqref{aveD} with a general deterministic matrix
is of order $\psi^2$, in agreement of the general fact that averaged bounds are ``one order better"
than isotropic ones. The key novelty is the improved average bound~\eqref{aveDimp}
for block constant off-diagonal test matrix, which is (roughly) of order $\sigma \psi^2$. 
Recall that $\sigma$ is the cusp-parameter~\eqref{sigma} that is small in the cusp regime.

\begin{proof}
Exactly as in~\cite[Section 5.2]{AEK19b},
the proof of~\eqref{isoD}--\eqref{aveD} is identical to those of~\cite[Eq. (3.11a) and (3.11b)]{EKS20}
which directly follow from~\cite[Theorem 4.1]{EKS19}
since it did not use flatness. 

Considering~\eqref{aveDimp},
the only difference between this result and  \cite[Eq. (5.5c) in Proposition 5.5]{AEK19b} is  that $\sigma_q$ is 
redefined by adding the $|\sigma|$ term, which was zero on the imaginary axis in the setup of~\cite{AEK19b}. 
The proof of~\cite[Eq. (5.5c)]{AEK19b} 
exploited the special almost off-diagonal structure of $M$ on the imaginary axis only at one critical point
where the term $\s[ MJK^{(b)}M^*]$ in~\cite[Eq. (5.32)]{AEK19b} was estimated by $O(\rho)$
(see the explicit comment below~\cite[Eq. (5.32)]{AEK19b}). 
Here $K^{(b)}$ is a bounded diagonal matrix whose precise form is irrelevant and 
$$
    J = \begin{pmatrix} 0 & 1\cr 1 & 0\end{pmatrix}.
$$
In our case we can also decompose $M$ as $M=M_d+M_o$ with
\begin{align}\label{split_M}
	 M_d:=\begin{pmatrix}
		m^{z}(w)  &  0  \\
		0  & m^{z}(w)
	\end{pmatrix}, \quad M_o:=\begin{pmatrix}
		0 &  -zu^z(w)  \\
		-\overline{z}u^z(w)   & 0
	\end{pmatrix}, \qquad \| M_d\|\lesssim \rho +|\sigma|,
\end{align}
where we used~\eqref{Realsigma}, i.e. our $M$ is also essentially off-diagonal since $\rho +|\sigma|\lesssim \tau_*^{1/3}$
 is small in our parameter regime.
 We then easily obtain that $\s[ MJK^{(b)}M^*]= O(\rho+|\sigma|)$.
In other words, we use exactly the same mechanism as in~\cite[Eq. (5.32)]{AEK19b}, just
the off-diagonal part of $M$ is of order $\rho+|\sigma|$ instead of $\rho$. Following this change
along the proof of~\cite[Eq. (5.5c)]{AEK19b}, one easily sees that the consequence
is only in the indicated redefinition of $\sigma_q$. This completes the proof of Theorem~\ref{MD_bound}.
\end{proof}

%
%

Then we use Theorem \ref{MD_bound} and similar arguments as in \cite[Section 5.1]{AEK19b} to estimate the error terms in (\ref{E_star}) and (\ref{epsilon_star}). In particular, using (\ref{split_M})
we obtain an improved bound for 
$$|\<MD\>| \leq |\<M_d D\>|+|\<M_o D\>| \prec \Big(\rho+|\sigma| +\sqrt{\frac{\eta}{\rho}}+\sqrt{\frac{\rho+\Xi}{n\eta}}+\theta\Big)\frac{\rho+\Xi}{n\eta},$$
whenever $|\Im (G-M)| \prec \Xi$ and $|G-M| \prec \theta$ hold for some deterministic control parameters $\Xi, \theta$.
Then we obtain the following cubic relation for $\Theta$ and initial bounds for $G-M$ which are the same as in~\cite[Lemma 3.8]{EKS20}. The proof of Lemma \ref{lemma_cubic} is exactly the same as in \cite[Section 5.1]{AEK19b}, so we omit the details.
\begin{lemma}\label{lemma_cubic}
	Suppose that $|G-M| \prec \Gamma$, $|\Im(G-M)| \prec \Xi$ and $|\Theta| \prec \theta$ for a fixed $w$ such that $\Im w \geq n^{-1+\zeta}$ with a small $\zeta>0$ and assume that these deterministic control parameters satisfy $\Gamma+\Xi+\theta \lesssim N^{-c}$. Then for sufficiently small $\epsilon>0$ we have
	$$|\Theta^3+\xi_2 \Theta^2+\xi_1 \Theta| \prec N^{2\epsilon}\Big(\rho+|\sigma|+\sqrt{\frac{\eta}{\rho}}+\sqrt{\frac{\rho+\Xi}{n\eta}} \Big)\frac{\rho+\Xi}{n\eta}+N^{-\epsilon}\theta^3,$$
	as well as, for any bounded and deterministic $\mathbf{x},\mathbf{y}\in \C^{N}$ and $B\in \C^{N\times N}$,
	$$|\<\mathbf{x},(G-M) \mathbf{y}\>| \prec \theta+\sqrt{\frac{\rho+\Xi}{n\eta}}, \qquad |\<B(G-M)\>| \prec \theta+\frac{\rho+\Xi}{N\eta}.$$
\end{lemma}
Thus the bootstrap procedure using this cubic relation in Lemma \ref{lemma_cubic} is the same as in \cite[Section 3.3]{EKS20} with auxiliary coefficients~(\cf Eq. (3.7e) of \cite{EKS20})
\begin{align}\label{xi_12_wt}
	\wt\xi_1:=\begin{cases}
		(|\kappa|+\eta)^{1/2} (|\kappa|+\eta+\Delta)^{1/6},\\
		\big(\rho(0)+(|\kappa|+\eta)^{1/3}\big)^{2},
	\end{cases} 
		 \qquad 
		 \wt \xi_2:=\begin{cases}
		 	(|\kappa|+\eta+\Delta)^{1/3},& \qquad \mbox{if~} |z|>1,\\
		 	\rho(0)+(|\kappa|+\eta)^{1/3},& \qquad \mbox{if~} |z|\leq 1,
		 \end{cases}
\end{align}
where $\kappa$ and $\eta$ are from $w=\frac{\Delta}{2}+\kappa+\ii \eta$ as in (\ref{rho_E}) with $\Delta \sim (|z|-1)^{3/2}$ for $|z|>1$ and $\Delta=0$ for $|z|\leq 1$. It is clear that $\xi_1,\xi_2$ in (\ref{xi_12}) and the auxiliary ones $\wt\xi_1,\wt \xi_2$ in (\ref{xi_12_wt}) satisfy the relations in Lemma 3.3 and the assumptions of Lemma 3.10 in \cite{EKS20}, hence we follow \cite[Section 3.3]{EKS20} to finish the proof of Theorem~\ref{local_thmw}.

\section{Proof of Lemma \ref{lemma_third_order}}\label{appendix_A}

The proof of Lemma \ref{lemma_third_order} is based on \cite[Section 4-5]{maxRe} using iterative cumulant expansions via an unmatched index. We assume that the reader is familiar with this idea, but for completeness we will restate necessary definitions and results from \cite{maxRe}, and refer to \cite[Section 4-5]{maxRe} for the detailed proofs. Most statements are essentially the same as in \cite[Section 4]{maxRe}, and we will also clarify our differences and improvements over \cite{maxRe}.

We first recall the definitions of unmatched indices and unmatched terms from 
\cite[Definition 4.3-4.4]{maxRe} with the only  difference that now we also allow a possible 
$\wh{ \F_t^{z_0}}$-factor. For the reader's convenience we recall the notational conventions that, 
for any fixed $l_1,l_2 \in \N$, $\mathcal{I}_{l_1,l_2}$ denotes a set of $l_1$ lower case letters and $l_2$ upper case letters, in general denoted by $v_j~(1\leq j\leq l_1)$ and $V_j~(1\leq j\leq l_2)$ respectively. Each element in $\mathcal{I}_{l_1,l_2}$ will represent a summation index and the font type of each letter indicates the range of the summation for that index: the 
lower case letters $v_j$ run from 1 to $n$, and the upper case letters $V_j$ run from $n+1$ to $2n$.  
We denote the free sum over these $l:=l_1+l_2$ summation indices by $\sum_{\mathcal{I}_{l_1,l_2}}$. We also introduce a partial summation restricted to distinct indices, 
\begin{align}\label{distinct_sum}
	\sum^*_{\mathcal{I}_{l_1,l_2}}:=\sum_{v_1,\cdots v_{l_1}, V_1,\cdots,V_{l_2}} \Big(\prod_{j\neq j'}^{l_1} \delta_{ v_j \neq v_{j'}} \Big)\Big(\prod^{l_2}_{j\neq j'} \delta_{ V_j \neq V_{j'}} \Big) \Big( \prod_{j=1}^{l_1} \prod_{j'=1}^{l_2} \delta_{ v_j \neq \ud{V_{j'}}}  \Big),
\end{align}
\ie
each summation index in $\mathcal{I}_{l_1,l_2}$ is different from all the other indices and their conjugates.

\begin{definition}
	\label{def:unmatch_form}
	Given $l_1, l_2\in \mathbb{N}$ and a collection of lower and upper case summation indices 
	$\mathcal{I}_{l_1,l_2}=\{v_j\}_{j=1}^{l_1} \cup \{V_j\}_{j=1}^{l_2}$,
	we consider a product of $d$ generic shifted Green function entries $ \wh{G^{z_1}_{x_1y_1}} \wh{G^{z_2}_{x_2y_2}} \cdots \wh{G^{z_d}_{x_dy_d}}$ at the level $\eta_0=n^{-1+\epsilon}$ with different $z_i\in \C$
	and assign a summation index $v_j$, $V_j$ or their conjugates $\overline{v_j}, \ud{V_j}$ to each generic index $x_i, y_i$
	(\eg $x_1\equiv v_2, y_1\equiv \ud{V_5}, x_2\equiv \overline{v_3}, y_2 \equiv V_5$, etc.). 
	We also include a possible factor $\wh{ \F_t^{z_0}}$ given in (\ref{F}) with $z_0\in \C$.
	A term of the form 
	\begin{align}\label{form}
	(\wh{ \F_t^{z_0}} )^{\alpha_0}  \frac{1}{n^{l}}  \sum^*_{\mathcal{I}_{l_1,l_2}} 
	 \prod_{i=1}^{d}  \wh{G^{z_i}_{x_iy_i}}(\ii \eta_0),\qquad l=l_1+l_2,\quad\alpha_0=0,1,
	\end{align}
	with a concretely specified assignment is denoted by $P_d$.  The number of shifted Green 
	function factors $d$ is also referred to as the degree of such term. The collection of the
	 terms of the form in (\ref{form}) with degree $d$ is denoted by $\mathcal{P}_{d}$. 
	
	Given a term $P_d \in \mathcal{P}_d$ in (\ref{form}), we say that a lower case
	index $v_j \in \mathcal{I}_{l_1,l_2}$ is {\it matched} if the number of assignments of $v_j$ and
	its conjugate $\overline{v_j}$ to a row index in the product agrees with their number 
	of assignments to a column index, \ie 
	\begin{align}\label{match_condition}
		\#\{i : x_i\equiv {v_j}\}+\#\{i: x_i\equiv \overline{v_j}\}=\#\{i: y_i\equiv{v_j}\}+\#\{i: y_i\equiv\overline{v_j} \}.
	\end{align}
	Otherwise, we say that $v_j$ is an {\it unmatched} index. Similarly, we say that an upper case
	index $V_j \in \mathcal{I}_{l_1,l_2}$ is matched if 
	\begin{align}\label{match_condition_2}
		\#\{i: x_i\equiv {V_j}\}+\#\{i: x_i\equiv \ud{V_j}\}=\#\{i: y_i\equiv{V_j}\}+\#\{i: y_i\equiv\ud{V_j} \}.
	\end{align}
	Otherwise, $V_j$ is an unmatched index.	
	
	If all the summation indices in $\mathcal{I}_{l_1,l_2}$ are matched, then $P_d$ is a {\it matched term}. 
	Otherwise, if there exists at least one unmatched index, $P_d$ is an {\it unmatched term}. 
	If a term $P_d$ is unmatched, we indicate this fact by denoting it 
	by $P_d^o$. The collection of the unmatched terms of the form in (\ref{form})
	with degree $d$ is denoted by $\mathcal{P}_d^o \subset \mathcal{P}_d$. 
\end{definition}

To study the third order terms given by (\ref{third}) in general, compared to the form used
 in~\cite[Definition 4.3~(4.23)]{maxRe}, we not only allow the parameters $z_i \in \C$ of the
  shifted Green function entries in (\ref{form}) have different values, but also allow a possible
   factor $(\wh{\F^{z_0}_t})^{\alpha_0}$ in front of the shifted Green function entries;  see 
   also~\cite[Eq~(5.25)]{maxRe} for a similar form with such generalizations. We have the
    following lemma for these  unmatched terms.
\begin{lemma}\label{lemma:expand}
	Let  $P_d^o\in  \mathcal{P}_d^o$ be a given unmatched term in (\ref{form}),  with a fixed 
	degree $d\in \N$ and a fixed number of summation indices $l \in \N$ and $\eta_0=n^{-1+\epsilon}$. 
	Without loss of generality we assume the index $a$ assigned to $x_1$, \ie $P^o_d=P_d^o(x_1 \equiv a)$,
	 is an unmatched index satisfying 
	\begin{align}\label{choose}
		k^{(r)}_a:=\#\{i:x_i \equiv a, \overline{a}\} > k^{(c)}_a:=\#\{i:y_i \equiv a, \overline{a}\}.
	\end{align}
	Then there exist the following finite (bounded by a constant depending on $d$) subsets 
	\begin{align}\label{subset}
		{\mathcal{A}}^o_{d}\subset \mathcal{P}^o_{d},
		\quad {\mathcal{A}}^o_{>d}\subset  
		\mathcal{P}^o_{d+1}, \quad
		{\mathcal{B}}^o_{\ge d}\subset  \bigcup_{d'\ge d}\mathcal{P}^o_{d'}, \quad
		{\mathcal{C}}^o_{\ge d-2}\subset  \bigcup_{d'\ge d-2}\mathcal{P}^o_{d'}
	\end{align}
	such that we have the bound
	\begin{align}\label{aa}
		\big|\E[P_d^o(x_1 \equiv a)] \big| \lesssim &\sum_{P^o_{d} \in{\mathcal{A}}^o_{d}} \big|\E [ P^o_{d'} ]\big|+	
		\sum_{P^o_{d'} \in {\mathcal{A}}^o_{>d}} \big|\E [ P^o_{d'} ]\big| \nonumber\\
		&+ \frac{1}{\sqrt{n}}\sum_{P^o_{d'} \in {\mathcal{B}}^o_{\ge d}} \big|\E [ P^o_{d'} ]\big| 
		+ \frac{1}{n}\sum_{P^o_{d'} \in {\mathcal{C}}^o_{\ge d-2}} \big|\E [ P^o_{d'} ]\big| +O_\prec(n^{-3/2}),
	\end{align}
	where the number of summation indices in all
	elements of ${\mathcal{A}}^o_{d}$ is increased to $l+1$, and the number of $a/\bar a$-assignments as a row or column index in all
	elements of ${\mathcal{A}}^o_{d}$ is reduced to $k^{(r)}_a-1$ and $k^{(c)}_a-1$, respectively. In particular, 
	if $k^{(c)}_{a}=0$, then ${\mathcal{A}}^o_{d}$ is an empty set. 
	
	Moreover, if we further assume $d\geq 4$ or $l\geq 3$, then the last error term in (\ref{aa})
	can be improved to $O_\prec(n^{-3/2-\epsilon})$.
\end{lemma}

The proof of this lemma is postponed to the end of this section. Here we only remark that
the expansion in (\ref{aa}) was already proved in \cite[Lemma 4.8]{maxRe} for unmatched terms in (\ref{form}) without the $\wh{\F}$ factor; including this factor will be an easy exercise.  
The real novelty of Lemma \ref{lemma:expand} is the $n^{-\epsilon}$
improvement on this error term 
for $d\geq 4$ or $l\geq 3$. 	This improvement is essential for our entire proof to balance the
factor $\int |\Delta f| \sim n^{1/2}$ and it is one of the novelties of the current
GFT proof compared with the one
given  in~\cite{maxRe}.

We next briefly recall the following statements from \cite[Proposition 4.5]{maxRe} the origin of these unmatched terms in \eqref{aa} and their features and improvements compared to the initial term.
\begin{enumerate}
	\item The set ${\mathcal{A}}^o_{d}$ contains four types of unmatched terms (if exist) of degree $d$
	obtained by index replacements, \ie  
	\begin{align}\label{aaexpl}
		&m^{z_1}  \sum_{i: y_i \equiv a} m^{z_i} \E\big[P^o_d(x_1, y_i \rightarrow J)\big] 
		+m^{z_1}  \sum_{i: y_i \equiv \bar a} \mathfrak{m}^{z_i} \E\big[P^o_d(x_1, y_i \rightarrow J)\big]\nonumber\\
		&+ \mathfrak{m}^{z_1} \sum_{i:  y_i \equiv \bar a} m^{z_i}\E\big[ P ^o_d(x_1, y_i \rightarrow j)\big]
		+\mathfrak{m}^{z_1} \sum_{i: y_i \equiv a} \overline{\mathfrak{m}^{z_i}}\E\big[ P ^o_d(x_1, y_i \rightarrow j)\big],
	\end{align}	
	where both $j$ and $J$ are fresh (averaged) summation indices, although the 
	only important fact is that the number of $a/\bar a$-indices 
	is reduced by two (\ie one from the row and one from the column) 
	compared with the initial term $P_d^o(x_1 \equiv a)$.  
	For a concrete example of the above index replacement see \eqref{leading} below. 
	Once the number of $a/\bar a$-indices has been reduced to one, 
	the corresponding set ${\mathcal{A}}^o_{d}$ is then empty.

	\item The set ${\mathcal{A}}^o_{>d}$ corresponds to all the other second 
	order terms except (\ref{aaexpl}) with higher degrees, 
	\eg from the first two lines of~\eqref{example_expand} below; their degree
	is increased by one compared to the original term.
	
	\item The set ${\mathcal{B}}^o_{\ge d}$ comes from the third order cumulant expansion,
	indicated by the additional $1/\sqrt{n}$
	prefactor (see~\eqref{example_expand_term} below with $p+q+1=3$). 
	The degree remains at least $d$ and we gained $1/\sqrt{n}$ from the third order cumulants.
	
	\item The set ${\mathcal{C}}^o_{\ge d-2}$ coming with a prefactor $1/n$ has two very different sources. 
	On the one hand, it comes from the fourth order cumulant expansion
	carrying an extra $1/n$ and the degree remains at least $d$. 
	On the other hand, in the second order cumulant expansion 
	the fresh index $J$ or $j$ may coincide with an old index~(which yields an extra $1/n$ from the restricted summation)
	 creating a diagonal term. 
	The degree may be reduced by two from these diagonal elements; 
	see \eg (\ref{leading}) below with $J=B$ or $j=\ud{B}$.
\end{enumerate}

Note that all terms in the rhs. of~\eqref{aa} remain unmatched with improvements shown as above; this key
feature  allows us to iterate this estimate exactly in the same way as 
in~\cite[Section 4]{maxRe}. We thus proved the following.
\begin{proposition}\label{prop:unmatched}
	Given an unmatched term $P^{o}_d$ of the form in (\ref{form}) with a fixed degree $d\in \N$ and a fixed number of summation indices $l \in \N$ and $\eta_0=n^{-1+\epsilon}$. Then we have
	\begin{align}\label{unmatch_estimate}
	|\E[P^o_d]| =\OO_\prec(n^{-3/2}).
	\end{align}
	Moreover, if we further assume $d\geq 4$ or $l\geq 3$, then the estimate can be improved to
	\begin{align}\label{unmatch_estimate_improve}
		|\E[P^o_d]| =\OO_\prec(n^{-3/2-\epsilon}).
	\end{align}		
\end{proposition}

Armed with Proposition \ref{prop:unmatched}, we are ready to obtain the improved estimate in Lemma \ref{lemma_third_order}, i.e. to extract the deterministic leading term and a small error for the third order terms in~(\ref{third}) with $a\neq \ud{B}$.

\begin{proof}[Proof of Lemma \ref{lemma_third_order}]
	
We first consider the last third order term in (\ref{third}) with restricted summations $a\neq \ud{B}$ as an example, \ie
\begin{align}\label{third_aB_goal}
	\frac{\sqrt{n}}{n^{2}} \sum_{a \neq \ud{B}}\E \Big[ {\ga_{aB}} {\gb_{Ba}} {\gb_{Ba}} \Big]=\frac{\sqrt{n}}{n^{2}} \sum_{a \neq \ud{B}}\E \Big[ \wh{\ga_{aB}} \wh{\gb_{Ba}} \wh{\gb_{Ba}} \Big].
\end{align}
All the other third order terms in (\ref{third}) with $a\neq \ud{B}$ can be handled similarly and we only sketch the proofs for the other terms later in (\ref{general_third}).

From Definition~\ref{def:unmatch_form}, since the summation index $a$~(or $B$) are 
assigned three times as the row/column index of the shifted Green function entries, 
the term in (\ref{third_aB_goal}) is an unmatched term in $\mathcal{P}_3^o$ with $l=2$ and with a prefactor $\sqrt{n}$. 
Next we will use iterative expansions as in (\ref{aa}) and the improved bound in~(\ref{unmatch_estimate_improve}) 
to show the refined estimate of the unmatched term in (\ref{third_aB_goal}) (omitting the prefactor $\sqrt{n}$), \ie
\begin{align}\label{p_3_o0}
	\frac{1}{n^{2}} \sum_{a \neq \ud{B}}\E \Big[ \wh{\ga_{aB}} \wh{\gb_{Ba}}
	 \wh{\gb_{Ba}} \Big]=\frac{C}{n^{3/2}} \big(\m^{z_1}\big)^2 (\overline{\m^{z_2}})^4 +O_\prec(n^{-3/2-\epsilon}),
\end{align}
for some numerical constant $C\in \R$.

Recall the following identity for the shifted Green function \cite[Eq (5.2)]{CES19}:
\begin{align}\label{identity_0}
	\wh{\gz}=-M^{z} \ud{ WG^{z}}+\<\wh \gz\> M^z G^z,
\end{align}
with $\wh{\gz}=\gz-M^{z}$ and
\begin{align}\label{Mmatrix_00}
	 W:=\begin{pmatrix}
		0  &  X  \\
		X^*   & 0
	\end{pmatrix}, \qquad 
M^{z}=\begin{pmatrix}
		m^{z}  &  \mathfrak{m}^{z}  \\
		\overline{\mathfrak{m}^{z}}   & m^{z}
	\end{pmatrix}.
\end{align}
This formula needs some explanation. First, the normalized trace of the resolvent can be expressed in different
ways:
\begin{align}\label{normal_trace}
	\<{\gz}\> =\frac{1}{n} \sum_{v=1}^{n} {\gz_{vv}}=\frac{1}{n} \sum_{V=n+1}^{2n} {\gz_{VV}}
	=\frac{1}{2n} \sum_{\mathfrak{v}=1}^{2n} {\gz_{\mathfrak{vv}}},
\end{align}
which follows from the spectral symmetry induced by the $2 \times 2$ block matrix structure in (\ref{initial}). 
Note that the same relation holds true for $\langle \wh{\gz}\rangle $ since the
diagonal entries of $M^z$ in \eqref{Mmatrix_00} are the same.
Second, we need to recall the underline notation $\ud{ WG^{z}}$. 
For a function $f(W)$ of the random matrix 
$W$, we define
\begin{align}\label{ud_W}
	\ud{ Wf(W)}:= Wf(W)-\wt \E \wt W (\partial_{\wt W} f)(W),
\end{align}
where  $\wt W$ is independent of $W$ defined as in (\ref{Mmatrix_00}) with $X$ being replaced 
with a complex Ginibre ensemble. Here $\partial_{\widetilde{W}}$ denoted the directional derivative
in the direction $\widetilde{W}$, the expectation in \eqref{ud_W} is with respect to this matrix.
In particular we have $\ud{ WG^{z}} = WG^{z} + \< G^z \>G^{z}$, where we used \eqref{normal_trace}.
The underline is  a simple renormalization: it is designed such that for Gaussian matrices
we have exactly $\E \ud{ Wf(W)}=0$ and for general matrices $\E \ud{ Wf(W)}$ is given by the third
and higher order cumulants, i.e. it effectively removes the second order cumulants.

Applying the identity in (\ref{identity_0}) on the first Green function factor $\wh {\ga_{aB}}$ in (\ref{p_3_o0})
and performing cumulant expansion formula on the resulting $\ud{W \gz}$ given in (\ref{ud_W}), we have
\begin{align}\label{example_expand}
	\frac{1}{n^{2}} \sum_{a \neq \ud B} \E \Big[ \wh{\ga_{aB}} \wh{\gb_{Ba}} \wh{\gb_{Ba}}\Big]=&-\frac{m^{z_1}}{n^{3}}\sum_{a \neq \ud B} \sum_{J} 
	\E \Big[\frac{\partial (\wh{\gb_{B a}} \wh{\gb_{Ba}})}{\partial w_{Ja}} \ga_{J B} \Big]
	+\frac{m^{z_1}}{n^{2}}\sum_{a \neq \ud B} \E \Big[\ga_{a B} \wh{\gb_{B a}} \wh{\gb_{B a}} \<\wh \ga \> \Big]\nonumber\\
	&-\frac{\mathfrak{m}^{z_1}}{n^{3}}\sum_{a \neq \ud B} \sum_{j} 
	\E \Big[\frac{\partial(\wh{\gb_{B a}} \wh{\gb_{ B a}}) }{\partial w_{j\bar a}} \ga_{j B} \Big]
	+\frac{\mathfrak{m}^{z_1}}{n^{2}}\sum_{a \neq \ud B}
	\E \Big[\ga_{\bar a B} \wh{\gb_{B a}} \wh{\gb_{B a}} \<\wh \ga \> \Big]\nonumber\\
	&+\sum_{p+q+1 = 3}^4 \Big(H^{(1)}_{p+1,q}+H^{(2)}_{q,p+1}\Big)+O_\prec(n^{-3/2-\epsilon}),
\end{align}
where $H^{(1)}_{p+1,q}$ and $H^{(2)}_{q,p+1}$ are the third and fourth order terms given by
\begin{align}\label{example_expand_term}
	H^{(1)}_{p+1,q}:=&-\frac{m^{z_1}}{n^2}  \frac{c^{(p+1,q)} }{p!q!n^{\frac{p+q+1}{2}}} 
	\Big( \sum_{a \neq \ud B}  \sum_{J} 
	\E \Big[\frac{\partial^{p+q} (\wh{\gb_{B a}} \wh{\gb_{B a}}  \ga_{JB})}{\partial w^{p}_{aJ} \partial w^q_{Ja}}  \Big] \Big),\nonumber\\
	H^{(2)}_{q,p+1}:=&-\frac{\mathfrak{m}^{z_1}}{n^2}  \frac{ c^{(q,p+1)} }{p!q!n^{\frac{p+q+1}{2}}} 
	\Big( \sum_{a \neq \ud B}  \sum_{j} 
	\E \Big[\frac{\partial^{p+q} (\wh{\gb_{B a}} \wh{\gb_{B a}}  \ga_{j B})}{\partial w^{p}_{\bar aj} \partial w^q_{j \bar a}}  \Big] \Big),
\end{align}
with $c^{(p,q)}$ being the $(p,q)$-th cumulants of the normalized i.i.d. entries $\sqrt{n}w_{aB}$. 
The last error term in \eqref{example_expand}
is from the fifth order cumulants using the local law in (\ref{localg}) and the moment condition
 in (\ref{eq:hmb}), since they contain at least one off-diagonal Green function 
 factor that contributes $O_\prec(n^{-\epsilon})$.

Recalling Statement (1) below Lemma \ref{lemma:expand}, the leading second order terms of degree three from the first two lines of (\ref{example_expand}) are obtained by index replacements as in (\ref{aaexpl}), \ie
\begin{align}\label{leading}
	\frac{m^{z_1}m^{z_2}}{n^{3}} \sum_{a \neq \ud B \neq \ud J} \E \Big[ \wh{\ga_{JB}} \wh{\gb_{BJ}} \wh{\gb_{Ba}}\Big]+\frac{\m^{z_1}\overline{\m^{z_2}}}{n^{3}} \sum_{a \neq \ud B \neq j} \E \Big[ \wh{\ga_{jB}} \wh{\gb_{Bj}} \wh{\gb_{Ba}}\Big]=O_\prec(n^{-3/2-\epsilon}),
\end{align}
where the last estimate follows directly from (\ref{unmatch_estimate_improve}) since the number of summation indices $l=3$. In addition, from Statement (2) and (4) below Lemma \ref{lemma:expand}, all the other second order terms with higher degrees
and the cases with an index coincidence, \ie $J=\bar a$ or $B$, $j=a$ or $\ud{B}$ are given by
\begin{align}
	\sum_{P^o_{d'} \in {\mathcal{A}}^o_{>3}} \E [ P^o_{d'} ] +\frac{1}{n}\sum_{P^o_{d'} \in {\mathcal{C}}^o_{\ge 1}} 
	\E [ P^o_{d'} ]=O_\prec(n^{-3/2-\epsilon}),
\end{align}
with ${\mathcal{A}}^o_{>3}\subset  
\mathcal{P}^o_{4}$ and ${\mathcal{C}}^o_{\ge 1}\subset  \bigcup_{d'\ge 1}\mathcal{P}^o_{d'}$, where the last estimate follows from (\ref{unmatch_estimate_improve}) with degree $d=4$. Moreover, from Statement (4) below Lemma \ref{lemma:expand}, the fourth order terms in (\ref{example_expand}) are given by
\begin{align}\label{fourth_terms}
	&\sum_{p+q+1= 4} \left(H^{(1)}_{p+1,q}+ H^{(2)}_{q,p+1}\right)=\frac{1}{n}\sum_{P^o_{d'} \in {\mathcal{C}}^o_{\ge 3}} \E [ P^o_{d'} ]+O_{\prec}(n^{-2})=O_{\prec}(n^{-2}),
\end{align}
with ${\mathcal{C}}^o_{\ge 3}\subset  \bigcup_{d'\ge d}\mathcal{P}^o_{d'}$, where the last estimate follows from (\ref{unmatch_estimate}).

Next we focus on the most critical
third order terms in (\ref{example_expand_term}) with $p+q+1=3$. 
Recalling \cite[Remark 4.6]{maxRe} or Statement (3) below
Lemma \ref{lemma:expand}, the error term of size $O_\prec(n^{-3/2})$ in~(\ref{aa})
indeed comes from the third order terms with an index coincidence, \ie $J=B$ or 
$\bar a$, $j=a$ or $\ud B$. More precisely, we have
\begin{align}\label{third_terms}
	\sum_{p+q+1=3} \left(H^{(1)}_{p+1,q}+ H^{(2)}_{q,p+1}\right)=&\frac{1}{\sqrt{n}}
	\sum_{P^o_{d'} \in {\mathcal{B}}^o_{\ge 3}} \E [ P^o_{d'} ]+\sum_{p+q+1=3} H^{(1)}_{p+1,q}
	\Big|_{J=B,\bar a}+ \sum_{p+q+1=3} H^{(2)}_{q,p+1}\Big|_{j=a,\ud B}\nonumber\\
	=&\sum_{p+q+1=3} H^{(1)}_{p+1,q}\Big|_{J=B,\bar a}+\sum_{p+q+1=3} H^{(2)}_{q,p+1}\Big|_{j=a,\ud B}
	+O_\prec(n^{-2}),
\end{align}
where ${\mathcal{B}}^o_{\ge 3}\subset  \bigcup_{d'\ge 3}\mathcal{P}^o_{d'}$ and the last error term $O_\prec(n^{-2})$ follows from (\ref{unmatch_estimate}). 

Using that $m^{z} \lesssim \sqrt{n}\eta$ with $\eta=n^{-1+\epsilon}$, it is easy to check 
that $H^{(1)}_{p+1,q}$ in (\ref{example_expand_term}) with $J=B,\overline{a}$ can be bounded by
$$\sum_{p+q+1=3} H^{(1)}_{p+1,q}\Big|_{J=B,\bar a}=O_\prec(n^{-2+\epsilon}).$$
Moreover, by direct computations, $H^{(2)}_{q,p+1}$ in (\ref{example_expand_term}) with $j=a$ contain at least three off-diagonal Green function entries and thus can be bounded by
$$\sum_{p+q+1=3} H^{(2)}_{q,p+1}\Big|_{j=a} =O_\prec\big( n^{-3/2}\Psi^3\big) =O_\prec\big(n^{-3/2-3\epsilon}\big).$$
Similarly, for $j=\ud{B}$, the resulting terms from $H^{(2)}_{q,p+1}\Big|_{j=\ud{B}}$ with $(p,q)=(2,0)$ and $(1,1)$ contain at least two off-diagonal Green function entries and thus can be bounded by $O_\prec(n^{-3/2-2\epsilon})$.
For the last remaining case with $(p,q)=(0,2)$ and $j=\ud{B}$, by direct computations, we have 
\begin{align}\label{expand_result}
	H^{(2)}_{2,1}\Big|_{j=\ud{B}}=&-\frac{\mathfrak{m}^{z_1}}{n^{7/2}}  \sum_{a \neq \ud B} \sum_{j}
	\E \Big[\frac{\partial^{2} (\wh{\gb_{B a}} \wh{\gb_{B a}}  \ga_{j B})}{ \partial w^2_{j \bar a}} \delta_{j=\ud{B}} \Big]\nonumber\\
	=&-\frac{\mathfrak{m}^{z_1}}{n^{7/2}} \sum_{a \neq \ud B} \E \Big[\gb_{B \ud B} \gb_{\bar a a }\gb_{B \ud B} \gb_{\bar a a } \ga_{\ud B B}\Big]+O_\prec(n^{-3/2-\epsilon})\nonumber\\
	=&-\frac{1}{n^{3/2}} \big(\m^{z_1}\big)^2 (\overline{\m^{z_2}})^4  +O_\prec(n^{-3/2-\epsilon}),
\end{align}
where we also used the local law trivially. To sum up, we have obtained from (\ref{example_expand}) the improved estimate in (\ref{p_3_o0}). Thus the last third order term in (\ref{third}) satisfies
\begin{align}\label{p_3_o0_result}
	\frac{\sqrt{n}}{n^{2}} \sum_{a \neq \ud{B}}\E \Big[ {\ga_{aB}} {\gb_{Ba}} {\gb_{Ba}} \Big]=\frac{C}{n} \big(\m^{z_1}\big)^2 (\overline{\m^{z_2}})^4 +O_\prec(n^{-1-\epsilon}).
\end{align}
Notice that $\m^{z}=zu^{z}$ from (\ref{Mmatrix_0}) and both $f(z)$ and $u^{z}$ are radial functions in $|z|$,  then
\begin{align}\label{pq_m_int}
	\int_{\C} \Delta f(z) \big(\m^{z}\big)^p  (\overline{\m^{z}})^q \dd^2 z=0, \qquad \mbox{ unless $p=q$}.
\end{align}
Thus the leading deterministic term in (\ref{p_3_o0_result}) satisfies the integral condition in (\ref{delta_int}).

The other third order terms in (\ref{third}) with restricted summations $a\neq \ud{B}$ can be estimated similarly. 
More precisely, applying (\ref{identity_0}) to the $\ga_{aB}$ factor in each term in (\ref{third}) with 
$a\neq \ud{B}$~(omitting the $n^{1/2}$ prefactor) and performing cumulant expansions as in (\ref{example_expand}),
 the corresponding third order expansion terms with an index coincidence $j=\ud{B}$ that contribute 
 $O_\prec(n^{-3/2})$ as in (\ref{expand_result}) are given by, in general,
\begin{align}\label{general_third}
	\sum_{p+q+1=3} H^{(2)}_{q,p+1}\Big|_{j=\ud{B}}=&-\sum_{p+q+1=3}  \frac{\mathfrak{m}^{z_1}}{n^{7/2}} \sum_{a \neq \ud B}  \sum_{j} 
	\E \Big[\frac{\partial^{2}\big( (\wh{\F_t^{z_0}})^{\alpha_0}\wh{G^{z}_{x_2 y_2}} \wh{G^{z'}_{x_3y_3}}  \ga_{j B}\big)}{\partial w^{p}_{\bar aj} \partial w^q_{j \bar a}} \delta_{j=\ud B} \Big]\nonumber\\
	=&\mbox{M-terms}(z_1,z_2)+O_\prec(n^{-3/2-\epsilon})
\end{align}
where $\alpha_0=0$ or $1$, $z_0,z,z'$ stands for either $z_1$ or $z_2$ and $(x_2,y_2,x_3,y_3)$ 
is an assignment of two $a$'s and two $B$'s. The last line follows trivially from the local law in (\ref{localg}) 
and $\mbox{M-terms}(z_1,z_2)$ is a linear combination of products of $\m^{z_1}$ and $\m^{z_2}$ 
(multiplied by $n^{-3/2}$) that vanishes after the $z$-integrations using (\ref{pq_m_int}).

In summary, we have the following improved estimate for all the third order terms with $a \neq \ud{B}$,
\begin{align}\label{third_off_goal}
	\sum_{p+q+1=3}\mathcal{K}^{z_1,z_2}_{p+1,q}\Big|_{a \neq \ud B}=\mathcal{M}^{(1)}_{p+1,q}(z_1,z_2)+O_\prec(n^{-1-\epsilon}),
\end{align}
where $\mathcal{M}^{(1)}_{p+1,q}(z_1,z_2)$ is a linear combination of products of $\m^{z_1}$ and $\m^{z_2}$
 (multiplied by $n^{-1}$) satisfying the integral condition in (\ref{delta_int}). 
 Therefore using the $L^1$ norm in (\ref{deltaf}), we have proved Lemma \ref{lemma_third_order}.
\end{proof}

We now finish this section with the proof of Lemma \ref{lemma:expand}.

\begin{proof}[Proof of Lemma \ref{lemma:expand}]
	The expansion in (\ref{aa}) with the indicated $O_\prec(n^{-3/2})$ error was already proved 
	in~\cite[Lemma 4.8]{maxRe} for unmatched terms in (\ref{form}) without the $\wh \F$ prefactor, \ie $\alpha_0=0$.
	 As explained below~\cite[Eq (5.25)]{maxRe}, the same result (\ref{aa}) still hold true with a possible 
	 $\wh \F$ factor in front and we now briefly explain the minor modifications needed in the proof.
	  Starting from the cumulant expansions in \cite[Eq (4.54)-(4.55)]{maxRe}, we get extra
	   expansion terms from acting partial derivatives on $\wh \F$ using the latter differentiation rule 
	   in (\ref{rule_12}) for $\wh \F$. For these extra terms, the numbers of index assignments to the
	    row/column of Green function entries remain the same as the standard cases without $\wh \F$, 
	    thus they are still unmatched terms. 
	Moreover,  these extra terms have higher degrees than the initial term since taking partial 
	derivatives of $\wh \F$ typically yields an extra Green function factor from (\ref{rule_12}). 
	In this way we can prove (\ref{aa}) with the $\wh\F$ prefactor and for brevity we omit the details.

Next we focus on improving the last error term $\OO_\prec(n^{-3/2})$ in (\ref{aa}) slightly 
to $O_\prec(n^{-3/2-\epsilon})$ for $d\geq 4$ or $l\geq 3$. To present the proof in full generality, we always include
	the $\wh {\F^{z_0}}$ factor in front. 
	Let $P_d^o(x_1 \equiv a)$ be a given term in $\mathcal{P}_d^o$ 
	with an unmatched index $a$ satisfying (\ref{choose}) and without loss of generality 
	$x_1 \equiv a$, $y_1 \not\equiv a,\bar a$. Using the identity in (\ref{identity_0})
	on the first Green function factor $\wh{\gz_{ay_1}}$ and performing the cumulant expansions 
	similarly as in (\ref{example_expand}), we have
	\begin{align}\label{case1}
		\E[P_d^o(x_1 \equiv a)]=&
		-\frac{m^{z_1}}{n^{l+1}}\sum^{*}_{\mathcal{I}_{l_1,l_2}} \sum_{J} 
		\E \Big[\frac{\partial \big(\wh {\F_t^{z_0}} \prod_{i=2}^{d} \wh{G^{z_i}_{x_iy_i}}\big)}{\partial w_{Ja}} G^{z_1}_{J y_1} \Big]+\frac{m^{z_1}}{n^{l}}\sum^{*}_{\mathcal{I}_{l_1,l_2}}
		\E \Big[\wh {\F_t^{z_0}} G^{z_1}_{a y_1} \prod_{i=2}^{d} \wh{G^{z_i}_{x_iy_i}} \<\wh \ga \> \Big]\nonumber\\
		&-\frac{\mathfrak{m}^{z_1}}{n^{l+1}}\sum^{*}_{\mathcal{I}_{l_1,l_2}}\sum_{j} 
		\E \Big[\frac{\partial\big(\wh {\F_t^{z_0}} \prod_{i=2}^{d} \wh{ G^{z_i}_{x_iy_i}}\big)}{\partial w_{j\bar a}} G^{z_1}_{j y_1} \Big]+\frac{\mathfrak{m}^{z_1}}{n^{l}}\sum^{*}_{\mathcal{I}_{l_1,l_2}} \E \Big[\wh {\F_t^{z_0}} G^{z_1}_{\bar a y_1}
		\prod_{i=2}^{d} \wh{G^{z_i}_{x_iy_i}} \<\wh \ga \> \Big]
		\nonumber\\
		&+\sum_{p+q+1=3}^{4} \Big(H^{(1)}_{p+1,q}+H^{(2)}_{q,p+1}\Big)+O_\prec(n^{-3/2-\epsilon}),
	\end{align}
	where $\sum^*_{\mathcal{I}}$ is the restricted summation defined in (\ref{distinct_sum}), 
	and the last error term is from the fifth order cumulants using the local law in (\ref{localg}) 
	and the moment condition in (\ref{eq:hmb}). The third and fourth order terms
	$H^{(1)}_{p+1,q}$ and $H^{(2)}_{q,p+1}$ with $p+q+1=3,4$ are given by
	\begin{align}\label{higher}
		H^{(1)}_{p+1,q}:=&- \frac{m^{z_1} c^{(p+1,q)} }{p!q! n^{\frac{p+q+1}{2}+l}}  \sum^{*}_{\mathcal{I}_{l_1,l_2}} \sum_{J} 
		\E \Big[\frac{\partial^{p+q} \big(\wh {\F_t^{z_0}} \prod_{i=2}^{d} \wh{G^{z_i}_{x_iy_i}}G^{z_1}_{J y_1}\big)}{\partial w^{p}_{aJ} \partial w^q_{Ja}}  \Big],\nonumber\\
		H^{(2)}_{q,p+1}:=	&- \frac{\mathfrak{m}^{z_1}c^{(q,p+1)} }{p!q! n^{\frac{p+q+1}{2}+l}}  \sum^{*}_{\mathcal{I}_{l_1,l_2}} 
		\sum_{j} \E \Big[\frac{\partial^{p+q} 
			\big(\wh {\F_t^{z_0}} \prod_{i=2}^{d} \wh{G^{z_i}_{x_iy_i}}G^{z_1}_{j y_1}\big)}{\partial w^{p}_{\bar aj} \partial w^q_{j \bar a}}  \Big].
	\end{align}
We refer to Statement (1)-(4) below Lemma \ref{lemma:expand} for a brief discussion on the second, third and fourth order terms in (\ref{case1}).

	Recall from \cite[Remark 4.6]{maxRe}, that the critical error term $\OO_\prec(n^{-3/2})$ in (\ref{aa}) 
	comes only  from the third order terms $H^{(1)}_{p+1,q}$ and $H^{(2)}_{q,p+1}~(p+q+1=3)$ with an index coincidence, 
	\ie when $J$ or $j$ coincides with an original summation index in $\mathcal{I}_{l_1,l_2}$ or 
	its index conjugation; all other error terms were already smaller in \cite{maxRe}. 
	For a concrete example see (\ref{expand_result}) with an index coincidence $j=\ud{B}$.
	It then suffices to estimate, for $p+q+1=3$,
	\begin{align}\label{diagonal_estimate}
		H^{(2)}_{q,p+1}\Big|_{j \in \mathcal{I}_{l_1,l_2}}:=\frac{1}{n^{l}}  \sum^{*}_{\mathcal{I}_{l_1,l_2}} \frac{1}{n^{3/2}}
		\sum_{j} \E \Big[\frac{\partial^{2} \big(\wh {\F_t^{z_0}}
			\prod_{i=2}^{d} \wh{G^{z_i}_{x_iy_i}}G^{z_1}_{j y_1}\big)}{\partial w^{p}_{\bar aj} \partial w^q_{j \bar a}}  \Big( 	\sum_{k=1}^{l_{1}}\delta_{j=v_k}+\sum_{k=1}^{l_{2}} \delta_{j=\ud{V_k}} \Big)\Big],
	\end{align}
	and the other part $H^{(1)}_{p+1,q}$ in (\ref{higher}) restricted to $J \in \mathcal{I}_{l_1,l_2}$ can be estimated similarly.
	
	If we assume $d\geq 4$, then at least one $\wh G$ factor survives after taking derivatives 
	$\partial^2/\partial w^{p}_{\bar aj} \partial w^q_{j \bar a}$ with $p+q+1=3$, so (\ref{diagonal_estimate}) can be bounded by
	$O_\prec(n^{-3/2-\epsilon})$ since $|\wh{G_{xy}}| \prec \Psi=n^{-\epsilon}$. Otherwise if $d=3$, then there is only 
	one scenario when no $\wh{G}$ survives after $\partial^2/\partial w^{p}_{\bar aj} \partial w^q_{j \bar a}$, \ie 
	when one partial derivative acts on $\wh{G^{z_2}_{x_2y_2}}$ and another one acts on $\wh{G^{z_3}_{x_3y_3}}$.
	If we further assume $l=l_1+l_2\geq 3$, then the resulting products of Green function entries with the index coincidence $j=v_k~(1\leq k\leq l_1)$ or $j=\ud{V_k}~(1\leq k\leq l_2)$ contain at least one off-diagonal Green function factor yielding an additional $O_\prec(n^{-\epsilon})$ 
	from the local law. Hence  (\ref{diagonal_estimate}) can be bounded by $O_\prec(n^{-3/2-\epsilon})$ and
	this finishes the proof of Lemma \ref{lemma:expand}. 
\end{proof}

\section{Proof of Lemma \ref{lemma:aaBB}-\ref{lemma4} and Lemma \ref{lemma5}}\label{sec:some_lemma}
In this section we prove some technical lemmas used in the Sections \ref{subsec:part2}--\ref{subsec:part3}. 
\begin{proof}[Proof of Lemma \ref{lemma:aaBB}]
	Recall that the eigenvalues of $H^{z}$ are $\{\lambda^z_{\pm i}\}_{i \in \llbracket 1,n \rrbracket}$ with 
	$\lambda^z_{-i}=-\lambda^z_{i}$, and the corresponding normalized eigenvectors of 
	$\lambda^z_{\pm i}$ are given by $\mathbf{w}^z_{\pm i}=(\mathbf{u}^z_{i}, \pm \mathbf{v}^z_{i})$.
	Using spectral decomposition, we have
	\begin{align}
		\Im G^z_{aa}(\ii \eta)=\Im\Big( \sum_{j=1}^{n} \frac{|\<\mathbf{e}_{a}, \mathbf{u}^z_{i}\>|^2}{\lambda^z_{j}-\ii \eta}
		+\sum_{j=1}^{n} \frac{|\<\mathbf{e}_{a},\mathbf{u}^z_{j}\>|^2}{-\lambda^z_{j}-\ii \eta} \Big)
		=2\Big( \sum_{\lambda^z_j\leq n^{-\zeta}}+
		\sum_{\lambda^z_j\geq n^{-\zeta}}\Big)\frac{ \eta |\<\mathbf{e}_{a}, \mathbf{u}^z_{j}\>|^2}{(\lambda^z_{j})^2+\eta^2},
	\end{align}
	with any small $\zeta>0$. Recall that $\mathbf{w}^z_{i}$ denotes the normalized eigenvector of $\lambda^z_i$. By the delocalization of eigenvectors in Corollary~\ref{cor:eigenvector} we have $|\<\mathbf{w}^z_{i}, \mathbf{x}\>|^2 \prec n^{-1}$, for eigenvectors corresponding to $|\lambda^z_i|\leq n^{-\zeta}$, and for any deterministic unit vector $\mathbf{x}\in \C^{n}$.
		We thus have
	\begin{align}\label{temp1}
		\sum_{\lambda_j^z\leq n^{-\zeta}}\frac{\eta |\<\mathbf{e}_{a}, \mathbf{u}^z_{j}\>|^2}{(\lambda^z_{j})^2+\eta^2} 
		\prec\frac{\eta}{n}\sum_{\lambda^z_j\leq n^{-\zeta}}\frac{1}{(\lambda^z_{j})^2+\eta^2}  \leq \Im \<\gz(\ii \eta)\>.
	\end{align}
	For the remaining eigenvalues $\lambda^z_j\geq n^{-\zeta}$, using that $\sum_j |\<\mathbf{e}_{a}, \mathbf{u}^z_{j}\>|^2 \le \| \mathbf{e}_{a}\|^2=1$, 
	we have
	$$\sum_{\lambda^z_j\geq n^{-\zeta}}\frac{ \eta|\<\mathbf{e}_{a}, \mathbf{u}^z_{j}\>|^2}{(\lambda^z_{j})^2+\eta^2} \leq 
	n^{2\zeta}\eta\sum_j |\<\mathbf{e}_{a}, \mathbf{u}^z_{j}\>|^2 \le n^{2\zeta}\eta,$$
	with $\zeta>0$ an arbitrary small number. Combining this 
	with (\ref{temp1}), we conclude the proof of Lemma~\ref{lemma:aaBB}.
\end{proof}

\begin{proof}[Proof of Lemma \ref{lemma4}]
	
	We start with the first estimate in (\ref{estimate4}). Note that by $\wh{\ga_{a \bar a}} =\ga_{a \bar a}  - \m^{z_1} $ 
	and $|\m^{z_1}|\lesssim 1$, we have 
	\begin{align}\label{right_side}
		\Big|\frac{1}{n} \sum_{a}\E \big[ \wh{\F_t^{z_2}} G^{z_1}_{a \bar a } G^{z_1}_{  a \bar a} 
		 G^{z_1}_{  a \bar a}  \big] \Big| \lesssim & 
		 \Big|\frac{1}{n} \sum_{a}\E \big[ \wh{\F_t^{z_2}} \wh{G^{z_1}_{a \bar a }}\big]\Big|+\Big|\frac{1}{n}
		 \sum_{a}\E \big[ \wh{\F_t^{z_2}} \wh{G^{z_1}_{a \bar a }} \wh{G^{z_1}_{  a \bar a}}  \big]\Big|\nonumber\\
		&+\Big|\frac{1}{n} \sum_{a}\E \big[ \wh{\F_t^{z_2}} \wh{G^{z_1}_{a \bar a }} 
		\wh{G^{z_1}_{  a \bar a}} \wh{G^{z_1}_{  a \bar a}}\big]\Big|.
	\end{align}
	For the first term on the right side of (\ref{right_side}), applying the identity in
	 (\ref{identity_0}) to $\wh{G^{z_1}_{a \bar a }}$ and performing cumulant expansions, we obtain
	\begin{align}\label{expand_3}	
		\frac{1}{n} \sum_{a}\E \big[ \wh{\F_t^{z_2}} \wh{G^{z_1}_{a \bar a }}\big]=&-\frac{m^{z_1}}{n^{2}}\sum_{a,J} \E \Big[\frac{\partial \wh{\F_t^{z_2}} }{\partial w_{Ja}} G^{z_1}_{J \bar a} \Big]+\frac{m^{z_1}}{n}\sum_{a}\E \Big[\wh{\F_t^{z_2}} G^{z_1}_{a \bar a} \<\wh{G^{z_1}} \>  \Big]\nonumber\\
		&-\frac{\m^{z_1}}{n^2}\sum_{a,j} \E \Big[\frac{\partial \wh{\F_t^{z_2}}}{\partial w_{j \bar a}} 
		G^{z_1}_{j \bar a} \Big]
		+\frac{\m^{z_1}}{n}\sum_{a}\E \Big[ \wh{\F_t^{z_2}} G^{z_1}_{\bar a \bar a} \<\wh{G^{z_1}} \>  \Big]
		+\OO_\prec(n^{-1/2-\epsilon})\nonumber\\
		= & \frac{\m^{z_1}}{n^{2}}\sum_{a,j} \E \Big[  G^{z_2}_{\bar a j}  G^{z_1}_{j \bar a} \Big]
		+ \OO_\prec(n^{-1/2-\epsilon}+n \eta^2),
	\end{align}	
	where we truncated the cumulant expansions at the second order with the error term $O_\prec(n^{-1/2-\epsilon})$
	from the local law in (\ref{localg}). In the last line we used \eqref{rule_12} to compute the 
	derivatives of $\wh\F$, and the error term $O_\prec(n \eta^2)$ was obtained from 
	Proposition \ref{lemma2} with $k=2$ and that 
	$m^{z} \sim \sqrt{n}\eta$. Furthermore, using the Cauchy-Schwarz inequality, 
	the Ward identity and Proposition \ref{lemma2} with $k=1$, we have  
	\begin{align}\label{bad}
		\Big|\frac{\m^{z_1}}{n^{2}}\sum_{a,j} \E \Big[  G^{z_2}_{\bar a j}  G^{z_1}_{j \bar a} \Big] \Big|
		 \prec \frac{\E[\Im \<G^{z_1}+ G^{z_2}\>]}{n \eta} =\OO_\prec(n^{-1/2}).
	\end{align}
	Thus from (\ref{expand_3}), the first term on the right side of (\ref{right_side}) is bounded by
	\begin{align}\label{expand_30}
		\Big|\frac{1}{n} \sum_{a}\E \big[ \wh{\F_t^{z_2}} \wh{G^{z_1}_{a \bar a }}\big]\Big|=\OO_\prec(n^{-1/2}).
	\end{align}
	The second term in (\ref{right_side}) with two $\wh G$ factors can be handled similarly 
	with an even better error term,
	since we have more $\wh G$ factors and we always pick up additional small 
	diagonal $\gz_{\bar a\bar a}$ for the leading terms in the expansion~(see~(\ref{aaexpl})) 
	due to the $a/\bar a$ assignments. 
	More precisely, we have as in (\ref{expand_3}), \ie
		\begin{align}\label{expand_301}
		\frac{1}{n} \sum_{a}\E \big[ \wh{\F_t^{z_2}} \wh{G^{z_1}_{a \bar a }} \wh{G^{z_1}_{a \bar a }}\big]
		=&-\frac{\m^{z_1}}{n^2}\sum_{a,j} \E \Big[\frac{\partial \wh{\F_t^{z_2}} \wh{G^{z_1}_{a \bar a }}}{\partial w_{j \bar a}} G^{z_1}_{j \bar a} \Big]
		+ \OO_\prec(n^{-1/2-\epsilon}+n \eta^2)\nonumber\\
		= & \frac{\m^{z_1}}{n^{2}}\sum_{a,j} \E \Big[G^{z_2}_{\bar a j} \wh{G^{z_1}_{a \bar a }} G^{z_1}_{j \bar a} \Big]+ \frac{\m^{z_1}}{n^{2}}\sum_{a,j} \E\Big[\wh{\F_t^{z_2}} G^{z_1}_{ a j} G^{z_1}_{ \bar a \bar a}  G^{z_1}_{j \bar a} \Big]\nonumber\\
		&+\OO_\prec(n^{-1/2-\epsilon}+n \eta^2),
	\end{align}	
where the first two terms are from acting $\partial/\partial w_{j\bar a}$ on $\wh{\F_t^{z_2}}$ and $\wh{G^{z_1}_{a \bar a }}$ respectively. Using that $|\wh{G^{z_1}_{a \bar a }}|\prec n^{-\epsilon}$, $G^{z_1}_{ \bar a \bar a}=m^{z_1}+O_\prec(n^{-\epsilon})$ and $m^{z_1} \sim \sqrt{n}\eta$, we obtain
		\begin{align}\label{expand_31}
				\Big|\frac{1}{n} \sum_{a}\E \big[ \wh{\F_t^{z_2}} \wh{G^{z_1}_{a \bar a }} \wh{G^{z_1}_{a \bar a }}\big]\Big|	=\OO_\prec(n^{-1/2-\epsilon}+n \eta^2).
		\end{align}
Similarly, the same upper bound applies to the last term in (\ref{right_side}) with three $\wh G$ factors, \ie
	\begin{align}\label{expand_32}
		\Big|\frac{1}{n} \sum_{a}\E \big[ \wh{\F_t^{z_2}} \wh{G^{z_1}_{a \bar a }} \wh{G^{z_1}_{  a \bar a}} \wh{G^{z_1}_{  a \bar a}}\big]\Big|=\OO_\prec(n^{-1/2-\epsilon}).
	\end{align}
	Hence using (\ref{expand_30}) and (\ref{expand_31})-(\ref{expand_32}), we have proved the first estimate in (\ref{estimate4}).
	
	We next obtain a slightly better estimate than (\ref{bad}) for $|z_1-z_2| \geq n^{-\gamma}$. We continue to expand the left side of (\ref{bad}) by applying (\ref{identity_0}) to $\wh{\ga_{j \bar a}}$ and performing cumulant expansions as in (\ref{expand_3}), \ie
	\begin{align}\label{self_g1g2_1}
		\frac{1}{n^{2}}\sum_{a,j} \E \Big[ G^{z_1}_{j \bar a} G^{z_2}_{\bar a j}   \Big]=&\frac{1}{n^{2}}\sum_{a,j} \E \Big[  \wh{G^{z_1}_{j \bar a }}  G^{z_2}_{ \bar a j} \Big]+\OO_\prec(n^{-1})\nonumber\\
		=&\frac{\m^{z_1} \overline{\m^{z_2}}}{n^{2}}\sum_{a,j'}  \E \Big[ G^{z_1}_{j' \bar a} G^{z_2}_{\bar a j'}  \Big]+\OO_\prec(n^{-1/2-\epsilon}+n\eta^2),
	\end{align}
	where we have replaced one pair of the index $j$ with a fresh index $j'$ as in (\ref{aaexpl}) for the leading term, 
	which can be moved to the left side with a combined stability factor $1-\m^{z_1} \overline{\m^{z_2}}$.
	Thus using the lower bound of $1-\m^{z_1} \overline{\m^{z_2}}$ in (\ref{stable_0}) 
	we conclude from (\ref{self_g1g2_1}) that, 
	\begin{align}\label{improve_zz}
		\Big|\frac{1}{n^{2}}\sum_{a,j} \E \Big[ G^{z_1}_{j \bar a} G^{z_2}_{\bar a j}   \Big]\Big|
		=\OO_\prec( n^{-1/2-\epsilon+\gamma}), \qquad |z_1-z_2|\geq n^{-\gamma}.
	\end{align}
	Hence combining with (\ref{expand_3}), (\ref{expand_31})-(\ref{expand_32}) 
	and (\ref{right_side}), we have proved (\ref{estimate7}).
	
	We next prove the second estimate in (\ref{estimate4}) similarly.  
	Note that by $\wh{\ga_{a \bar a}} =\ga_{a \bar a}  - \m^{z_1} $
	and $|\m^{z_1}|\lesssim 1$ we have
	\begin{align}\label{some_term}
		\frac{1}{n} \sum_{a} \E \Big[ \ga_{a \bar a} \gb_{ a \bar a} &\gb_{ a \bar a} \Big]
		= \m^{z_1} ({\m^{z_2}})^2 + \frac{(\mb)^2}{n} \sum_{a} \E \Big[ \wh{\ga_{a \bar a}} \Big]
		+\frac{2\ma\mb}{n} \sum_{a} \E \Big[ \wh{\gb_{ a \bar a}} \Big]\nonumber\\
		&+ \frac{\mb}{n} \sum_{a} \E \Big[\wh{\ga_{a \bar a}} \wh{\gb_{a \bar a}}\Big] 
		+\frac{\ma}{n} \sum_{a} \E \Big[\big(\wh{\gb_{a \bar a}}\big)^2\Big]
		+\frac{1}{n} \sum_{a} \E \Big[ \wh{\ga_{a \bar a}} \wh{\gb_{ a \bar a}} \wh{\gb_{ a \bar a}} \Big].
	\end{align}
	For the terms above with only one $\wh \gz$ factor, we perform similar cumulant expansions
	 as in (\ref{expand_3}) without the $\wh \F$ factor. Recall that the leading expansion term that 
	 contribute $O_\prec(n^{-1/2})$ in (\ref{bad}) is indeed from acting $\partial/\partial w_{j\bar a}$
	  on $ \wh{\F_t^{z_2}}$ and all other expansion terms are much smaller, \ie bounded by 
	  $O_\prec(n^{-1/2-\epsilon}+n\eta^2)$ as in (\ref{expand_3}). Hence we obtain a better 
	  estimate, \ie for $z=z_1$ or $z_2$, 
	$$\Big|\sum_{a} \E \big[ \wh{\gz_{a \bar a}} \big]\Big|=O_\prec(n^{-1/2-\epsilon}+n\eta^2).
	$$ 
	The other terms in (\ref{some_term}) with two or more $\wh \gz$ factors are similar;
	we can obtain similar estimates as in (\ref{expand_301})-(\ref{expand_32}) without the $\wh \F$ factor. 
	Thus we conclude from (\ref{some_term}) that
	\begin{align}\label{expand_5}
		\frac{1}{n} \sum_{a} \E \Big[ \ga_{a \bar a} \gb_{ a \bar a} \gb_{ a \bar a} \Big]	
		=&\m^{z_1} ({\m^{z_2}})^2 +\OO_\prec(n^{-1/2-\epsilon}+n\eta^2).
	\end{align}

	Finally we prove the estimate in (\ref{estimate5}). Using that $\wh{\gz}=\gz -M^z$, we have
	\begin{align}\label{expand_10}
		\frac{1}{n} \sum_{a}\E \big[ G^{z_1}_{a \bar a } G^{z_2}_{ \bar a a}  G^{z_2}_{ \bar a a}  \big]
		=\frac{\m^{z_1}}{n} \sum_{a}\E \big[  G^{z_2}_{ \bar a a} G^{z_2}_{ \bar a a} \big]
		+\frac{1}{n} \sum_{a}\E \big[ \wh{G^{z_1}_{a \bar a }} G^{z_2}_{ \bar a a} G^{z_2}_{ \bar a a} \big].
	\end{align}
The first term on the right side of (\ref{expand_10}) can be estimated similarly as in
 (\ref{some_term})-(\ref{expand_5})~(with a similar $a/\bar a$ assignment though for three $G$'s), \ie
	\begin{align}\label{expand_100}
		\frac{\m^{z_1}}{n} \sum_{a}\E \big[  G^{z_2}_{ \bar a a} G^{z_2}_{ \bar a a} \big]
		=\m^{z_1} \big(\overline{\m^{z_2}}\big)^2+O_\prec(n^{-1/2-\epsilon}+n\eta^2).
	\end{align}
	The second term in (\ref{expand_10}) is more critical since the first row index $a$ in 
	$\wh{G^{z_1}_{a \bar a }}$ will be paired with the column index $a$ in the other two 
	$G^{z_2}_{ \bar a a}$ factors so that we cannot pick up a small $m^z$ factor in 
	the expansion~(see (\ref{aaexpl})).  Applying (\ref{identity_0}) to $\wh{G^{z_1}_{a \bar a }}$ 
	and performing cumulant expansions as in (\ref{expand_3}), we obtain 
	\begin{align}\label{expand_1}
		\frac{1}{n} \sum_{a}\E \big[ \wh{G^{z_1}_{a \bar a }} G^{z_2}_{ \bar a a} G^{z_2}_{ \bar a a} \big]
		=&-\frac{m^{z_1}}{n^{2}}\sum_{a,J} 
		\E \Big[\frac{\partial (G^{z_2}_{ \bar a a} G^{z_2}_{ \bar a a}) }{\partial w_{Ja}} G^{z_1}_{J \bar a} \Big]
		+\frac{m^{z_1}}{n}\sum_{a}\E \Big[G^{z_1}_{a \bar a} G^{z_2}_{ \bar a a} G^{z_2}_{ \bar a a} 
		\<\wh{G^{z_1}} \>  \Big]\nonumber\\
		&-\frac{\m^{z_1}}{n^2}\sum_{a,j} 
		\E \Big[\frac{\partial (G^{z_2}_{ \bar a a} G^{z_2}_{ \bar a a})}{\partial w_{j \bar a}} G^{z_1}_{j \bar a} \Big]
		+\frac{\m^{z_1}}{n}\sum_{a}\E \Big[ G^{z_1}_{\bar a \bar a} G^{z_2}_{ \bar a a} 
		G^{z_2}_{ \bar a a}\<\wh{G^{z_1}} \>  \Big]
		+\OO_\prec(n^{-1/2-\epsilon})\nonumber\\
		= &\frac{2\m^{z_1}\overline{\m^{z_2}}}{n^{2}}\sum_{a,j} \E \Big[  G^{z_2}_{\bar a j}  
		 G^{z_2}_{\bar a a} G^{z_1}_{j \bar a} \Big]+ \OO_\prec(n^{-1/2-\epsilon}+n \eta^2),
	\end{align}	
	where we replaced one pair of the index $a$~(one from the row and one from the column) 
	with a fresh index $j$ as in (\ref{aaexpl}) for the leading term. Furthermore, using the Cauchy-Schwarz
	 inequality, the Ward identity and Proposition \ref{lemma2} for $k=1$, we have
	\begin{align}
		\Big|\frac{2\m^{z_1}\overline{\m^{z_2}}}{n^{2}}\sum_{a,j} \E \Big[  G^{z_2}_{\bar a j}  
		 G^{z_2}_{\bar a a} G^{z_1}_{j \bar a} \Big]\Big|\prec\frac{1}{n^2} \sum_{a,j} \
		 E \Big[  \big|G^{z_2}_{\bar a j}  G^{z_1}_{j \bar a}\big| \Big]
		 \prec \frac{\E[\Im \<G^{z_1}+G^{z_2}\>]}{n \eta} =\OO_\prec(n^{-1/2}).
	\end{align}
	Combining this with (\ref{expand_10})-(\ref{expand_1}), we have proved (\ref{estimate5}).
	 If we further assume $|z_1-z_2| \geq n^{-\gamma}$, then
	$$\Big|\frac{2\m^{z_1}\overline{\m^{z_2}}}{n^{2}}\sum_{a,j} \E \Big[  G^{z_2}_{\bar a j}   
	G^{z_2}_{\bar a a} G^{z_1}_{j \bar a} \Big]\Big| \lesssim \Big|\frac{2\m^{z_1}(\overline{\m^{z_2}})^2}{n^{2}}
	\sum_{a,j} \E \Big[  G^{z_2}_{\bar a j}    G^{z_1}_{j \bar a} \Big]\Big|+O_\prec(n^{-1/2-\epsilon})
	=O_\prec(n^{-1/2-\epsilon+\gamma}),$$
	where we also used the local law~(\ref{localg}) and the improved bound in (\ref{improve_zz}). 
	This proves the last estimate in (\ref{estimate6}) for $|z_1-z_2| \geq n^{-\gamma}$. 
\end{proof}

\begin{remark}\label{handwave_real}
	The slightly better error term in the second estimate in~\eqref{estimate4}, resulting from $a/\bar a$ 
	assignments in Remark~\ref{remark:aa}, is only valid for the matrix $X$ being complex-valued. 
	For the real cases, we will not distinguish between the row and column of a Green function due
	 to a slightly different differentiation rule from~(\ref{rule_12}). Instead using the real version of 
	 differentiation rule we get additional terms \eg $G^{z} G^{\bar z}$. For these extra terms we need 
	 to gain a little extra smallness from restricting $z$ away from the real axis, similarly
	  to~(\ref{self_g1g2_1})-(\ref{improve_zz}) for $z_1=z$ and $z_2=\bar z$.	
\end{remark}

\begin{proof}[Proof of Lemma \ref{lemma5}]

	We start with the estimate in (\ref{estimate_1}). Note that
	\begin{align}\label{T_term}
		\frac{1}{n^2} \sum_{a, B}\E\Big[ \ga_{aB} \ga_{aB} \gb_{Ba} \gb_{Ba}\Big]
		=&\frac{1}{n^2} \sum_{a, B}\E\Big[ \wh{\ga_{aB}} \ga_{aB} \gb_{Ba} \gb_{Ba}\Big]
		+\frac{\m^{z_1}}{n^2} \sum_{a}\E\Big[  \ga_{a \bar a} \gb_{\bar a a} \gb_{\bar a a}\Big]\nonumber\\
		=&\frac{1}{n^2} \sum_{a, B}\E\Big[ \wh{\ga_{aB}} \ga_{aB} \gb_{Ba} \gb_{Ba}\Big]
		+\frac{1}{n} (\m^{z_1})^2 (\overline{\m^{z_2}})^2+\OO_\prec(n^{-1-\epsilon}),
	\end{align}
	where we used that $\wh{\gz}=\gz-M^z$ and the local law in (\ref{localg}).
	 Applying the identity in (\ref{identity_0}) to $\wh{\ga_{aB}}$ in (\ref{T_term}) 
	 and performing cumulant expansions, we have 
	\begin{align}\label{term2}
		\frac{1}{n^2} \sum_{a, B}\E\Big[ \wh{\ga_{aB}} \ga_{aB} &\gb_{Ba} \gb_{Ba}\Big]
		=-\frac{m^{z_1}}{n^{3}}\sum_{a, B}\sum_{J} 
		\E \Big[\frac{\partial (\ga_{aB} \gb_{Ba} \gb_{Ba})}{\partial w_{Ja}} \ga_{J B} \Big]\nonumber\\
		&+\frac{m^{z_1}}{n^{2}}\sum_{a, B}\E \Big[\ga_{aB} \ga_{aB} \gb_{Ba} \gb_{Ba} \<\wh \ga \> \Big]
		-\frac{\m^{z_1}}{n^{3}}\sum_{a, B} \sum_{j} 
		\E \Big[\frac{\partial (\ga_{aB} \gb_{Ba} \gb_{Ba})}{\partial w_{j \bar a}} \ga_{j B} \Big]\nonumber\\
		&+\frac{\m^{z_1}}{n^{2}}\sum_{a, B} \E \Big[\ga_{\bar a B} \ga_{aB} \gb_{Ba} \gb_{Ba} \<\wh \ga \> \Big]
		+O_\prec(n^{-3/2}+n^{-1-\epsilon})\nonumber\\
		=:&T_1+T_2+T_3+T_4+O_\prec(n^{-3/2}+n^{-1-\epsilon}), 
	\end{align} 
	where 
	the error term $O_\prec(n^{-3/2})$ is from the unmatched third order terms using Proposition \ref{prop:unmatched}, and the other error term $O_\prec(n^{-1-\epsilon})$ is from the higher order terms using the local law naively.

	By direct computations, using that $m^{z} \sim n^{1/2} \eta$ and $\m^{z} \sim 1$, $T_1$, $T_2$ and $T_4$
	 can be bounded using the Cauchy-Schwarz inequality, the Ward identity and Proposition \ref{lemma2} 
	 with $k=1,2$, \ie 
	\begin{align}
		|T_1|+|T_2| \prec n^{1/2} \eta \frac{\E[\Im \<G^{z_1}\>]}{(n\eta)^4} =\OO_\prec(n^{-1-2\epsilon}), 
		\qquad |T_4| \prec \frac{\E[(\Im \<G^{z_1}\>)^2]}{(n\eta)^3}  =\OO_\prec( n^{-1-\epsilon}).
	\end{align}
    We then focus on the third term $T_3$, \ie by a direct computation using (\ref{rule_12}) and (\ref{localg}),
    \begin{align}\label{two_line}
    	T_3=&\frac{2\m^{z_1}\overline{\m^{z_2}}}{n^{3}}\sum_{a, B} \sum_{j} \E \Big[\ga_{aB} \gb_{Bj}\gb_{Ba} 
	 \ga_{j B}\Big]\nonumber\\
    	&+\frac{\m^{z_1}}{n^{3}}\sum_{a, B} \sum_{j} \E \Big[2\ga_{aB} \gb_{Bj}
	 \big(\gb_{\bar a a}-\overline{\m^{z_2}}\big) \gb_{Ba}  \ga_{j B}+\ga_{aj} \ga_{\bar a B}\gb_{Ba} \gb_{Ba}  \ga_{j B}\Big].
    \end{align}
    Estimating one $G$ factor trivially with (\ref{localg}) and using the Cauchy-Schwarz inequality, the
     Ward identity and (\ref{aaBB}), all subleading terms in the last line of~(\ref{two_line}) can be bounded by
      $\OO_\prec(n^{-1-\epsilon})$. Thus we have
	\begin{align}
		T_3
		=& \frac{2\m^{z_1} \overline{\m^{z_2}}}{n^{3}}\sum_{a, B} \sum_{j} \E \Big[ \ga_{aB} \gb_{Bj} 
		\gb_{Ba}  \ga_{j B}\Big]+ \OO_\prec(n^{-1-\epsilon}).
	\end{align}
	Compared to the initial term in (\ref{T_term}), we have replaced one pair of the index
	 $a$~(one from the row and one from the column) with a fresh index $j$ as in (\ref{aaexpl}).
	  Hence we conclude from (\ref{term2}) that
	\begin{align}\label{T}
		\frac{1}{n^2} \sum_{a, B}\E\Big[ \wh{\ga_{aB}} \ga_{aB} \gb_{Ba} \gb_{Ba}\Big]
		=\frac{2\m^{z_1} \overline{\m^{z_2}}}{n^{3}}\sum_{a, B} \sum_{j} 
		\E \Big[ \ga_{aB} \gb_{Bj} \gb_{Ba}  \ga_{j B}\Big]+\OO_\prec(n^{-1-\epsilon}).
	\end{align}
	For the leading term in (\ref{T}) (omitting irrelevant factor $2\m^{z_1} 
	\overline{\m^{z_2}}$), we further apply the Cauchy-Schwarz inequality, the Ward identity and (\ref{aaBB}) to obtain  
	\begin{align}\label{expand_result_4}
		\Big|\frac{1}{n^{3}}\sum_{a, B} \sum_{j} \E \Big[ &\ga_{aB} \gb_{Bj} \gb_{Ba} 
		 \ga_{j B}\Big] \Big|\prec \frac{\E[(\Im\<G^{z_1}+ G^{z_2}\>+\eta)^2]}{(n\eta)^2}= \OO_\prec(n^{-1}).
	\end{align}
	Hence the estimate in (\ref{estimate_1}) follows from (\ref{T_term}), (\ref{T}) 
	and (\ref{expand_result_4}).

	Next, we aim to obtain a slightly better estimate than (\ref{expand_result_4}) for 
	$|z_1-z_2| \geq n^{-\gamma}$. We expand the left side of (\ref{expand_result_4}) 
	by applying (\ref{identity_0}) to $\wh{\ga_{aB}}$ 
	and performing cumulant expansions as in (\ref{term2}), \ie 
	\begin{align}\label{self_g1g2}
		\frac{1}{n^{3}}\sum_{a, B} \sum_{j} \E \Big[ \ga_{aB} \gb_{Bj} \gb_{Ba}  \ga_{j B}\Big]
		=&\frac{1}{n^{3}}\sum_{a, B} \sum_{j} \E \Big[ \wh{\ga_{aB}} \gb_{Bj} \gb_{Ba}  \ga_{j B}\Big]
		+\OO_\prec(n^{-1-\epsilon})\nonumber\\
		=&\frac{\m^{z_1} \overline{\m^{z_2}}}{n^{3}}\sum_{B} \sum_{j,j'} \E \Big[ \ga_{j'B} \gb_{Bj} \gb_{Bj'} 
		 \ga_{j B}\Big]+\OO_\prec(n^{-1-\epsilon}),
	\end{align}
	where we have replaced one pair of the index $a$ with a fresh index $j'$ as in (\ref{self_g1g2_1}). 
	Moving the leading term to the left side with a combined factor $1-\m^{z_1} \overline{\m^{z_2}}$ 
	and using (\ref{stable_0}), we conclude from (\ref{self_g1g2}) that
	\begin{align}\label{self_g1g2_result}
		\Big| \frac{1}{n^{3}}\sum_{a, B} \sum_{j} \E \Big[ \ga_{aB} \gb_{Bj} \gb_{Ba} & 
		\ga_{j B}\Big]\Big| =\OO_\prec( n^{-1-\epsilon+\gamma}), \qquad |z_1-z_2|\geq n^{-\gamma}.
	\end{align}
	Combining this with (\ref{T_term}) and (\ref{T}), we hence finished the proof of (\ref{estimate_2}).
\end{proof}

\section{Proof of Proposition \ref{prop1} and \ref{lemma2}}\label{sec:proof_tail}

In this section, we follow the proof strategy used in \cite{EX22} for Wigner matrix to prove Proposition~\ref{prop1}. 

\begin{proof}[Proof of Proposition \ref{prop1}]
Recall that the eigenvalues of $H^{z}$ are denoted by $\{\lambda^z_{\pm j}\}_{j=1}^{n}$ in a non-decreasing order, with $\lambda^z_{j}=-\lambda^z_{-j}$. Then the tail bound of $\lambda^z_1 \in \R^+$ can be written as
\begin{equation}\label{right0}
	\P \big( \lambda^z_1 \leq E \big)=\P\big( \Tr \chi_E(H^z) >0 \big)=\E\big[F\big(\Tr \chi_{E}(H^z) \big) \big], \qquad \chi_{E}:=\one_{[-E,E]},
\end{equation}
where $F\,:\,\R_+\longrightarrow \R_+$ is a smooth and non-decreasing cut-off function such that 
\begin{equation}\label{F_function}
	F(x)=0, \quad \mbox{if} \quad 0 \leq x \leq 1/9; \qquad F(x)=1, \quad \mbox{if} \quad x \geq 2/9,
\end{equation}
For any $\eta>0$, we define $\theta_{\eta}(x):=\frac{\eta}{\pi(x^2+\eta^2)}=\frac{1}{\pi} \Im \frac{1}{x-\ii \eta}$ and then
\begin{equation}\label{approx}
 \Tr \chi_{E} \star \theta_{\eta}(H^z)=\frac{1}{\pi} \int_{-E}^{E} \Im \Tr G^z(y+\ii \eta) \dd y\,.
\end{equation}
Following \cite[Lemma 6.1-6.2]{EYY12} or \cite[Lemma 2.4-2.5]{SX22}, and using the rigidity bound \eqref{rigidity},
one can show that $\Tr \chi_E(H^z)$ 
can be bounded by $\Tr \chi_{E\pm l}\star \theta_\eta(H^z)$  
from below and above whenever  $\eta \ll l\ll E \ll n^{-3/4}$.
The following lemma gives the precise formulation and we omit its standard proof. 
\begin{lemma}
	Fix small $\epsilon_1, \epsilon_2>0$ and large $D>0$.
	 For any $n^{-1+\epsilon_1}\leq E\leq n^{-3/4-\epsilon_2}$, setting $\eta=n^{-7\zeta}E$ and $l=n^{-\zeta} E$ with $\zeta=\epsilon_1/100$, then we have
	\begin{equation}\label{approx3}
		\Tr \chi_{E-l} \star \theta_{\eta}(H^z) -n^{-\zeta} \leq \Tr \chi_E(H^z) \leq  \Tr \chi_{E+l} \star \theta_{\eta}(H^z) +n^{-\zeta},
	\end{equation}
	with very high probability. Moreover, we have
	\begin{equation}\label{approx2}
		\E \Big[ F\Big( \Tr \chi_{E-l} \star \theta_{\eta}(H^z) \Big)\Big]-n^{-D} \leq \E\big[F\big(\Tr \chi_{E}(H^z) \big) \big]\leq   \E \Big[F\Big( \Tr \chi_{E+l} \star \theta_{\eta}(H^z) \Big)\Big]+n^{-D}.
	\end{equation}
\end{lemma}
 The main technical  result is the following GFT: 
\begin{proposition}\label{green_comparison}
Fix small $\epsilon_1, \epsilon_2,\tau>0$ with $\epsilon_2>\tau/2$. For any $E$ with $n^{-1+\epsilon_1}\leq E\leq n^{-3/4-\epsilon_2}$ 
and any $\delta=|z|^2-1$ with $n^{-1/2} \ll \delta \leq n^{-1/2+\tau}$, setting $\eta=n^{-7\zeta}E$ and $l=n^{-\zeta} E$ 
with $\zeta=\epsilon_1/100$, then we have
	\begin{align}\label{tail_thm}
		\left|\E\Big[F\big(\Tr \chi_{E\pm l} \star \theta_{\eta}(H^z)  \big)  \Big]-\E^{\mathrm{Gin}}
		\Big[F\big(\Tr \chi_{E \pm l} \star \theta_{\eta}(H^z)  \big)  \Big] \right|
		= O_\prec\Big(\frac{n^{3/2} E^2 e^{-\frac{n\delta^2}{2}}}{(n\eta)^4}+n^{-100}\Big).
	\end{align}
\end{proposition}
We remark that the last error term $n^{-100}$ in (\ref{tail_thm}) is irrelevant. It is  introduced only to 
incorporate all the polynomially small errors with an arbitrary large power that occur along the proof.
Using this comparison result~\eqref{tail_thm} together with (\ref{right0}), (\ref{approx2}), we have
	\begin{align}
		\P \big( \lambda^z_1 \leq E \big)\leq &\E \Big[F\Big( \Tr \chi_{E+l} \star \theta_{\eta}(H^z) \Big)\Big]
		+O(n^{-D})\nonumber\\
		=&\E^{\mathrm{Gin}} \Big[F\Big( \Tr \chi_{E+l} \star \theta_{\eta}(H^z) \Big)\Big]
		+O_\prec\Big(\frac{n^{3/2} E^2 e^{-\frac{n\delta^2}{2}}}{(n\eta)^4}+n^{-100}\Big)+O(n^{-D})\nonumber\\
\leq &\P^{\mathrm{Gin}} \big( \lambda^z_1 \leq E+2l \big)+O_\prec\Big(\frac{n^{3/2} E^2 e^{-\frac{n\delta^2}{2}}}{(n\eta)^4}+n^{-100}\Big) +O(n^{-D})\nonumber\\
\lesssim & n^{3/2} E^2 e^{-\frac{n\delta^2}{2}}+n^{-100},
	\end{align}
where we chose $D>100$ and used the tail bound in (\ref{tail}) for the Ginibre ensemble 
and that $l\ll E$. We hence proved Proposition \ref{prop1}.
\end{proof}

Next we will prove Proposition \ref{green_comparison} using an iterative  GFT together with a Gronwall argument
 \cite{EX22}.

\begin{proof}[Proof of Proposition \ref{green_comparison}] 

		Recall the matrix interpolating flow $H_t^{z}$ in (\ref{flow}) and its Green function $G_t^z$. 
		 Similarly to the explanations above (\ref{localg}) on the imaginary axis,
		the local law in Theorem \ref{local_thmw} holds true for the time-dependent
		 Green function $G^{z}_t$ simultaneously for all $t
		\in \R^+$. In particular, fixing small $\epsilon'_1,\epsilon'_2>0$, for any $||z|-1| \leq \tau$, and 
		any $w=E+\ii \eta$ with $n^{-1+\epsilon'_1} \leq \eta \leq n^{-3/4-\epsilon'_2}$ 
		and $|E| \lesssim n^{-3/4}$, we~have
\begin{equation}\label{G}
	\sup_{t\geq 0}\max_{i,j} \Big\{\big| G^{z}_{ij}(t,w) - M^z_{ij} (w) \big|\Big\}\prec \frac{1}{n\eta} +\sqrt{\frac{\rho}{n\eta}}
	=:\Psi, \qquad \Im \<M^z(w)\> \sim \frac{\eta}{|z|^2-1}.
\end{equation}
We often omit the dependence on the parameters $t$, $z$ and $w$ for brevity.

	For notational simplicity we set, for any fixed $k\in \Z$
	\begin{equation}\label{Xidef}
		\X_k:=\Tr \chi_{E_k} \star \theta_{\eta}(H_t^z)=\int_{-E_k}^{E_k}  \Im \Tr G_t^z(y+\ii \eta) \dd y, 
		\qquad E_k:=E+kl, \quad \eta \ll l \ll E.
	\end{equation}
Then using~\eqref{right0} and (\ref{approx2}), we have
\begin{align}\label{approxxing}
	\P(\lambda_1 \leq E_{k-1}) \leq \E \big[ F\big( \X_k \big)\big] \leq  \P(\lambda_1 \leq E_{k+1}) \leq \E \big[ F\big( \X_{k+2} \big)\big],
\end{align}
  \ie essentially the existence of a small $\lambda^z_1$ is equivalent to $\X\ge 1/9$ from (\ref{F_function}). 
  Along the iterative GFT procedure, 
the quantity $\X$  will keep track of the
relatively small probability event, indicated by the factor $e^{-n\delta^2/2}$, that there is 
a small eigenvalue $\lambda^z_1$ despite that
$z$ is far away from the unit disk. Recalling the tail estimate for the Ginibre ensemble
 in (\ref{tail}) and using \eqref{approxxing}, we have
\begin{align}\label{Gaussian}
	\E^{\mathrm{Gin}} \big[ F\big( \X_k \big)\big] \lesssim n^{3/2}(E_{k+1})^2 e^{-\frac{n\delta^2}{2}} \lesssim n^{3/2}E^2 e^{-\frac{n\delta^2}{2}},
\end{align}
with $E_k=E+kl$ and $l\ll E$. Since the lower and upper bounds in \eqref{approxxing} involve slightly different $E_k$'s,
 we need a whole sequence of $\X_k$ but this is a minor technicality.
 
In order to prove (\ref{tail_thm}), it then suffices to prove that
\begin{align}\label{tail_goal}
	\Big|\E\big[ F(\X_{\pm 1}) \big]-\E^{\mathrm{Gin}}\big[ F(\X_{\pm 1}) \big]  \Big|= O_\prec\Big( \Psi^4 n^{3/2} E^2 e^{-\frac{n\delta^2}{2}}+n^{-100}\Big)
	=:\mathcal{E}_0.
\end{align}
Despite that the target error $\mathcal{E}_0$ contains a term with an exponential factor, $\mathcal{E}_0$ is actually polynomially small in $1/n$ due to the other irrelavent term $n^{-100}$. This fact will allow us to incorporate into $\mathcal{E}_0$ other polynomially small errors that come with an arbitrary large power $n^{-D}$ in the proof.

For any fixed $k\in \Z$, taking the time derivative of $\E[F(\X_k)]$, we have
\begin{align}\label{step00}
	\frac{\dd}{\dd t}\E [F(\X_k)]
	=&-\frac{1}{2} \sum_{a=1}^n\sum_{B=n+1}^{2n}\left( \sum_{p+q+1=3}^{2M_0}    \frac{c^{(p+1,q)}}{p!q!n^{\frac{p+q+1}{2}}}   \E \left[  \frac{\partial^{p+q+1} F(\X_k)}{\partial w_{aB}^{p+1} \partial \overline{w_{aB}}^{q} }\right] \right)\nonumber\\
	&-\frac{1}{2} \sum_{a=1}^n\sum_{B=n+1}^{2n} \left(\sum_{p+q+1=3}^{2M_0}   \frac{c^{(q,p+1)}}{p!q!n^{\frac{p+q+1}{2}}}   \E \left[  \frac{\partial^{p+q+1} F(\X_k)}{\partial \overline{w_{aB}}^{p+1} \partial w_{aB}^{q} }\right] \right)+O_\prec(n^{-M_0+2})\nonumber\\
	=:&\sum_{p+q+1=3}^{2M_0} \Big(I_{p+1,q}+I'_{q,p+1}\Big)+O_\prec(n^{-M_0+2}),
\end{align}
where $c^{(p,q)}$ are the $(p,q)$-cumulants of the normalized complex-valued i.i.d. entries 
$\sqrt{n} w_{aB}$ in (\ref{Wmatrix}) as before.
 Here we truncated the cumulant expansions at the $(2M_0)$-th order and the last error term is 
 obtained using the local law in (\ref{G}) 
 and the moment condition in (\ref{eq:hmb}). We choose a sufficiently large $M_0>0$ such 
 that $n^{-M_0+2}\leq \mathcal{E}_0$ 
 since  $\mathcal{E}_0$, defined in (\ref{tail_goal}), is polynomially small.

 Since derivatives of $F$ show up in \eqref{step00}, we need to find the analogue of 
 \eqref{approxxing}-\eqref{Gaussian} for $F^{(j)} (\X_k)$, in fact 
we will need only the upper bound, this will be given in~\eqref{F_k} below.
To alleviate the notations, we introduce the following abbreviations. 
Let $P\,:\,\R^+ \times \C \setminus \R \longrightarrow\C$ be an arbitrary function, then we introduce  
\begin{align}\label{dim}
	\widetilde{{\Im}}P(t,w):=\frac{1}{2 \ii} (P(t,w)-P(t,\bar w)),\qquad
	\Delta_k \widetilde \Im P :=\widetilde{{\Im}}P(t,E_k+\ii \eta)-\widetilde{{\Im}}P(t,-E_k+\ii \eta)\,,
\end{align}
with $E_k=E+kl$. From the differentiation rules in (\ref{rule_12}), for any fixed $j \in \N$, we have
\begin{align}\label{rule_3}
	\frac{ \partial F^{(j)}(\X_k) }{ \partial h_{aB} }
	=- F^{(j+1)}(\X_k)\sum_{v=1}^{2n} \widetilde \Im  \Big(\int_{-E_k}^{E_k} G_{va} G_{Bv}(y+\ii \eta) \dd y \Big)
	=- F^{(j+1)}(\X_k) \Dim_k G_{ab}\,,
\end{align}
where we used that $G^2(w)=\frac{\dd }{\dd w} G(w)$. From the definition of the function $F$ in (\ref{F_function})
all its derivatives are bounded, i.e. there exists some constant $C_j>0$ such that $\sup_{x\in \R}|F^{(j)}(x)| \leq C_j$. 
Thus for any $j\ge 1$
\begin{align}
	\E\big[|F^{(j)}(\X_k)|\big] 	\leq C_j \P \Big( \X_k \in [1/9,2/9]\Big).
\end{align}
Recall that the inequalities in (\ref{approx3}) imply that  
\begin{align}
	\# \{j: |\lambda^z_j| < E_{k-1}\}-n^{-\zeta} \leq \X_k=\Tr \chi_{E_k} \star \theta_{\eta}(H^z)
	 \leq \# \{j: |\lambda^z_j| \leq E_{k+1}\}+n^{-\zeta},
\end{align}
with a very high probability. If $\X_k \in [1/9,2/9]$, 
then $\# \{j: |\lambda^z_j| < E_{k-1}\}= 0$ and $\# \{j: |\lambda^z_j| < E_{k}\} \geq 1 $. 
Thus we have, for any fixed $j \in \N$,
\begin{align}\label{F_k}
	\E\big[|F^{(j)}(\X_k)|\big] \leq& C_j \P\big(\lambda^z_1  \in[ E_{k-1},E_{k+1}]\big)\leq C_j\E\big[F(\X_{k+2})\big],
\end{align}
where we also used (\ref{approxxing}) in the last step.

Now we return to (\ref{step00}). By direct computations using the former rule in (\ref{rule_12}) and (\ref{rule_3}), 
each $I_{p+1,q}$ or $I'_{q,p+1}$ on the right side of (\ref{step00}) consists of finitely many products of $p+q+1$ 
Green function entries (bounded by $\Psi^{p+q+1}$ from (\ref{G})) with a derivative of $F$ in front, 
\ie $F^{(j)}(\X_k)$ with $1\leq j\leq p+q+1$. Using the local law in (\ref{G}) and (\ref{F_k}), we
 have for any $4 \leq p+q+1 \leq 2M_0$,
\begin{align}\label{I_k+1}
	|I_{p+1,q}|+|I'_{q,p+1}| \prec& n^{-\frac{p+q-3}{2}}\Psi^{p+q+1}\sum_{j=1}^{p+q+1}
	 \E |F^{(j)}(\X_k)| \lesssim \Psi^4\E\big[F(\X_{k+2})\big].
\end{align}
This takes care of all the fourth and higher order terms in (\ref{step00}). 

Next we focus on estimating the most involved third order terms $I_{p+1,q},~I'_{q,p+1}$ in (\ref{step00}) with $p+q+1=3$. From (\ref{rule_12}) and (\ref{rule_3}), they are linear combinations of terms consisting of three G-factors, \eg
\begin{align}\label{third_F}
&\frac{\sqrt{n}}{n^{2}}  \sum_{a, B}  \E \Big[ F'(\X_k) \Dim_k \big( (G_{aB})^3\big)\Big]
\qquad \frac{\sqrt{n}}{n^{2}} \sum_{a, B} \E \Big[ F''(\X_k) \Dim_k G_{aB} \Dim_k \big( (G_{aB})^2\big)\Big],
\nonumber\\
	&\frac{\sqrt{n}}{n^{2}} \sum_{a, B}  \E \Big[ F'''(\X_k) \Dim_k G_{aB} \Dim_k G_{aB} \Dim_k G_{aB}\Big].
\end{align}
Using (\ref{G}) and (\ref{F_k}), these third order terms with $a=\ud{B}$ can be bounded by 
\begin{align}\label{I_3_ab}
	\sum_{p+q+1=3}\big(I_{p+1,q}+ I'_{q,p+1}\big)\Big|_{a=\ud{B}} \prec n^{-1/2}\sum_{j=1}^{3}\E|F^{(j)}(\X_k)|
	\lesssim n^{-1/2}\E\big[F(\X_{k+2})\big],
\end{align}
where we gained from the index coincidence. 

For the remaining third order terms with $a \neq \ud{B}$~(omitting the $\sqrt{n}$-prefactor), we can adapt 
Definition~\ref{def:unmatch_form} for unmatched terms in (\ref{form}) to a slightly different form as in (\ref{third_F}), 
\ie averaged products of Green function entries on which $\Dim$ acts and with a derivative
 of $F$ in front. Using additionally the
 differentiation rule (\ref{rule_3}) for the function $F$, we can derive similar expansions for these 
 modified unmatched terms as in Lemma~\ref{lemma:expand}, and hence iterate these expansions 
 similarly to Proposition~\ref{prop:unmatched}. The only difference is that we need to keep the derivatives 
 of $F$ in the expansions all of which carry a small probability event indicated by the factor 
 $e^{-n\delta^2/2}$~(see~(\ref{Gaussian}) and (\ref{F_k})). Note that in each expansion as 
 in Lemma~\ref{lemma:expand} the orders of the derivatives of $F$ may be increased by three using~(\ref{rule_3}),
  stemming from the fourth order cumulants as in~(\ref{case1}). Performing expansions 
   iteratively for sufficiently many, say $D_0$ times, the degrees of the unmatched terms 
   generated in iterative expansions have
    been increased to at least $D_0$ and the derivatives of $F$ in these terms may be
     raised to the $(3D_0)$-th order at most. 
    Therefore, these third order terms with $a\neq \ud{B}$ can be bounded by
\begin{align}\label{I_3}
	\sum_{p+q+1=3}\big(I_{p+1,q}+ I'_{q,p+1}\big)\Big|_{a\neq \ud{B}} \prec& \sqrt{n}
	 \Big( n^{-3/2} \sum_{j=1}^{3D_0}\E|F^{(j)}(\X_k)|+O_\prec(\Psi^{D_0})\Big)\nonumber\\
	\lesssim& n^{-1}\E\big[F(\X_{k+2})\big]+\mathcal{E}_0,
\end{align}
where we used (\ref{F_k}) and chose $D_0$ large enough
 such that $\sqrt{n}\Psi^{D_0} \ll \mathcal{E}_0$ with $\mathcal{E}_0$ polynomially small given in (\ref{tail_goal}). 

Therefore, combining (\ref{I_k+1}), (\ref{I_3_ab}), and (\ref{I_3}) with (\ref{step00}), for any fixed $k\in \Z$, we have
\begin{align}\label{stepk}
	\frac{\dd}{\dd t}\E \big[ F\big(\X_{k}\big) \big] \prec  \Psi^4\E \big[ F\big(\X_{k+2}\big) \big]+O(\mathcal{E}_0).
\end{align}
Since the function $F$ in (\ref{F_function}) is uniformly bounded, we integrate
 (\ref{stepk}) in time (up to $t_0=800 \log n$ as before)  in combination with (\ref{approxxxx})
 and obtain
\begin{align}\label{compare_1}
	\Big|\E \big[ F\big(\X_{k}\big) \big]-\E^{\mathrm{Gin}} \big[ F\big(\X_{k}\big) \big] \Big|=O_\prec(\Psi^4+\mathcal{E}_0).
\end{align}
Combining with the estimate in~\eqref{Gaussian} for the Ginibre ensemble, we have proved 
\begin{align}\label{step2}
\E \big[ F\big(\X_{k}\big) \big]=O_\prec\big(\Psi^4+\mathcal{E}_0+n^{3/2}E^2 e^{-\frac{n\delta^2}{2}}\big).
\end{align}
Plugging (\ref{step2}) into the right side of (\ref{stepk}) with $k$ replaced by $k-2$, we have
$$\Big|\frac{\dd}{\dd t}\E \big[ F\big(\X_{k-2}\big) \big] \Big|=
O_\prec\big(\Psi^8+\Psi^4\mathcal{E}_0+\Psi^4n^{3/2}E^2 e^{-\frac{n\delta^2}{2}}+\mathcal{E}_0\big),$$
which further implies after integration that
$$
\Big|\E \big[ F\big(\X_{k-2}\big) \big]-\E^{\mathrm{Gin}}\big[ F\big(\X_{k-2}\big) \big]\Big|
=O_\prec\Big(\Psi^8+\mathcal{E}_0\Big).
$$
Note that we have slightly improved the first error term in (\ref{compare_1}) by an 
additional factor $\Psi^4$ with $k$ reduced to $k-2$.

Now we can iterate this procedure to gain sufficient many $\Psi^4$ improvements to prove (\ref{tail_goal}). 
We choose $S_0$ sufficiently large so that $(\Psi^4)^{S_0} \leq \mathcal{E}_0$ with $\mathcal{E}_0$ 
polynomially small given in (\ref{tail_goal}). We start with (\ref{stepk}) for $k=2S_0\pm 1$ and 
performing the above arguments iteratively for $S_0$ times until $k$ has been reduced to $\pm 1$
 with an error term $(\Psi^4)^{S_0}+\mathcal{E}_0=O(\mathcal{E}_0)$. Note that the number of iterations 
 $S_0$ is indepdendent of $n$. Set $e^{-\frac{n\delta^2}{2}} \lesssim n^{-1/2}$ as needed for 
 (\ref{tail_bound_iid}). If we set $E=n^{-3/4-\epsilon_2}$, then we need to iterate only once. 
 However if we set $E=n^{-1+\epsilon_1}$, then we have to run the iterations for $O(\epsilon_1^{-1})$ times. 
 In this way we have obtained the desired upper bound in (\ref{tail_goal}) and this completes
  the proof of Proposition~\ref{green_comparison}.
\end{proof}

We next use Proposition \ref{prop1} to prove the following proposition.
\begin{proof}[Proof of Proposition \ref{lemma2}]
	From (\ref{tau_result}), for any $n^{-1+\epsilon_1} \leq \eta\leq n^{-3/4-\epsilon_2}$, we have
			\begin{align}\label{use}
				\P\big(|\lambda_1^z| \leq \eta \big) \lesssim n^{-1/2} \eta^2 e^{-n \delta^2/2}+n^{-100}.
			\end{align}
	Using spectral decomposition of $H^{z}$ and the spectrum symmetry, we split the
	 eigenvalues $\{\lambda_i^z\}_{j=1}^n$ into two parts: $\lambda_i^z   \leq  \wt \eta$ 
	 and $\lambda_i^z   \geq  \wt \eta$ with $\wt \eta=n^{-3/4-\alpha}$ for a small $\alpha>0$
	  to be fixed later which will depend on $k$.  We then further divide $\lambda^z_i\in [0,\wt \eta]$ into
	   triadic partitions as in~(\ref{dy}), \ie
		\begin{align}\label{step_lemma_2_1}
		\E\Big[ \Big(\Im \<G^z(\ii \eta)\> \Big)^{k}\Big] =&  \E^{\mathrm{Gin}} 
		\Big[ \Big(\frac{1}{n} \sum_{l=0}^{O(\log n)}\sum_{3^{l-1} \eta \leq \lambda_i^z  < 3^{l} \eta}
		\frac{\eta}{(\lambda^z_{i})^2+\eta^2} +\frac{1}{n}\sum_{\lambda_i^z   \geq  \wt \eta }  
		\frac{\eta}{(\lambda^z_{i})^2+\eta^2} \Big)^k \Big]\nonumber\\
		\lesssim &  ( \log n)^{k} \E\Big[ \frac{1}{n^k}\sum_{l=0}^{O(\log n)} \sum_{3^{l-1} \eta
		 \leq \lambda_i^z  < 3^{l} \eta} \Big( \frac{\eta}{(\lambda^z_{i})^2+\eta^2}\Big)^{k}
		 +\Big(\frac{1}{n}\sum_{\lambda_i^z   \geq  \wt \eta} \frac{\eta}{(\lambda^z_{i})^2+\eta^2}\Big)^k \Big],
	\end{align}
recalling the convention that for $l=0$ we set $3^{l-1}\eta :=0$  in the lower limit of the summations. 
For the first part in (\ref{step_lemma_2_1}), we have
\begin{align}\label{step_lemma_2_2}
	 \frac{1}{n^k} \E\Big[ \sum_{l=0}^{O(\log n)} \sum_{3^{l-1} \eta
	  \leq \lambda_i^z  < 3^{l} \eta} \Big( \frac{\eta}{(\lambda^z_{i})^2+\eta^2}\Big)^{k}\Big] 
	  \lesssim &\frac{n^{k\xi}}{n^k} \sum_{l=0}^{O(\log n)} \Big( \frac{\eta}{(3^{l-1} \eta)^2
	  +\eta^2}\Big)^{k}  \P^{\mathrm{Gin}}\big( \lambda^z_1 \leq 3^{l} \eta \big)\nonumber\\
	 &\lesssim n^{k\xi}\log n\Big( \frac{n^{-1/2}}{(n\eta)^{k-2}}e^{-n \delta^2/2}+n^{-100}\Big),
\end{align}
where we used the rigidity of eigenvalues in (\ref{rigidity0}) for a small $\xi>0$ to be fixed, 
and we also used (\ref{use}) in the last line. For the second part in (\ref{step_lemma_2_1}), we have
\begin{align}\label{step_lemma_2_3}
	\E\Big[ \Big( \frac{1}{n} \sum_{\lambda_i^z   \geq  \wt \eta} \frac{\eta}{(\lambda^z_{i})^2
	+\eta^2}\Big)^k \Big] \lesssim &\E\Big[ \Big( \frac{1}{n} \sum_{\lambda_i^z   \geq  \wt \eta} \frac{\eta}{(\lambda^z_{i})^2
	+\wt \eta^2}\Big)^k \Big]\lesssim \frac{\eta^k}{\wt \eta^k}\E\Big[ \Big(\Im \<G^z(\ii \wt \eta)\> \Big)^{k}\Big] \nonumber\\
	\prec &\frac{\eta^k}{\wt \eta^k}\Big( \frac{1}{(n\wt \eta)^k} +(\sqrt{n}\wt \eta)^k\Big) 
	\lesssim n^{2\alpha k} (\sqrt{n}\eta)^k,
\end{align}
which follows from the local law in (\ref{averagew}) and (\ref{rho}). 
Combining (\ref{step_lemma_2_2}) with (\ref{step_lemma_2_3}), we obtain from~(\ref{step_lemma_2_1}) 
\begin{align}
	\E\Big[ \Big(\Im \<G^z(\ii \eta)\> \Big)^{k}\Big] \prec (n^{\xi}\log n)^{k}
	 \Big(\frac{n^{-1/2}e^{-n \delta^2/2}}{(n\eta)^{k-2}}+n^{-100}\Big)+(n^{2\alpha}\log n)^k (\sqrt{n}\eta)^k.
\end{align}
For any fixed $k\in \N$ and any small $\tau>0$, choosing $\alpha \leq \frac{\tau}{2k}$ and $\xi\leq \frac{\tau}{k}$, we obtain
$$\E\Big[ \Big(\Im \<G^z(\ii \eta)\> \Big)^{k}\Big] 
\lesssim n^{\tau} \Big(\frac{n^{-1/2}e^{-n \delta^2/2}}{(n\eta)^{k-2}}+n^{-100}+(\sqrt{n}\eta)^k\Big).$$
We hence finished the proof of Proposition~\ref{lemma2}.
\end{proof}

We finally present the proof of Proposition \ref{prop2_old}.

\begin{proof}[Proof of Proposition \ref{prop2_old}]
	The expectation estimate in (\ref{difference}) was already stated in 
	\cite[Proposition 3.7]{maxRe} and proved via an iterative  GFT argument in
	 \cite[Section 4]{maxRe}, where in each step 
	of the iteration an additional factor $1/(n\eta)\le n^{-\epsilon}$ was gained. 
	The proof of (\ref{difference_var}) is quite similar to this argument in \cite[Section 4]{maxRe}, 
	so we only sketch the proof for brevity.  To simply the notations, we may assume $z=z'$, and the same proof also applies to general $z \neq z'$.

		In the following we use the same index conventions 
	as in Notation \ref{not:index} and also some statements in Section~\ref{sec:GFT}. 
	We also comment that comparing the error terms in (\ref{difference_var}) to those 
	in~(\ref{difference}), we almost gain 
	an additional factor $1/(n\eta)$; this is mainly because the a priori bound of the 
	variance is $1/(n\eta)$ better than the corresponding expectation.

	 Recall the matrix interpolating flow $H_t^{z}$ in (\ref{flow}) and 
	 the local law for the Green function of $H_t^{z}$, denoted by $G^z_t$ in (\ref{localg}). 
	Applying Ito's formula to $\<G_t^z\>$ in (\ref{normal_trace}) and performing 
	the cumulant expansion formula, we obtain \cf \cite[Eq (4.11)]{maxRe}
	\begin{align}\label{var_deri}
		\frac{\dd}{\dd t} \V[\<G^{z}_t\>]=&-\frac{1}{2n^2}\sum_{u,v,a=1}^{n} \sum_{B=n+1}^{2n}
		\left( \sum_{p+q+1= 3}^{K_0}\frac{c^{(p+1,q)}}{p!q!n^{\frac{p+q+1}{2}}} 
		\E \left[  \frac{\partial^{p+q+1} \big(G_{uu}-\E G_{uu}\big)\big(G_{vv}-\E G_{vv}\big)}{\partial w_{aB}^{p+1} 
		\partial \overline{w_{aB}}^{q} } \right]\right)\nonumber\\
		&- \frac{1}{2n^2} \sum_{u,v,a=1}^{n} \sum_{B=n+1}^{2n} 
		\left(\sum_{p+q+1= 3}^{K_0}\frac{c^{(q,p+1)}}{p!q!n^{\frac{p+q+1}{2}}} 
		\E \left[  \frac{\partial^{p+q+1} \big(G_{uu}-\E G_{uu}\big)
		\big(G_{vv}-\E G_{vv}\big) }{\partial \overline{w_{aB}}^{p+1} \partial w_{aB}^{q}}\right]\right)\nonumber\\
		&\qquad  +\OO_{\prec}(n^{-\frac{K_0}{2}+2})\nonumber\\
		=:&\sum_{p+q+1= 3}^{K_0} \big( L_{p+1,q}+L'_{q,p+1}\big)+\OO_{\prec}(n^{-\frac{K_0}{2}+2}),
	\end{align}
	where $c^{(p,q)}$ are the $(p,q)$-cumulants of the normalized complex-valued i.i.d. entries 
	of $\sqrt{n}w_{aB}$ given in~(\ref{Wmatrix}) and we choose $K_0=100$. For brevity we
	 dropped the parameters $t$ and $z$ from $G$ and also from $L$ and $L'$.  
	Using the differentiation rules in (\ref{rule_12}), each term $L_{p+1,q}$ or $L'_{q,p+1}$ in~(\ref{var_deri})
	 consists of finitely many products of $p+q+3$ Green function entries. 
	 Thus using the local law in~(\ref{localg})-(\ref{localm}), they can be bounded by, for any $p+q+1\geq 4$,
	\begin{align}\label{var_four}
		|L_{p+1,q}|+|L'_{q,p+1}| \prec n^{-\frac{p+q-3}{2}} \big(\Psi^{p+q+3}+n^{-1}\big),\qquad \Psi=(n\eta)^{-1},
	\end{align}
	where the term $n^{-1}$ is from the cases with an index coincidence, \eg $a=\ud{B}$. This easy estimate 
	is sufficient for all terms of order four or higher.  
	
	It then suffices to estimate the third order terms in (\ref{var_deri}) with $p+q+1=3$. 
	We split the summations over $u,v,a,B$ into two parts: the restricted summations 
	with $a \neq u \neq v\neq \ud{B}$ and the remaining summations with at least one
	index coincidence, \eg $a=\ud{B}$. Then the third order terms in (\ref{var_deri}) with 
	$a \neq u \neq v\neq \ud{B}$ are unmatched terms with a factor $\sqrt{n}$ (see \cite[Definition 4.4]{maxRe} 
	or Definition~\ref{def:unmatch_form}), since the index $a$ or $B$ appears three times as 
	the row/column index of Green function entries. Using \cite[Proposition 4.5]{maxRe} or 
	Proposition \ref{prop:unmatched}, they can be bounded by $O_\prec(n^{-1})$. 
	Moreover, for the remaining summations with an index coincidence, \eg $a=\ud{B}$, 
	by direct computations they contain at least three off-diagonal (or centered diagonal) 
	Green function entries and thus are bounded by $O_\prec(n^{-1/2}\Psi^3)$. Therefore we have
	\begin{align}\label{var_three}
		\sum_{p+q+1=3} |L_{p+1,q}|+|L'_{q,p+1}|=O_\prec(n^{-1}+n^{-1/2}\Psi^3).
	\end{align}
	Using (\ref{var_four}) and (\ref{var_three}), we obtain from (\ref{var_deri}) that
	\begin{align}\label{var_time}
		\Big|\frac{\dd}{\dd t} \V[\<G^{z}_t\>]\Big|=O_\prec(n^{-1/2}\Psi^3+\Psi^6+n^{-1}).
	\end{align}
	Integrating (\ref{var_time}) over $t\in [0,t_0]$ with $t_0=800\log n$ as before and using the estimate in (\ref{approxxxx}), we hence finished the proof of (\ref{difference_var}).

\end{proof}

\section{Proof of Lemma~\ref{lemma_Gin}}
\label{app:Gincal}

The proof of this lemma is similar to the proof of~\cite[Lemma 3.1]{maxRe}, which itself 
relies on~\cite[Section 2]{maxRe_Gin}. Before presenting the proof of Lemma~\ref{lemma_Gin}, 
we introduce some notation from \cite{maxRe_Gin} (see also~\cite{RS14}). We prove 
Lemma~\ref{lemma_Gin} only in the complex case to keep the presentation short; the proof 
in the real case is analogous proceeding similarly to~\cite[Section 3]{maxRe_Gin} instead
 of similarly to~\cite[Section 2]{maxRe_Gin}. Define the kernel
\begin{equation}
\label{eq:defkerzw}
\widetilde{K}_n(z,w):=\frac{n}{\pi} e^{-(|z|^2+|w|^2-2z\overline{w})/2}\frac{\Gamma(n,nz\overline{w})}{\Gamma(n)},
\end{equation}
with $\Gamma(\cdot,\cdot)$ the incomplete gamma function defined as
\[
\Gamma(s,z):=\int_z^\infty t^{s-1}e^{-t}\,\dd t,
\]
where $s\in\mathbb{N}$ and the integration contour is from $z\in\C$ to real infinity. Then we have the following expression for expectation and variance of linear statistics
\begin{equation}
\begin{split}
\label{eq:linstatker}
\E \sum_i f(\sigma_i)&=\int_\C f(z)\widetilde{K}_n(z,z)\, \dd^2z \\
\V \sum_i f(\sigma_i)&=\int_\C f(z)^2\widetilde{K}_n(z,z)\, \dd^2z-\int_\C\int_\C f(z)f(w)\big|\widetilde{K}_n(z,w)\big|^2\, \dd^2z\dd^2w.
\end{split}
\end{equation}

We will now prove Lemma~\ref{lemma_Gin} as a consequence of the following technical result.
 
\begin{lemma}
\label{lem:kernbound}
Let $\gamma=\gamma_n:=\log n-2\log\log n-\log 2\pi$. Rescale the kernel variable as
\begin{equation}
\label{eq:resvar}
z=e^{\ii \theta}\left(1+\sqrt{\frac{\gamma}{4n}}+\frac{x}{\sqrt{4\gamma n}}\right),
\end{equation}
with $\theta\in [0,2\pi)$ and $x\in\R$. Then, in the regime $|x|\le \sqrt{\log n}/2$ we have the asymptotics
\begin{equation}
\label{eq:kernasympneed}
\frac{\widetilde{K}_n(z,z)}{\sqrt{n\gamma}}=\frac{e^{-x}}{\pi}\left(1+O\left(\frac{\log \log n+x^2}{\log n}\right)\right).
\end{equation}
Furthermore, for $x\ge 0$ we have the uniform bound
\begin{equation}
\label{eq:zzeno}
\frac{\widetilde{K}_n(z,z)}{\sqrt{\gamma n}}\lesssim |z|^2 e^{-x/3}.
\end{equation}
\end{lemma}

\begin{proof}[Proof of Lemma~\ref{lemma_Gin}]

By \eqref{eq:linstatker}, given the asymptotic in \eqref{eq:kernasympneed} and the uniform bound in \eqref{eq:zzeno}, the proof of Lemma~\ref{lemma_Gin} is completely analogous to the proof of~\cite[Lemma 3.1]{maxRe}.

\end{proof}

We conclude this section with the proof of Lemma~\ref{lem:kernbound}.

\begin{proof}[Proof of Lemma~\ref{lem:kernbound}]

The proof of this lemma is completely analogous (actually easier) to~\cite[Lemma 6]{maxRe_Gin}. 
In particular, the current proof needs only $\widetilde{K}_n(z,z)$, instead of~\cite[Lemma 6]{maxRe_Gin} where  $|\widetilde{K}_n(w,z)|^2$ has been considered. We present the detailed proof here for completeness
following the steps from~\cite{maxRe_Gin}.

We start with the bound in \eqref{eq:zzeno}. Recall from \eqref{eq:defkerzw} that
\begin{equation}
\label{eq:defkerzz}
\widetilde{K}_n(z,z)=\frac{n}{\pi}\cdot\frac{\Gamma(n,n|z|^2)}{\Gamma(n)},
\end{equation}
and recall the asymptotic \cite[Lemma 3.2]{RS14}
\begin{equation}
\label{eq:gammaexp}
\frac{\Gamma(n,nt)}{\Gamma(n)}=\frac{t\mu(t)\mathrm{erfc}(\sqrt{n}\mu(t))}{\sqrt{2}(t-1)}\left(1+O\left(n^{-1/2}\right)\right), \qquad \mu(t):=\sqrt{t-\log t-1},
\end{equation}
uniformly in $t>1$. Next, plugging \eqref{eq:gammaexp} into \eqref{eq:defkerzz}, and using
\[
|z|^2=1+\frac{\gamma+x}{\sqrt{\gamma n}}+\frac{(x+\gamma)^2}{4\gamma n}\ge 1+\frac{\gamma+x}{\sqrt{\gamma n}}
\]
together with $\mathrm{erfc}(x)\lesssim e^{-x^2}/x$, we obtain that, for $x\geq 0$ 
\begin{equation}
\label{eq:finbzz}
\frac{\widetilde{K}_n(z,z)}{\sqrt{\gamma n}}\lesssim \frac{|z|^2e^{-n\mu(|z|^2)^2}}{\sqrt{\gamma}(|z|^2-1)}\lesssim \frac{\sqrt{n}}{\gamma}|z|^2 e^{-\gamma/2}e^{-x/3}.
\end{equation}
We remark that in the last inequality we also used that
\begin{equation}
\label{eq:uslbmu}
\mu(t)^2\ge t-\log t-1\ge \delta(1-\delta)(t-1)/2,
\end{equation}
for $t\ge 1+\delta$ and $\delta\in [0,1)$, i.e. for $t=|z|^2$ and $\delta=(\gamma/n)^{1/2}$ we used
\[
\big(\mu(|z|^2)\big)^2\ge \frac{\gamma+x}{2n}\left(1-\sqrt{\frac{\gamma}{n}}\right)\ge \frac{\gamma}{2n}\left(1-\sqrt{\frac{\gamma}{n}}\right)+\frac{x}{3n}.
\]
We then conclude the bound \eqref{eq:zzeno} plugging
\begin {equation}
\label{eq:egamma}
e^{-\gamma/2}=\exp\left(-\frac{1}{2}\log\frac{n}{(\log n)^2 2\pi}\right)=\frac{(2\pi)^{1/2}\gamma}{\sqrt{n}}\left(1+O\left(\frac{\log\log n}{\log n}\right)\right)
\end{equation}
in the rhs. of \eqref{eq:finbzz}.

Next, to compute the asymptotic in \eqref{eq:kernasympneed} we use the Taylor expansions
\begin{equation}
\mu(1+d)=\frac{d}{\sqrt{2}}+O(|d|^2),\qquad \mathrm{erfc}(t)= \frac{e^{-t^2}}{\sqrt{\pi}t}\left(1+O\left(\frac{1}{t^2}\right)\right),
\end{equation}
which, by \eqref{eq:gammaexp}, imply
\begin{equation}\label{expdec}
\begin{split}
\frac{\Gamma (n,n|z|^2)}{\Gamma(n)}&= \frac{|z|^2 e^{-n\mu(|z|^2)^2}}{\sqrt{2\pi} (|z|^2-1)}\left(1+O\left(\frac{1}{n\mu(|z|^2)^2}+\frac{1}{\sqrt{n}}\right)\right) \\
&= \frac{e^{-\frac{n}{2}(|z|^2-1)^2}}{\sqrt{2\pi n} (|z|^2-1)}\left(1+O\left(|z|^2-1+n(|z|^2-1)^3+\frac{1}{n(|z|^2-1)^2}\right) \right).
\end{split}
\end{equation}
Note that to go from the first to the second line we also used \eqref{eq:uslbmu} in
 the estimate of the error term. We thus finally conclude that, for $|x|\le \sqrt{\log n}/2$
\begin{equation}
\begin{split}
\frac{\widetilde{K}_n(z,z)}{\sqrt{n\gamma}}&
=\frac{e^{-x-\gamma/2}}{\sqrt{2\gamma}\pi^{3/2}(|z|^2-1)}\left(1+O\left(\frac{x^2}{\log n}\right)\right) \\
&=\frac{e^{-x}}{\pi}\left(1+O\left(\frac{\log \log n+x^2}{\log n}\right)\right),
\end{split}
\end{equation}
where to go from the first to the second line we used 
$(|z|^2-1)=\sqrt{\gamma/n}(1+O(|x|/\gamma))$ and \eqref{eq:egamma}.
\end{proof}

\section{Proof of Proposition~\ref{lemma_new_2}}
\label{app:decres}

First, we note that  by the rigidity estimate in \eqref{rigidity} we have 
(see~\cite[Eqs. (7.6)--(7.7)]{CES19} for exactly the same computations in the bulk regime) 
\begin{equation}\label{eq:0}
\<G^{z_1}(\ii \eta_2)\>=\frac{1}{2n}\sum_{|i|\le n^{\widehat{\omega}}} \frac{\eta_2}{(\lambda_i^{z_1})^2
+\eta_2^2}+O_\prec\left(\frac{\sqrt{n}\eta_2}{n^{\widehat{\omega}/2}}\right),
\end{equation}
where we chose $\wh \omega$ as in Theorem~\ref{theo:ind}, since we will shortly use
 Theorem~\ref{theo:ind} in the leading term of \eqref{eq:0} to compare it to the same 
 quantity with $\lambda_i^{z_1}$ replaced with $\mu_i^{(1)}$. Next, we compute
\[
\left|\E  \left[\<G^{z_2}(\ii \eta_2)\>O_\prec\left(\frac{\sqrt{n}\eta_2}{n^{\widehat{\omega}/2}}\right)\right]\right|
\le O_\prec\left(\frac{\sqrt{n}\eta_2}{n^{\widehat{\omega}/2}}\right)\E\big|\<G^{z_2}(\ii \eta_2)\>\big|
= O_\prec\left(\frac{\sqrt{n}\eta_2}{n^{\widehat{\omega}/2}}\right)\E\<\Im G^{z_2}(\ii \eta_2)\>
=O_\prec\left(\frac{n\eta_2^2}{n^{\widehat{\omega}/2}}\right),
\]
where in the penultimate equality we used that $\big|\<G^{z_2}(\ii \eta_2)\>\big|=\<\Im G^{z_2}(\ii \eta_2)\>$, 
and in the last equality we used Proposition~\ref{lemma2} for $k=1$. We thus obtain
\begin{equation}
\label{eq:1}
\E\Big[\<G^{z_1}(\ii \eta_2)\> \<G^{z_2}(\ii \eta_2)\>\Big]
=\E\left(\frac{1}{2n}\sum_{|i|\le n^{\widehat{\omega}}} 
\frac{\eta_2}{(\lambda_i^{z_1})^2+\eta_2^2}\right)
\left(\frac{1}{2n}\sum_{|i|\le n^{\widehat{\omega}}} \frac{\eta_2}{(\lambda_i^{z_2})^2+\eta_2^2}\right)
+O_\prec\left(\frac{n\eta_2^2}{n^{\widehat{\omega}/2}}\right),
\end{equation}
which shows that it is sufficient to prove \eqref{eq:almostlast} when all $G$'s are replaced by its 
leading approximation, the first term in the right hand side of~\eqref{eq:0}.

By a simple GFT argument (see e.g. the last display in~\cite[Section 5.4]{EKS20}) it follows that
\[
\E\Big[\<G^{z_1}(\ii \eta_2)\> \<G^{z_2}(\ii \eta_2)\>\Big] =\E\Big[\<G^{z_1}_{ct_1}(\ii \eta_2)\> 
\<G^{z_2}_{ct_1}(\ii \eta_2)\>\Big] + O_\prec\left(t_1n^{-1/4+10\delta}\right),
\]
for $t_1:=n^{-1/2+\omega_1}$ as in Theorem~\ref{theo:ind} and 
$c=c(t_1)=1+(1-e^{-t_1})t_1^{-1}=1+O(t_1)$, and $G_{ct_1}^z$ being the
 resolvent of the Hermitization of $X_{ct_1}-z$, with $X_t$ defined in~(\ref{eq:diffmat}).

Then, together with \eqref{eq:1}, this implies
\begin{equation}
\label{eq:2}
\begin{split}
&\E\left[\left(\frac{1}{2n}\sum_{|i|\le n^{\widehat{\omega}}} \frac{\eta_2}{(\lambda_i^{z_1})^2+\eta_2^2}\right)
\left(\frac{1}{2n}\sum_{|i|\le n^{\widehat{\omega}}} \frac{\eta_2}{(\lambda_i^{z_2})^2+\eta_2^2}\right)\right] \\
&\qquad=\E\left[\left(\frac{1}{2n}\sum_{|i|\le n^{\widehat{\omega}}}
 \frac{\eta_2}{(\lambda_i^{z_1}(ct_1))^2+\eta_2^2}\right)
 \left(\frac{1}{2n}\sum_{|i|\le n^{\widehat{\omega}}} \frac{\eta_2}{(\lambda_i^{z_2}(ct_1))^2
 +\eta_2^2}\right)\right]+O_\prec\left(t_1 n^{-1/4+10\delta}+\frac{n\eta_2^2}{n^{\widehat{\omega}/2}}\right).
\end{split}
\end{equation}

We now want to replace the $\lambda_i^{z_l}(ct_1)$ with $\mu_i^{(l)}(ct_1)$ using Theorem~\ref{theo:ind}.  
Under the assumptions $-C_*n^{-1/2+\tau}\le |z_l|^2-1\le C^* n^{-1/2+\tau}$ and $|z_1-z_2|\ge n^{-\gamma}$, 
the bound \eqref{ovass}, which is the main assumption in Theorem~\ref{theo:ind}, holds 
 (see the detailed argument around \eqref{uu}). 
Hence, by Theorem~\ref{theo:ind} the high probability bound \eqref{eq:mainb} holds as well. 
Then, by \eqref{eq:mainb}, we readily obtain (note that \eqref{eq:mainb} also holds if $t_1$ is
 replaced by $ct_1$ since $c\approx 1$)
\begin{equation}
\label{eq:2.1}
\begin{split}
\frac{1}{2n}\sum_{|i|\le n^{\widehat{\omega}}} \frac{\eta_2}{(\lambda_i^{z_l}(ct_1))^2+\eta_2^2}
-\frac{1}{2n}\sum_{|i|\le n^{\widehat{\omega}}} \frac{\eta_2}{(\mu^{(l)}(ct_1))^2+\eta_2^2}&
=\frac{1}{2n}\sum_{|i|\le n^{\widehat{\omega}}}\frac{\eta_2(\lambda_i^{z_l}(ct_1)-\mu_i^{(l)}(ct_1))(\lambda_i^{z_l}(ct_1)+\mu_i^{(l)}(ct_1))}{[(\lambda_i^{z_l}(ct_1))^2+\eta_2^2][(\mu_i^{(l)}(ct_1))^2+\eta_2^2]} 
\\
&=O_\prec\left(\frac{1}{n^{3/4+\omega}\eta_2}\left[\big|\<G_{ct_1}^{z_l}(\ii \eta_2)\>\big|
+\big|\<\widetilde{G}_{ct_1}^{z_l}(\ii \eta_2)\>\big|\right] \right).
\end{split}
\end{equation}
Here $\widetilde{G}_t^{z_l}$ denotes the resolvent of the Hermitization of $X_t^{(l)}-z_l$, with $X_t^{(l)}$ 
 being the solution of \eqref{eq:Gindiff}. Note that to estimate the error term in \eqref{eq:2.1} 
 we also used the simple bound
\[
\frac{\eta_2(\lambda_i^{z_l}(ct_1)+\mu_i^{(l)}(ct_1))}{[(\lambda_i^{z_l}(ct_1))^2+\eta_2^2][(\mu_i^{(l)}(ct_1))^2+\eta_2^2]}\le \frac{1}{(\mu_i^{(l)}(ct_1))^2+\eta_2^2}+\frac{1}{(\lambda_i^{z_l}(ct_1))^2+\eta_2^2}.
\]

We now compute
\begin{equation}
\label{eq:addblong}
\begin{split}
&\E\left[\left(\frac{1}{2n}\sum_{|i|\le n^{\widehat{\omega}}} \frac{\eta_2}{(\lambda_i^{z_1}(ct_1))^2+\eta_2^2}\right)\left(\frac{1}{2n}\sum_{|i|\le n^{\widehat{\omega}}} \frac{\eta_2}{(\lambda_i^{z_2}(ct_1))^2+\eta_2^2}\right)\right]\\
&\qquad= \E\left[\left(\frac{1}{2n}\sum_{|i|\le n^{\widehat{\omega}}} \frac{\eta_2}{(\mu_i^{(1)}(ct_1))^2+\eta_2^2}\right)\left(\frac{1}{2n}\sum_{|i|\le n^{\widehat{\omega}}} \frac{\eta_2}{(\lambda_i^{z_2}(ct_1))^2+\eta_2^2}\right)\right]
\\
&\qquad\quad+\E\left[\left(\frac{1}{2n}\sum_{|i|\le n^{\widehat{\omega}}} \frac{\eta_2}{(\lambda_i^z(ct_1))^2+\eta_2^2}-\frac{1}{2n}\sum_{|i|\le n^{\widehat{\omega}}} \frac{\eta_2}{(\mu_i^{(1)}(ct_1))^2+\eta_2^2}\right)\left(\frac{1}{2n}\sum_{|i|\le n^{\widehat{\omega}}} \frac{\eta_2}{(\lambda_i^{z_2}(ct_1))^2+\eta_2^2}\right)\right] \\
&\qquad= \E\left[\left(\frac{1}{2n}\sum_{|i|\le n^{\widehat{\omega}}} \frac{\eta_2}{(\mu_i^{(1)}(ct_1))^2+\eta_2^2}\right)\left(\frac{1}{2n}\sum_{|i|\le n^{\widehat{\omega}}} \frac{\eta_2}{(\lambda_i^{z_2}(ct_1))^2+\eta_2^2}\right)\right] \\
&\qquad\quad+O_\prec\left(\frac{\big|\<G_{ct_1}^{z_2}(\ii \eta_2)\>\big|}{n^{3/4+\omega}\eta_2}\left[\big|\<G_{ct_1}^{z_1}(\ii \eta_2)\>\big|+\big|\<\widetilde{G}_{ct_1}^{z_1}(\ii \eta_2)\>\big|\right] \right),
\end{split}
\end{equation}
where in the last equality we used \eqref{eq:2.1}. An analogous bound holds when
 we replace the sum over the $\lambda_i^{z_2}(ct_1)$ with the one over the $\mu_i^{(2)}(ct_1)$.

Next, using the Cauchy-Schwarz inequality and Proposition~\ref{lemma2} for $k=2$, we get
\begin{equation}
\label{eq:goodbnew}
\E \Big|\<G_{ct_1}^{z_1}(\ii \eta_2)\> \<G_{ct_1}^{z_2}(\ii \eta_2)\>\Big|\le \left(\E \Big|\<G_{ct_1}^{z_1}(\ii \eta_2)\>\Big|^2\right)^{1/2}\left( \E\Big|\<G_{ct_1}^{z_2}(\ii \eta_2)\>\Big|^2\right)^{1/2}\le n\eta_2^2.
\end{equation}
A similar bound holds with $\<G_{ct_1}^{z_1}(\ii \eta_2)\>$ replaced by $\<\widetilde{G}_{ct_1}^{z_1}(\ii \eta_2)\>$. 
Then, using \eqref{eq:goodbnew} to estimate the error term in \eqref{eq:addblong}, we obtain
\begin{equation}
\label{eq:3}
\begin{split}
&\E\left(\frac{1}{2n}\sum_{|i|\le n^{\widehat{\omega}}} \frac{\eta_2}{(\lambda_i^{z_1}(ct_1))^2+\eta_2^2}\right)\left(\frac{1}{2n}\sum_{|i|\le n^{\widehat{\omega}}} \frac{\eta_2}{(\lambda_i^{z_2}(ct_1))^2+\eta_2^2}\right)\\
&\qquad=\E\left(\frac{1}{2n}\sum_{|i|\le n^{\widehat{\omega}}} \frac{\eta_2}{(\mu_i^{(1)}(ct_1))^2+\eta_2^2}\right)\left(\frac{1}{2n}\sum_{|i|\le n^{\widehat{\omega}}} \frac{\eta_2}{(\mu_i^{(2)}(ct_1))^2+\eta_2^2}\right)+O_\prec\left(\frac{n^{1/4}\eta_2}{n^\omega}\right).
\end{split}
\end{equation}

Combining \eqref{eq:1},\eqref{eq:2}, and \eqref{eq:3}, we conclude 
\begin{equation}
\label{eq:finsplit}
\begin{split}
\E\Big[\<G^{z_1}(\ii \eta_2)\> \<G^{z_2}(\ii \eta_2)\>\Big]&=\E\left[\left(\frac{1}{2n}\sum_{|i|\le n^{\widehat{\omega}}} \frac{\eta_2}{(\mu_i^{(1)}(ct_1))^2+\eta_2^2}\right)\left(\frac{1}{2n}\sum_{|i|\le n^{\widehat{\omega}}} \frac{\eta_2}{(\mu_i^{(2)}(ct_1))^2+\eta_2^2}\right)\right] \\
&\quad+O_\prec\left(\frac{n^{1/4}\eta_2}{n^\omega}+\frac{n\eta_2^2}{n^{\widehat{\omega}/2}}+\frac{t_1 n^{10\delta}}{n^{1/4}}\right).
\end{split}
\end{equation}
Additionally, since $\mu_i^{(1)}(t)$ and $\mu_i^{(2)}(t)$ are fully independent for any $t\ge 0$ by construction, we also have
\begin{equation}
\label{eq:splitexp}
\begin{split}
&\E\left[\left(\frac{1}{2n}\sum_{|i|\le n^{\widehat{\omega}}} \frac{\eta_2}{(\mu_i^{(1)}(ct_1))^2+\eta_2^2}\right)\left(\frac{1}{2n}\sum_{|i|\le n^{\widehat{\omega}}} \frac{\eta_2}{(\mu_i^{(2)}(ct_1))^2+\eta_2^2}\right)\right] \\
&\qquad\qquad\quad= \E\left[\left(\frac{1}{2n}\sum_{|i|\le n^{\widehat{\omega}}} \frac{\eta_2}{(\mu_i^{(1)}(ct_1))^2+\eta_2^2}\right)\right]\E\left[\left(\frac{1}{2n}\sum_{|i|\le n^{\widehat{\omega}}} \frac{\eta_2}{(\mu_i^{(2)}(ct_1))^2+\eta_2^2}\right)\right].
\end{split}
\end{equation}

Finally, proceeding exactly as in \eqref{eq:1},\eqref{eq:2}, and \eqref{eq:3}, this time for the product of the expectations rather than the expectation of the product, we conclude that
\begin{equation}
\begin{split}
\label{eq:finest}
&\E\left[\left(\frac{1}{2n}\sum_{|i|\le n^{\widehat{\omega}}} \frac{\eta_2}{(\mu_i^{(1)}(ct_1))^2+\eta_2^2}\right)\right]\E\left[\left(\frac{1}{2n}\sum_{|i|\le n^{\widehat{\omega}}} \frac{\eta_2}{(\mu_i^{(2)}(ct_1))^2+\eta_2^2}\right)\right] \\
&\qquad =\E\Big[\<G^{z_1}(\ii \eta_2)\>\Big]\E\Big[\<G^{z_2}(\ii \eta_2)\>\Big]+O_\prec\left(\frac{n^{1/4}\eta_2}{n^\omega}+\frac{n\eta_2^2}{n^{\widehat{\omega}/2}}+\frac{t_1 n^{10\delta}}{n^{1/4}}\right).
\end{split}
\end{equation}
Combining \eqref{eq:finest} with \eqref{eq:finsplit} we conclude the proof of Proposition~\ref{lemma_new_2}. 
\qed

\end{document}